\documentclass[11pt]{article}

\usepackage[margin=0.92in]{geometry}

\PassOptionsToPackage{dvipsnames}{xcolor}
\usepackage{
    amsmath,
    amssymb,
    amsthm,
    mathtools,
    mathrsfs,
    bm,
    enumitem,
    booktabs,
    array,
    longtable,
    tabularx,
    graphicx,
    float,
    xcolor,
    microtype,
    tikz
}

\usetikzlibrary{arrows.meta,calc,positioning}

\usepackage[T1]{fontenc}
\usepackage{lmodern}
\usepackage[normalem]{ulem}

\usepackage{aliascnt}

\newtheorem{theorem}{Theorem}[section]

\newaliascnt{proposition}{theorem}
\newtheorem{proposition}[proposition]{Proposition}
\aliascntresetthe{proposition}

\newaliascnt{lemma}{theorem}
\newtheorem{lemma}[lemma]{Lemma}
\aliascntresetthe{lemma}

\newaliascnt{corollary}{theorem}
\newtheorem{corollary}[corollary]{Corollary}
\aliascntresetthe{corollary}

\newaliascnt{remark}{theorem}
\newtheorem{remark}[remark]{Remark}
\aliascntresetthe{remark}

\newaliascnt{definition}{theorem}

\aliascntresetthe{definition}

\newaliascnt{assumption}{theorem}

\aliascntresetthe{assumption}

\usepackage{hyperref}

\hypersetup{
    colorlinks=true,
    linkcolor=violet!85!black,
    citecolor=YellowOrange!85!black,
    urlcolor=Aquamarine!85!black,
    pdftitle={A Fixed Potential with Infinitely Many Vanishing Magnetic Tunnels},
    pdfauthor={}
}

\usepackage[nameinlink]{cleveref}

\crefname{theorem}{theorem}{theorems}
\Crefname{theorem}{Theorem}{Theorems}

\crefname{proposition}{proposition}{propositions}
\Crefname{proposition}{Proposition}{Propositions}

\crefname{lemma}{lemma}{lemmas}
\Crefname{lemma}{Lemma}{Lemmas}

\crefname{corollary}{corollary}{corollaries}
\Crefname{corollary}{Corollary}{Corollaries}

\crefname{remark}{remark}{remarks}
\Crefname{remark}{Remark}{Remarks}

\crefname{definition}{definition}{definitions}
\Crefname{definition}{Definition}{Definitions}

\crefname{assumption}{assumption}{assumptions}
\Crefname{assumption}{Assumption}{Assumptions}

\newcommand{\R}{\mathbb R}
\newcommand{\C}{\mathbb C}

\newcommand{\eps}{\varepsilon}
\newcommand{\dd}{\,\mathrm d}
\newcommand{\cA}{\mathcal A}

\newcommand{\cE}{\mathcal E}

\newcommand{\cJ}{\mathcal J}

\newcommand{\cP}{\mathcal P}
\newcommand{\cQ}{\mathcal Q}
\newcommand{\cR}{\mathcal R}

\newcommand{\cO}{\mathcal O}
\newcommand{\abs}[1]{\left|#1\right|}
\newcommand{\norm}[1]{\left\|#1\right\|}
\newcommand{\ip}[2]{\left\langle #1,#2\right\rangle}
\newcommand{\wedgep}{\mathbin{\wedge}}
\newcommand{\supp}{\operatorname{supp}}
\newcommand{\dist}{\operatorname{dist}}
\newcommand{\Rea}{\operatorname{Re}}
\newcommand{\Ima}{\operatorname{Im}}
\newcommand{\Ran}{\operatorname{Ran}}

\title{Double-wells in a strong magnetic field: absence of tunneling for
infinitely many values of the coupling constant}
\author{
    \href{mailto:kevin_buck@brown.edu}{Kevin Buck}\\
    {\footnotesize Department of Mathematics, Brown University}\\[0.5em]
    \href{mailto:cf@math.princeton.edu}{Charles L. Fefferman}\\
    {\footnotesize Department of Mathematics, Princeton University}\\[0.5em]
    \href{mailto:javier_gomez_serrano@brown.edu}{Javier G\'omez-Serrano}\\
    {\footnotesize Department of Mathematics, Brown University}\\[0.5em]
    \href{mailto:amaury.hayat@enpc.fr}{Amaury Hayat}\\
    {\footnotesize CERMICS, \'Ecole des Ponts, Institut Polytechnique de Paris}\\
    {\footnotesize and Korea Institute for Advanced Study}\\[0.5em]
    \href{mailto:jacobshapiro@princeton.edu}{Jacob Shapiro}\\
{\footnotesize Department of Mathematics, Princeton University}\\
{\footnotesize and Faculty of Mathematics, University of Vienna}\\[0.5em]
    \href{mailto:miw2103@columbia.edu}{Michael I. Weinstein}\\
    {\footnotesize Department of Applied Physics and Applied Mathematics,}\\
    {\footnotesize and Department of Mathematics, Columbia University}
}
\date{July 2026}

\begin{document}
\maketitle

\begin{abstract}
We prove the fixed-potential conjecture of Fefferman, Shapiro, and Weinstein
for inversion-symmetric magnetic double wells.  We construct a number
$L_0>0$ and one smooth, compactly supported, nonpositive, nonradial
single-well potential $v$, independent of both the coupling parameter
$\lambda$ and the well displacement, such that for every fixed $L\ge L_0$
the associated magnetic double-well Hamiltonian has infinitely many exact
degeneracies of its two lowest eigenvalues as $\lambda\to\infty$.
The magnetic hopping coefficient also vanishes along an infinite sequence
tending to infinity.  The potential consists of a radial core and two
small, reflection-related perturbations supported on one-sided quadratic
cusps.  Their log-flat profiles isolate two dominant hopping contributions,
which we evaluate by steepest descent after inserting the radial
ground-state asymptotics.  The resulting asymptotic cosine expression, together with
error bounds for the even--odd splitting, yields infinitely many exact eigenvalue
crossings and changes of ground-state parity.
\end{abstract}
\tableofcontents

\section{Introduction and main theorem}\label{sec:introduction}

We study a single electron in $\R^2$ subject to a constant magnetic field
and a double-well electric potential. Let $P=-i\nabla$,
$x^\perp=(-x_2,x_1)$, and fix $b>0$. For a real one-well potential
$v\in C_c^\infty(\R^2;[-1,0])$ and a displacement $d=(L,0)$, consider
\begin{align}
 h_v(\lambda)&=\left(P-\frac{b\lambda}{2}x^\perp\right)^2
 +\lambda^2v(x),\label{eq:one-well-unscaled}\\
 H_v(\lambda)&=\left(P-\frac{b\lambda}{2}x^\perp\right)^2
 +\lambda^2v(x+d)+\lambda^2v(-x+d).
 \label{eq:double-well-unscaled}
\end{align}
The well centers are separated by $2L$. We suppress the fixed displacement
parameter $d$ from the notation for the double-well operator. 
The parameter $\lambda$ will be taken to be large. 

Quantum tunneling between two deep, well-separated potential wells is
reflected in the splitting between the two lowest eigenvalues of the
double-well Schr\"odinger operator. A low-energy state initially localized
in one well moves to the other on a time scale inversely proportional to
this splitting. In the absence of a magnetic field, the ground state is
simple and can be chosen strictly positive, and tunneling cannot vanish.
For symmetric double wells, each with a non-degenerate minimum, the
splitting is exponentially small in the semiclassical limit, with its
exponential rate governed by the Agmon distance between the wells; see
\cite{HS84,Simon84}. Exponential lower bounds for compactly supported wells
in the strong-binding regime were obtained in \cite{FLW18}; see
\cite{FSW-excited} for results on excited states.

A constant magnetic field introduces complex phases into the interaction
between localized states and permits cancellations in the tunneling
amplitude. Helffer and Sj\"ostrand developed a semiclassical theory of magnetic
tunneling for sufficiently small magnetic coupling, under analyticity
assumptions near the wells and the relevant minimizing Agmon geodesics
\cite{HS87}. Their analysis includes magnetic flux-induced cancellation in
symmetric configurations with two minimizing geodesics.
Related examples of ground-state degeneracy based on Aharonov--Bohm
and Kramers-symmetry mechanisms are discussed in
\cite{Helffer87,Helffer88}. In the regime of deep wells and strong magnetic fields,
where the electric potential is scaled by $\lambda^2$ and the magnetic field
strength is $b\lambda$ with $b>0$ fixed, the splitting for radial single wells
remains strictly positive \cite{FSW22,HK24,Mor24}. This conclusion can fail
without radial symmetry. In \cite{FSW-PNAS,FSW-absence}, we constructed nonradial
single-well potentials for which the two lowest eigenvalues of the
inversion-symmetric double well coincide exactly. By varying the potential,
we could also change the parity of the ground state from even to odd.
See \cite{EM25} for a related asymmetry mechanism with magnetic wells, and
\cite{SW22} for the tight-binding consequences of vanishing hopping.

The potential in that construction depends on $\lambda$. More precisely,
for each prescribed, sufficiently large $\lambda$, we added small exterior
perturbations, called \emph{sophons}, to a radial reference well. Their
amplitudes, support sizes, and positions were chosen as functions of
$\lambda$; a further adjustment of their positions produced a zero of the
splitting, or a zero of the hopping coefficient. This left open the
question, posed in \cite{FSW-PNAS}, of whether a fixed potential could
produce infinitely many such zero splittings as $\lambda\to\infty$.

In this paper we answer that question affirmatively. We construct one
smooth, compactly supported, nonpositive, nonradial single-well potential
$v$, independent of $\lambda$, for which tunneling vanishes at infinitely
many values of $\lambda$ tending to infinity. The same potential has
infinitely many zeros of the magnetic hopping coefficient. Thus the
earlier result chooses a potential after fixing $\lambda$, whereas here
\emph{the potential is fixed first and only $\lambda$ varies}. Moreover,
$v$ is independent of the well displacement: one construction works for
every fixed $d=(L,0)$ with $L$ sufficiently large. The sequences of
exceptional couplings may depend on $L$, and the zeros of the splitting and
of the hopping coefficient need not coincide. As $\lambda$ increases, the
ground state changes parity infinitely often, and at each sufficiently large
crossing its eigenspace is exactly two-dimensional, with one even and one odd
eigenfunction. These conclusions complement the generic lower bounds established in
\cite{FSW-lower}, which hold outside an exceptional set of zero Lebesgue
density. For each fixed sufficiently large displacement $L$, our
construction yields a sequence of exact crossings $\lambda_j\to\infty$
that is asymptotically equally spaced:
\[
\lambda_{j+1}-\lambda_j\longrightarrow\frac{\pi}{\Phi_*}.
\]
Thus exceptional couplings occur arbitrarily far into the strong-binding
regime. Moreover, this sequence satisfies
$\sum_j\lambda_j^{-1}=\infty$, providing an endpoint comparison with the
sparsity condition
$\sum_j\lambda_j^{-(1+\varepsilon)}<\infty$ in
\cite[Theorem~1.1(2)]{FSW-lower}.

Our formulae \Cref{eq:cosine formula for rho} and
\Cref{eq:introestimgap} for the hopping coefficient and the signed
eigenvalue splitting exhibit the same interference structure as
\cite[(4.13)--(4.14)]{HS87}: two complex-conjugate contributions
produce an dominant oscillatory cosine factor. In \cite{HS87}, these
contributions correspond to two symmetry-related minimizing Agmon
geodesics, and the analysis in Section~4 establishes both degeneracies
of the lowest eigenvalue and antisymmetric ground states. Here, the
phase $\Theta(\lambda)=\lambda\Phi_*+\vartheta(\lambda)$ in
\Cref{eq:intro-phase} arises from two dominant cusp contributions.
The novelty here is the realization of these phenomena for two
widely separated, inversion-related copies of a fixed smooth,
compactly supported, nonpositive, nonradial single-well potential
in a constant magnetic field. For this single potential, we obtain
infinitely many exact eigenvalue crossings and changes of
ground-state parity as $\lambda\to\infty$ in the scaling
\eqref{eq:double-well-unscaled}. Moreover, the potential is
independent of the well displacement: the same construction works
for every sufficiently large fixed separation.

The choice of compactly supported atomic wells is of interest in the modeling of macroscopic crystalline quantum materials. For example, a natural model is one where the  double well potential is replaced by a potential formed by centering the atomic potential at each vertex of an infinite discrete set, which can be translation invariant (periodic) or disordered. Under a suitable non-clustering assumption on the points of the discrete set, the large $\lambda$ regime, is characterized by a discrete effective ({\it tight binding}) magnetic  Hamiltonian. Such models are widely used in the physics community. A central point in a rigorous derivation of a limiting  Hamiltonian is an appropriate lower bound on the double-well magnetic hopping coefficient between minimally separated atoms \cite{SW22}. The  examples of vanishing hopping that we present in this article,  illustrate the non-applicability of the standard limiting tight-binding model for certain  atomic well configurations. 

The potential in \Cref{eq:double-well-unscaled} is invariant under inversion
$x\mapsto-x$. Hence $H_v(\lambda)$ preserves the even and odd subspaces.
Write
\[
 E_{\rm even}(\lambda)=\inf\sigma(H_v(\lambda)|_{L^2_{\rm even}}),
 \qquad
 E_{\rm odd}(\lambda)=\inf\sigma(H_v(\lambda)|_{L^2_{\rm odd}}),
\]
and define the signed splitting
\begin{equation}\label{eq:signed-splitting-intro}
 S_v(\lambda)=E_{\rm odd}(\lambda)-E_{\rm even}(\lambda).
\end{equation}
We write the discrete eigenvalues in increasing order, counted with
multiplicity, as
\[
 E_0(\lambda)\le E_1(\lambda)\le\cdots .
\]
For the isolated lowest doublet considered below,
$E_1(\lambda)-E_0(\lambda)=|S_v(\lambda)|$.

\begin{theorem}[The fixed-potential FSW conjecture]\label{thm:main}
There exist $b>0$, $L_0>0$, and a fixed nonradial potential
\[
 v\in C_c^\infty(\R^2;[-1,0])
\]
such that, for any choice of $L\ge L_0$ and $d=(L,0)$, the following hold.
\begin{enumerate}[label=\textup{(\roman*)},leftmargin=2.5em]
\item There are infinitely many values $\lambda_n\to\infty$ for which the
two lowest eigenvalues of $H_v(\lambda)$ coincide:
\[
 E_1(\lambda_n)-E_0(\lambda_n)=0.
\]
At each sufficiently large crossing the ground eigenspace is exactly
two-dimensional, with one even and one odd eigenfunction.
\item The magnetic hopping coefficient $\rho_\lambda$ defined in
\Cref{eq:rho-intro} has infinitely many zeros tending to infinity.
\item The parity of the ground state changes infinitely often as
$\lambda\to\infty$.
\end{enumerate}
\end{theorem}

The order of choices in the theorem is important: the potential and the
threshold $L_0$ are constructed first, and only then is an admissible
displacement $L$ fixed. After the parameter hierarchy at the end of
\Cref{sec:construction}, $L$ is treated as a background parameter:
it is omitted from all subscripts and retained only in the explicit bridge
and phase formulas.

The one-well potential used in the proof consists of a radial core together
with two small components $q_+$ and $q_-$, supported on one-sided quadratic
cusps. The leading contribution to $\rho_\lambda$ comes from pairing the
$q_+$ source in one well with the $q_-$ source in the other, and conversely.
We refer to these two dominant contributions as the active cross-channels,
and denote by $\mathfrak a_{1/\lambda}>0$ their common positive envelope.

Throughout the analysis we use $h=\lambda^{-1}$. Let $\phi_h^L$ and
$\phi_h^R$ denote the normalized ground states of the left and right shifted
wells, as defined in \Cref{sec:scaled-double-well}. The hopping
coefficient is
\begin{equation}\label{eq:rho-intro}
 \rho_\lambda:=h^{-2}\ip{\phi_h^L}{v^L\phi_h^R},
 \qquad v^L(x)=v(x+d).
\end{equation}
Here $R>0$ denotes the fixed radius of the cusp tips in the one-well
construction. It is chosen before $L$ and is independent of $L$ and
$\lambda$. The relevance of the hopping coefficient to the spectral
splitting follows from the parity Schur reduction, which gives
\begin{equation}\label{eq:introestimgap}
 E_{\rm odd}(\lambda)-E_{\rm even}(\lambda)
 =-2\rho_\lambda
 +\cO\!\left(e^{-c\lambda}\mathfrak a_{1/\lambda}\right).
\end{equation}
The cross-channel calculation gives
\begin{equation}\label{eq:introestimrhol}
 \rho_\lambda
 =-2|\mathscr C_*|\mathfrak a_{1/\lambda}
 \left[
  \cos\Theta(\lambda)
  +\cO\!\left((\log\lambda)^{-1}\right)
 \right],
 \qquad \mathscr C_*\ne0,
\end{equation}
where the phase has the form
\begin{equation}\label{eq:intro-phase}
 \Theta(\lambda)=\lambda\Phi_*+\vartheta(\lambda),
 \qquad
 \Phi_*=\sqrt3\,bR\left(L-\frac R4\right)>0,
\end{equation}
with $\vartheta(\lambda)=\cO(\log\lambda)$ and
$\vartheta'(\lambda)=\cO(\lambda^{-1})$. Thus $\Theta$ is strictly
increasing for all sufficiently large $\lambda$. 
\Cref{eq:introestimgap}--\Cref{eq:introestimrhol} yield alternating signs
on successive phase-point sequences, and continuity gives infinitely many
zeros of both the hopping coefficient and the signed splitting, together
with infinitely many changes of ground-state parity. The zeros in each sequence
can be chosen with asymptotic spacing $\pi/\Phi_*$.

\medskip\noindent\textbf{Idea of the construction.}
We begin with a radial core $v^\circ$ and add two small, reflected sophons:
\[
 v=v^\circ+\eps(q_++q_-);
\]
see \Cref{fig:dominant-cusp-pairings}. Here $\eps>0$ and the functions
$q_\pm$ are fixed independently of $\lambda$ and $L$. The compactly
supported sophons $q_\pm$ lie outside the radial core, with tips at
\begin{equation}\label{eq:cusplocation}
 p_\pm=\left(\frac R2,\,\pm\frac{\sqrt3R}{2}\right),
\end{equation}
and $q_-(x_1,x_2)=q_+(x_1,-x_2)$. Their supports are one-sided quadratic
cusps. Moving a distance $t>0$ into either cusp takes one radially outward
from the core and to the left of its tip, while the transverse width is of
order $t^2$. In the expansion of the exact integral formula for the hopping
coefficient, this geometry makes the two terms arising from pairing $q_+$
with $q_-$, in the two possible orders, dominant among the nine pairings.
The physical tips are $-d+p_\pm$ on the left and $d-p_\pm$ on the right;
thus the dominant pairs connect tips at the same height, above or below
the horizontal axis. The other seven pairings of the core and sophon
sources have strictly larger real decay actions when $L$ is sufficiently
large.

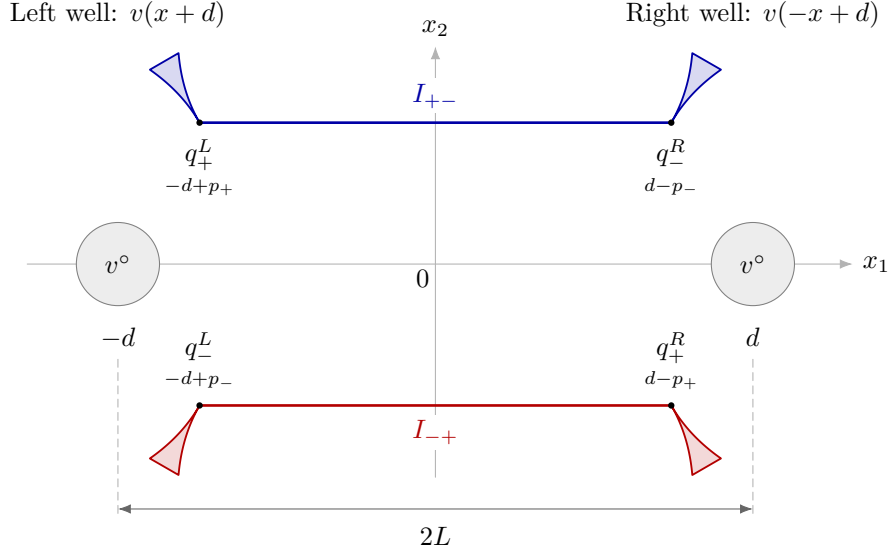
\begin{figure}[t]
\centering
\begin{tikzpicture}[
  x=1.2cm,y=1.2cm,font=\small,
  line cap=round,line join=round,
  axis/.style={-{Latex[length=1.7mm]},black!30,thin},
  tiplabel/.style={align=center,inner sep=2pt}
]
  \def\figL{3.5}
  \def\figR{1.8}
  \pgfmathsetmacro{\figX}{\figL-\figR/2}
  \pgfmathsetmacro{\figY}{sqrt(3)*\figR/2}

  \coordinate (CL) at (-\figL,0);
  \coordinate (CR) at ( \figL,0);
  \coordinate (Lp) at (-\figX, \figY); 
  \coordinate (Lm) at (-\figX,-\figY); 
  \coordinate (Rm) at ( \figX, \figY); 
  \coordinate (Rp) at ( \figX,-\figY); 

  \draw[axis] (-4.5,0) -- (4.6,0)
    node[right,text=black] {$x_1$};
  \draw[axis] (0,-2.35) -- (0,2.4)
    node[above,text=black] {$x_2$};
  \node[below left,inner sep=2pt] at (0,0) {$0$};

  \draw[black!50,fill=black!7] (CL) circle[radius=5.5mm]
    node[text=black] {$v^\circ$};
  \draw[black!50,fill=black!7] (CR) circle[radius=5.5mm]
    node[text=black] {$v^\circ$};
  \node[below=7mm] at (CL) {$-d$};
  \node[below=7mm] at (CR) {$d$};
  \node at (-\figL,2.75) {Left well: $v(x+d)$};
  \node at ( \figL,2.75) {Right well: $v(-x+d)$};

  \foreach \tipname/\cuspangle/\cuspcolor in {
    Lp/120/blue!65!black, Rm/60/blue!65!black,
    Lm/240/red!70!black, Rp/300/red!70!black}{
    \begin{scope}[shift={(\tipname)},rotate=\cuspangle]
      \path[draw=\cuspcolor,fill=\cuspcolor,fill opacity=0.15,
            line width=0.7pt]
        plot[domain=0:0.78,samples=25] (\x,{0.30*\x*\x})
        -- plot[domain=0.78:0,samples=25] (\x,{-0.30*\x*\x})
        -- cycle;
    \end{scope}
  }

  \draw[blue!65!black,line width=1pt] (Lp) --
    node[above=3pt,fill=white,inner sep=2pt] {$I_{+-}$} (Rm);
  \draw[red!70!black,line width=1pt] (Lm) --
    node[below=3pt,fill=white,inner sep=2pt] {$I_{-+}$} (Rp);

  \foreach \tipname in {Lp,Lm,Rm,Rp}
    \fill (\tipname) circle[radius=1.2pt];

  \node[tiplabel,anchor=north,yshift=-4pt] at (Lp)
    {$q_+^L$\\[-1pt]$\scriptstyle -d+p_+$};
  \node[tiplabel,anchor=north,yshift=-4pt] at (Rm)
    {$q_-^R$\\[-1pt]$\scriptstyle d-p_-$};
  \node[tiplabel,anchor=south,yshift=4pt] at (Lm)
    {$q_-^L$\\[-1pt]$\scriptstyle -d+p_-$};
  \node[tiplabel,anchor=south,yshift=4pt] at (Rp)
    {$q_+^R$\\[-1pt]$\scriptstyle d-p_+$};

  \draw[black!30,densely dashed]
    (-\figL,-1.05) -- (-\figL,-2.8);
  \draw[black!30,densely dashed]
    ( \figL,-1.05) -- ( \figL,-2.8);
  \draw[{Latex[length=1.7mm]}-{Latex[length=1.7mm]},black!60]
    (-\figL,-2.7) --
    node[below=3pt,text=black] {$2L$} (\figL,-2.7);
\end{tikzpicture}
\caption{The two dominant cusp pairings in the inversion-symmetric
double well, with $d=(L,0)$ and $p_\pm=(R/2,\pm\sqrt{3}R/2)$.
Here $q_\pm^L(x)=q_\pm(x+d)$ and $q_\pm^R(x)=q_\pm(-x+d)$.
The upper pair joins $-d+p_+$ to $d-p_-$ and contributes $I_{+-}$;
the lower pair joins $-d+p_-$ to $d-p_+$ and contributes $I_{-+}$.
Thus opposite one-well labels correspond to cusps at the same physical
height. The colored segments denote source pairings in the hopping
integral. The geometry is schematic, with core sizes and cusp widths
enlarged for visibility.}
\label{fig:dominant-cusp-pairings}
\end{figure}

The choice of profile within each cusp is important. In the cusp
coordinates of \Cref{sec:construction}, its normal factor is
\[
 \exp\!\left[-\beta\left(\log\frac{t_*}{t}\right)^2\right],
 \qquad t>0,
\]
with fixed $\beta,t_*>0$ and smooth cutoffs away from the tip. This factor
vanishes to infinite order at $t=0$, so the sophons extend smoothly by zero.
The balance between this profile and the linear increase of the decay
action selects
\[
 t\asymp h\log(1/h).
\]
On this scale the transverse displacements contribute only
$\cO(h\log^2(1/h))=o(1)$ to the exponent and are therefore negligible.
Thus the fixed cusp geometry concentrates the leading interaction near
its tips as $h\to0$.

To evaluate that interaction, we first insert the radial ground-state
Ansatz into the hopping integral. In fact, outside the core the radial
ground state has the exact representation
\[
 \phi_h^\circ(x)=\Gamma_h K_h^{-E_h^\circ}(x),
 \qquad |x|>r_0,
 \qquad \Gamma_h>0,
\]
where $E_h^\circ$ is the scaled one-well energy and $K_h^\cE$ is the radial
amplitude factor of the Landau Hamiltonian resolvent, which has Gaussian
decay. Substituting its asymptotic expansion into the hopping integral,
we obtain an oscillatory integral whose exponent contains the two radial
decay actions, the action for propagation between the sophons (the
\emph{bridge action}), and the large magnetic phase. We then apply
steepest descent. The leading normal integrals take the form
\[
 \int_0^{t_0}t^2\chi_a(t)
 \exp\!\left[
 -\beta\left(\log\frac{t_*}{t}\right)^2
 -\frac{(\alpha-i\theta)t}{h}
 \right]\dd t,
 \qquad \alpha>0,
\]
where $\alpha$ and $\theta$ are the normal action and phase slopes, and the
factor $t^2$ comes from the cusp Jacobian. The logarithmic term produces
an $h$-dependent critical point near $t=0$. The complex saddle is described
by the principal branch of the Lambert-$W$ function. Holomorphic control
of the radial and resolvent expansions along the deformed contours gives
an error small relative to the positive envelope $\mathfrak a_h$.

Our use of the radial ground-state approximation to the exact one-well
ground state must be justified. On each sophon, we separate the source
into the incoming radial contribution and the scattered response generated
by the full perturbation $\eps(q_++q_-)$, including scattering involving
either cusp. The leading term contains two log-flat factors, one from each
sophon. Weighted resolvent estimates show that every term involving a
scattered response incurs at least one additional log-flat decay cost.
Such terms are smaller than the positive leading envelope by a factor
$e^{-c\log^2(1/h)}=o(h^N)$ for every fixed $N>0$, with $c>0$.
This separation allows the sophon amplitudes to remain fixed as $h=\lambda^{-1}\to0$.

As asserted in \Cref{eq:introestimrhol}, the two dominant contributions
are complex conjugates of one another and yield
\begin{align}\label{eq:cosine formula for rho}
 \rho_\lambda
 =-2|\mathscr C_*|\mathfrak a_{1/\lambda}
 \left[\cos\Theta(\lambda)
 +\cO\!\left((\log\lambda)^{-1}\right)\right],
 \qquad \mathscr C_*\ne0,
\end{align}
with the phase given by \Cref{eq:intro-phase}. The error in the spectral
reduction \Cref{eq:introestimgap}, divided by the positive envelope
$\mathfrak a_{1/\lambda}$, is $\mathcal O(e^{-c\lambda})$ for all
sufficiently large $\lambda$. This envelope measures the size of the
dominant contributions before their oscillatory cancellation and remains
positive even where $\rho_\lambda$ vanishes. At successive phase points
where the cosine is $+1$ and $-1$, the magnitude of $\rho_\lambda$ is
comparable to $\mathfrak a_{1/\lambda}$. The reduction error is therefore
too small to change the sign of $-2\rho_\lambda$, so the exact signed
splitting also alternates in sign. Continuity then yields infinitely many
hopping zeros and infinitely many exact eigenvalue crossings, although
the two sequences need not coincide.

\medskip\noindent\textbf{Organization of the paper.}
\Cref{sec:construction} constructs the fixed potential.
\Cref{sec:one-well-source} and~\Cref{sec:schur-reduction} control the
one-well source and carry out the parity Schur reductions.
\Cref{sec:exact-hopping} derives an exact formula for the hopping
coefficient (see also \cite[(2.29)]{FSW-absence}) and isolates the two
dominant contributions, which are evaluated in
\Cref{sec:active-integral}.
\Cref{sec:full-hopping-section} and~\Cref{sec:spectral-transfer}
establish the sign changes in the hopping coefficient and transfer them
to eigenvalue crossings in the exact spectrum of the magnetic double-well
Hamiltonian. The appendices contain auxiliary resolvent estimates, the
complex steepest descent arguments and a formalization of the proof, up to classical results in the literature.

\begin{remark}
One can additionally impose reflection symmetry across the $x_2$-axis,
$v(-x_1,x_2)=v(x_1,x_2)$, resulting in four cusp components surrounding the
radial core instead of two. This symmetry is useful in multi-well
constructions; see \cite{FSW-PNAS}. Contributions from the additional cusp
components are negligible, as their actions are strictly larger than those
of the dominant contributions. Thus the same leading asymptotics hold.
\end{remark}

\begin{remark}
A similar asymptotic expansion in $L$ can be derived by the methods used
in the proof of \Cref{thm:main}. If we fix $\lambda$ instead of the
displacement $L$, the $L\to\infty$ asymptotic takes the form
\[
 \rho_L=-\cA_\lambda(L)\left[
 \cos\Theta_\lambda(L)+e^{-\frac34 b\lambda R^2}
 +O_\lambda\!\left(\frac{1}{\log L}\right)\right].
\]
Since $0<e^{-\frac34 b\lambda R^2}<1$, this yields infinitely many hopping
zeros provided $\cA_\lambda(L)>0$ and $\Theta_\lambda(L)$ is continuous and
unbounded. We shall return to these asymptotics and their spectral
consequences in \cite{FKSW-large-separation}.
\end{remark}


\begin{remark}
The main difference from \cite{FSW-absence} is how the sophon contributions
are controlled. There, their $\lambda$-dependent sizes permit a perturbative
analysis; here, their fixed profiles are chosen so that the leading
interaction can be evaluated by contour deformation.
\end{remark}

\section*{AI statement}
The authors used a workflow combining ChatGPT Pro 5.5/5.6 and Claude Fable 5 during the development of this work. The authors first found an example independently and subsequently used AI models to help identify the simpler example presented here. This was achieved through a process of iterative prompting, refinement, and evaluation across multiple interactions rather than through a single model query thread. For example, we discarded several routes that were not promising and we refined the results of our main theorem multiple times. All mathematical statements and proofs were independently and rigorously verified by the authors, who take full responsibility for the contents of the paper.

\section*{Acknowledgements}
KB was partially supported by NSF under grant DMS-2434314. CF was supported in part by NSF grant DMS-1700180. JGS was partially supported by NSF, under Grants DMS-2245017,
DMS-2554957 and DMS-2434314. AH was partially supported by French National Research Agency (ANR), under grant ANR-24-ERCS-0010, and by Hi! PARIS and State funding managed by the ANR under the France 2030 program, ANR-23-IACL-0005. JS was supported in part by NSF grant DMS-2510207. MIW was supported in part by NSF grants DMS-1908657, DMS-1937254, DMS-2510769  and Simons Foundation Math + X Investigator Award \# 376319 (MIW). Part of this research was carried out during the 2023-24 academic year, when MIW was a Visiting Member in the School of Mathematics - Institute of Advanced Study, Princeton, supported by the Charles Simonyi Endowment, and a Visiting Fellow in the Department of Mathematics at Princeton University.

\section{Construction of the fixed potential}\label{sec:construction}

In this section we construct the potential
\begin{equation}\label{eq:voverview}
 v=v^\circ+\eps(q_++q_-).
\end{equation}
To begin, we will fix a radial core $v^\circ$ and control the exterior tail of its ground state.

We will then place two reflected one-sided cusps, $q_\pm$, at distance $R$ from the
core as in \Cref{eq:cusplocation} and fix their supports, log-flat profiles, and amplitudes.  These are treated as perturbations for our analysis, and their effects will be controlled. After the potential is
complete, a lower bound on the well separation is imposed by the action
certificate.

\subsection{Magnetic scaling and notation}\label{sec:framework}
Put $h=\lambda^{-1}$ and introduce the semi-classically scaled Landau Hamiltonian:
\begin{equation}\label{eq:Dh}
 H_{h}^{\mathrm{Lan}}:=\left(hP-\frac b2x^\perp\right)^2.
\end{equation}
For a one-well potential $u$, write
\[
 H_h^{\rm atom}(u)=H_{h}^{\mathrm{Lan}}+u.
\]

\subsection{The radial reference well}\label{sec:radial-one-well}
Fix $r_0>0$, and take the explicit radial core $v^\circ$ by
\begin{equation}\label{eq:explicit-core}
 v^\circ(x):=
 \begin{cases}
 -\exp\!\left(-\dfrac{|x|^2}{r_0^2-|x|^2}\right),&|x|<r_0,\\[1ex]
 0,&|x|\ge r_0.
 \end{cases}
\end{equation}
Then $v^\circ\in C_c^\infty(\R^2;[-1,0])$ and has the unique global minimum $v^\circ(0)=-1$. 
Additionally,
 $\nabla^2v^\circ(0)=\omega^2 I$ and $\omega^2=2/r_0^2$.
Denote the normalized ground-state pair of $H_h^{\rm atom}(v^\circ)$ by $(E_h^\circ,\phi_h^\circ)$.

\subsubsection{Radial spectral data}
Next compute the relevant properties of the base well, beginning with its energy and gap.
\begin{lemma}[Radial one-well energy and gap]\label{lem:radial-spectral}
Let
\begin{equation}\label{eq:oscillator-operator}
 \mathcal H_{\rm osc}
 :=\left(-i\nabla_y-\frac b2y^\perp\right)^2
 +\frac{\omega^2}{2}|y|^2,
\end{equation}
and let $(\mu_0,\psi_0)$ be its normalized positive ground pair.  There are
$h_0,c_{\rm gap}>0$ such that, for $0<h<h_0$,
\begin{equation}\label{eq:base-energy}
 E_h^\circ=-1+\mu_0h+\cO(h^{3/2}),
\end{equation}
and
\begin{equation}\label{eq:base-gap}
 \dist\!\left(E_h^\circ,\sigma(H_h^{\rm atom}(v^\circ))\setminus\{E_h^\circ\}\right)
 \ge c_{\rm gap}h.
\end{equation}
The ground state is simple, radial, and can be chosen as a strictly positive
function of $r$. Moreover, for
\begin{equation}\label{eq:rescaled-radial-ground-state}
 \psi_h(y):=h^{1/2}\phi_h^\circ(h^{1/2}y),
\end{equation}
one has $\psi_h\to\psi_0$ strongly in $L^2_{\rm loc}(\R^2)$. 
\end{lemma}

\begin{proof}
This follows directly from \cite{HK24}.
Noting the scaling relation
$H_h^{\rm atom}(v^\circ)=b^2\mathcal L_{h/b}^{\rm sw}(v^\circ/b^2)$ in the notation of
\cite[Theorem~1.1, Proposition~2.1]{HK24}, the energy expansion and the
order-$h$ gap follow from those results.  Ground-state simplicity, radiality,
strict positivity, and strong local convergence of the rescaled ground state follow from the rotational symmetry and the ground-state
characterization in \cite[Proposition~2.1]{HK24}.
\end{proof}

In particular, after decreasing $h_0$ if necessary,
\begin{equation}\label{eq:radial-energy-window}
 \frac12<-E_h^\circ<\frac32.
\end{equation}
Thus the exterior energy parameter $\cE_h^\circ=-E_h^\circ$ remains in one
fixed compact subinterval of $(0,\infty)$, as required by the uniform
Landau-kernel estimates below. 

\subsection{Exact exterior tail and the bridge action}\label{sec:radial-exterior-tail}
Next we compute the exterior tail of the core well and its corresponding bridge action. Note that outside $\supp v^\circ$, the radial ground state solves the free Landau
equation
\[
 (H_{h}^{\mathrm{Lan}}-E_h^\circ)\phi_h^\circ=0.
\]
Since $\cE_h^\circ=-E_h^\circ>0$ for small $h$, its decaying exterior branch is governed by the free Landau resolvent $(H_{h}^{\mathrm{Lan}}+\cE)^{-1}$.  We first identify the exact resolvent representation and the exponential rate that it determines.

\subsubsection{Landau resolvent and the bridge action}
For $\cE>0$, the semigroup representation gives
\[
 (H_{h}^{\mathrm{Lan}}+\cE)^{-1}
 =\int_0^\infty e^{-t\cE}e^{-tH_{h}^{\mathrm{Lan}}}\,\dd t.
\]
In the symmetric gauge, the exact Landau heat kernel is
\[
 e^{-tH_{h}^{\mathrm{Lan}}}(x,y)
 =\frac{b}{4\pi h\sinh(bht)}
 e^{-\frac{ib}{2h}x\wedgep y}
 \exp\!\left[-\frac{b}{4h}\coth(bht)|x-y|^2\right].
\]

Consequently, the Landau resolvent $(H_{h}^{\mathrm{Lan}}+\cE)^{-1}$ has the symmetric-gauge kernel
\begin{equation}\label{eq:landau-kernel-master}
 (H_{h}^{\mathrm{Lan}}+\cE)^{-1}(x,y)
 =e^{-\frac{ib}{2h}x\wedgep y}K_h^\cE(x-y),
\end{equation}
where $K_h^\cE$ is radial and real-valued.   After the change of variables $\tau=ht$, its radial part is
\begin{equation}\label{eq:K-integral}
 K_h^{\cE}(r)
 =\frac{b}{4\pi h^2}
 \int_0^\infty
 \frac{1}{\sinh(b\tau)}
 \exp\!\left[-\frac1h\left(
 \cE\tau+\frac b4\coth(b\tau)r^2
 \right)\right]\dd\tau,
 \qquad r>0.
\end{equation}
Set
\[
 \mathcal H(\tau;r,\cE)
 :=\cE\tau+\frac b4\coth(b\tau)r^2.
\]
Thus \Cref{eq:K-integral} is a one-dimensional Laplace integral with phase $\mathcal H$.  Its exponential rate is therefore the minimum of this phase in the proper-time variable.  We define the bridge action by
\begin{equation}\label{eq:J}
 \cJ_\cE(r)
 :=\inf_{\tau>0}\mathcal H(\tau;r,\cE)
 =\inf_{\tau>0}\left(\cE\tau+\frac b4\coth(b\tau)r^2\right).
\end{equation}
The point of this representation is that the same function $\cJ_\cE$ will govern both the Landau resolvent decay and the exterior radial tail.

\begin{proposition} \label{prop:Jminimizer}
For $r>0$ and $\cE>0$, the infimum in \Cref{eq:J} is attained at a unique point $\tau_*(r,\cE)>0$, which satisfies
\begin{equation}\label{eq:radial-tau-star}
 \sinh(b\tau_*(r,\cE))=\frac{br}{2\sqrt\cE}.
\end{equation}
Consequently,
\begin{align}
 \cJ_\cE(r)
 &=\frac{\cE}{b}\operatorname{arsinh}\!\left(\frac{br}{2\sqrt\cE}\right)
 +\frac b4r^2\sqrt{1+\frac{4\cE}{b^2r^2}},\label{eq:J-explicit}\\
 \partial_r\cJ_\cE(r)
 &=\frac12\sqrt{b^2r^2+4\cE}.
 \label{eq:radial-J-prime}
\end{align}
In particular,
\begin{equation}\label{eq:J-eikonal}
 \bigl(\partial_r\cJ_\cE(r)\bigr)^2
 =\cE+\frac{b^2r^2}{4},
 \qquad
 \cJ_\cE(r)
 =\int_0^r\sqrt{\cE+\frac{b^2s^2}{4}}\,\dd s.
\end{equation}
\end{proposition}
\begin{proof}
Differentiating the phase gives
\[
 \partial_\tau\mathcal H
 =\cE-\frac{b^2r^2}{4}\operatorname{csch}^2(b\tau),
\]
which strictly increases from $-\infty$ to $\cE$ on $(0,\infty)$.  Hence it
has a unique zero, and solving the critical-point equation gives
\Cref{eq:radial-tau-star}.  Moreover,
\[
 \partial_\tau^2\mathcal H(\tau_*)
 =\frac{b^3r^2}{2}\operatorname{csch}^2(b\tau_*)
  \coth(b\tau_*)>0.
\]
Put $s=br/(2\sqrt\cE)$.  At the critical point,
$b\tau_*=\operatorname{arsinh}s$ and
$\coth(b\tau_*)=\sqrt{1+s^2}/s$.  Substitution in
\Cref{eq:J} yields \Cref{eq:J-explicit}.  Finally, the critical point is
nondegenerate, so differentiation of the minimized action with respect to
$r$ gives
\[
 \partial_r\cJ_\cE(r)
 =\partial_r\mathcal H(\tau_*;r,\cE)
 =\frac b2r\coth(b\tau_*)
 =\frac12\sqrt{b^2r^2+4\cE},
\]
which is \Cref{eq:radial-J-prime}.  Squaring gives the eikonal identity in
\Cref{eq:J-eikonal}. Since $\cJ_\cE$ extends continuously to $r=0$ with $\cJ_\cE(0)=0$, integration yields the second formula there.
\end{proof}

As a useful consequence of the explicit formula, uniformly for $\cE$ in a compact subset of $(0,\infty)$, the large-$r$ expansions give that the explicit action satisfies
\begin{equation}\label{eq:J-large-r}
 \cJ_\cE(r)=\frac b4r^2+\cO(\log(2+r)),\qquad r\to\infty.
\end{equation}

\subsubsection{Landau kernel asymptotics}
We next apply Laplace's method to \Cref{eq:K-integral}.  This gives the complete differentiated expansion of the Landau kernel with exponential rate $\cJ_\cE$.
\begin{proposition}[Differentiated radial Laplace expansion]\label{prop:K-Laplace}
Fix compact intervals $I_r\Subset(0,\infty)$ and $I_\cE\Subset(0,\infty)$.  For every $N,k$ there are smooth functions $k_j(r,\cE)$, $0\le j<N$, such that, uniformly for $(r,\cE)\in I_r\times I_\cE$,
\begin{equation}\label{eq:K-expansion}
 K_h^\cE(r)
 =h^{-3/2}e^{-\cJ_\cE(r)/h}
 \left(\sum_{j=0}^{N-1}h^jk_j(r,\cE)+h^NR_{N,h}(r,\cE)\right),
\end{equation}
with
\begin{equation}\label{eq:K-remainder}
 \sup_{0\le \ell+m\le k}
 \abs{\partial_r^\ell\partial_\cE^mR_{N,h}(r,\cE)}\le C_{N,k}.
\end{equation}
The leading coefficient is strictly positive:
\begin{equation}\label{eq:k0}
 k_0(r,\cE)
 =\frac{b}{4\pi\sinh(b\tau_*)}
 \left(\frac{2\pi}{\partial_\tau^2\mathcal{H}(\tau_*;r,\cE)}\right)^{1/2}>0,
\end{equation}
where $\tau_*=\tau_*(r,\cE)$ is the minimizer from \Cref{prop:Jminimizer}.  The asymptotic expansion remains valid after taking any fixed spatial derivative provided one uses the rescaled derivative $(h\partial_{x})^{\alpha} = h^{|\alpha|}\partial_{x}^{\alpha}.$
\end{proposition}

\begin{proof}
For $(r,\cE)$ in the stated compact sets, the two ends of the integral of \Cref{eq:K-integral} are uniformly negligible because
\[
 \mathcal H(\tau;r,\cE)\sim \frac{r^2}{4\tau}
 \quad (\tau\downarrow0),
 \qquad
 \mathcal H(\tau;r,\cE)\sim \cE\tau+\frac{br^2}{4}
 \quad (\tau\to\infty).
\]
After restricting to a fixed compact interval in $\tau$, split \Cref{eq:K-integral} into a fixed neighborhood of $\tau_*$ and its complement.  On the complement, $\mathcal H$ exceeds its minimum by a fixed positive amount.  Since $e^{-\delta/h}$ decays faster than every power of $h$, the asymptotic expansion comes only from a neighborhood of $\tau_*$.  In that neighborhood use the Morse coordinate
\[
    y^2=2\bigl(H(\tau;r,E)-H(\tau_*;r,E)\bigr),
    \qquad \operatorname{sgn} y=\operatorname{sgn}(\tau-\tau_*).
\]
Since the critical point is uniformly nondegenerate, this defines a
smooth change of variables $\tau=T(y;r,E)$, uniformly for
$(r,E)\in I_r\times I_E$, with
\[
    T(0;r,E)=\tau_*(r,E),
    \qquad
    \partial_y T(0;r,E)
    =\bigl(\partial_\tau^2H(\tau_*;r,E)\bigr)^{-1/2}.
\]
Thus the contribution from this neighborhood is
\[
 \frac{b}{4\pi h^2}e^{-J_E(r)/h}
 \int
 \frac{\partial_yT(y;r,E)}
      {\sinh(bT(y;r,E))}
 e^{-y^2/(2h)}\,dy.
\]
Taylor expand the smooth amplitude inside the integral at  $y=0$.  The odd Taylor terms integrate to zero, up to exponentially
small errors from the truncated Gaussian tails, while
\[
 \int_{\mathbb R} y^{2j}e^{-y^2/(2h)}\,dy
 = \sqrt{2\pi h}\,\frac{(2j)!}{2^j j!}\,h^j.
\]
It follows that the integral has a complete expansion in integer powers
of $h$.  Together with the prefactor $h^{-2}$, the Gaussian factor
$\sqrt h$ gives the overall power $h^{-3/2}$ in \Cref{eq:K-expansion}.  The
leading coefficient is
\[
 k_0(r,E)
 =\frac{b}{4\pi\sinh(b\tau_*)}
   \left(
      \frac{2\pi}
           {\partial_\tau^2H(\tau_*;r,E)}
   \right)^{1/2}>0,
\]
This gives \Cref{eq:K-expansion}--\Cref{eq:k0}.  Uniform differentiated bounds follow because all critical-point derivatives are uniform on compact sets.  Spatial differentiation either differentiates a smooth coefficient or produces $h^{-1}\partial_r\cJ$; multiplying by the corresponding power of $h$ gives the final assertion.

\end{proof}

\subsubsection{Exact radial tail and absolute exterior control}
The preceding analysis concerns the free Landau kernel.  We now identify the exact exterior ground-state tail with that kernel.
\begin{proposition}[Exact radial exterior tail]\label{prop:exact-radial-tail}
There is a positive constant $\Gamma_h$ such that
\begin{equation}\label{eq:radial-tail}
 \phi_h^\circ(x)=\Gamma_hK_h^{\cE_h^\circ}(x),
 \qquad |x|>r_0,
 \qquad \cE_h^\circ=-E_h^\circ.
\end{equation}
In particular, on every compact annulus outside the core, \Cref{eq:K-expansion} gives a complete differentiated expansion for the exact normalized radial tail, and its leading coefficient is nonzero.
\end{proposition}

\begin{proof}
Recall from \Cref{lem:radial-spectral} that $\phi_h^\circ$ is positive and radial.  Since $K_h^\cE$ is the radial part of the resolvent kernel,
\[
 (H_{h}^{\mathrm{Lan}}+\cE)K_h^\cE=0
 \qquad\text{on }\R^2\setminus\{0\}.
\]
Taking $\cE=\cE_h^\circ=-E_h^\circ$, both $K_h^{\cE_h^\circ}$ and $\phi_h^\circ$ therefore satisfy the same exterior equation.  For radial functions the angular momentum term in $H_{h}^{\mathrm{Lan}}$ vanishes, so this equation is
\[
 -h^2\left(f''+\frac1rf'\right)+\frac{b^2r^2}{4}f=E_h^\circ f.
\]
The integral representation \Cref{eq:K-integral} shows that $K_h^{\cE_h^\circ}(r)>0$ and decays as $r\to\infty$.  The space of radial solutions decaying at infinity is one-dimensional.  Hence the two functions are proportional, and positivity of both functions gives $\Gamma_h>0$.
\end{proof}

The exact representation \Cref{eq:radial-tail} contains the normalization factor $\Gamma_h$, which will be controlled separately.  For later estimates it is also useful to have an absolute exterior decay bound that does not depend on prior control of this factor.

\begin{lemma}[Absolute exterior Agmon bound]\label{lem:radial-agmon}
Let $K\Subset\R^2\setminus\overline{B_{r_0}(0)}$.  For every $k$ there are $c_K,C_{K,k},M_k>0$ such that
\begin{equation}\label{eq:absolute-agmon}
 \sup_{x\in K}\abs{(h\nabla)^\alpha\phi_h^\circ(x)}
 \le C_{K,k}h^{-M_k}e^{-c_K/h},
 \qquad |\alpha|\le k.
\end{equation}
The same estimate holds uniformly on any closed bounded set whose distance from $\supp v^\circ$ is positive.
\end{lemma}

\begin{proof}
Since $(H_h^{\rm atom}(v^\circ) - E_h^\circ)\phi_h^\circ = 0$, the standard magnetic Agmon identity for the full eigenfunction $\phi_h^\circ$ with a Lipschitz weight $\varphi \ge 0$ reads
\begin{equation}\label{eq:agmon-identity}
 \norm{\left(hP-\frac b2x^\perp\right)(e^{\varphi/h}\phi_h^\circ)}_{L^2}^2 
 + \int_{\R^2} \left(v^\circ(x) - E_h^\circ - |\nabla\varphi(x)|^2\right) e^{2\varphi(x)/h}|\phi_h^\circ(x)|^2 \, dx = 0.
\end{equation}
Let $U \Subset \R^2$ be an open neighborhood containing $\supp v^\circ$. Since $E_h^\circ = -1 + \cO(h)$ and $v^\circ = 0$ on $\R^2 \setminus \supp v^\circ$, we have $v^\circ(x) - E_h^\circ \ge 1/2$ on $\R^2 \setminus U$ for small $h > 0$. We select a smooth, Lipschitz weight function $\varphi \ge 0$ such that $\varphi \equiv 0$ on $U$, $\varphi(x) \ge c_K > 0$ on $K$, and $|\nabla\varphi|^2 \le 1/4$ almost everywhere on $\R^2$.

To control the integral in \Cref{eq:agmon-identity}, we split the integration domain into $U$ and $\R^2 \setminus U$.
On $U$, since $\varphi = 0$, we have $e^{2\varphi/h} = 1$. The potential term $v^\circ - E_h^\circ - |\nabla\varphi|^2$ is bounded below by a constant $-C_0 < 0$. Consequently,
 \[
   \int_U \left(v^\circ - E_h^\circ - |\nabla\varphi|^2\right) e^{2\varphi/h}|\phi_h^\circ|^2 \, dx 
   \ge -C_0 \int_U |\phi_h^\circ|^2 \, dx 
   \ge -C_0 \|\phi_h^\circ\|_{L^2}^2 = -C_0.
 \]
On the complement, we have $v^\circ - E_h^\circ - |\nabla\varphi|^2 \ge 1/2 - 1/4 = 1/4 > 0$, giving
 \[
   \int_{\R^2 \setminus U} \left(v^\circ - E_h^\circ - |\nabla\varphi|^2\right) e^{2\varphi/h}|\phi_h^\circ|^2 \, dx 
   \ge \frac{1}{4} \int_{\R^2 \setminus U} e^{2\varphi/h}|\phi_h^\circ|^2 \, dx.
 \]
Substituting these lower bounds into \Cref{eq:agmon-identity} yields
\[
 \frac{1}{4} \int_{\R^2 \setminus U} e^{2\varphi/h}|\phi_h^\circ|^2 \, dx + \norm{\left(hP-\frac b2x^\perp\right)(e^{\varphi/h}\phi_h^\circ)}_{L^2}^2 
 \le C_0.
\]
Adding $\frac{1}{4}\int_U e^{2\varphi/h}|\phi_h^\circ|^2 \, dx = \frac{1}{4}\int_U |\phi_h^\circ|^2 \, dx \le \frac{C_0}{4}$ to both sides establishes the global weighted bound
\[
 \norm{e^{\varphi/h}\phi_h^\circ}_{H^1_{A,h}(\R^2)} \le C_1,
 \qquad
 \norm{u}_{H^1_{A,h}}^2
 :=\norm{u}_{L^2}^2+\norm{\left(hP-\frac b2x^\perp\right)u}_{L^2}^2.
\]
Because $\varphi(x) \ge c_K > 0$ on $K$, applying standard semiclassical interior elliptic regularity estimates for $H_h^{\rm atom}(v^\circ) - E_h^\circ$ on a finite covering of $K$ upgrades this $H^1_{A,h}$ bound to the pointwise derivative bounds in \Cref{eq:absolute-agmon}. Uniformity for any closed bounded set at a positive distance from $\supp v^\circ$ follows identically by adjusting $U$ and $\varphi$.
\end{proof}

\subsection{Cusp geometry and action separation}
\label{ssec:geometry}

Now construct the full one-well potential by adding two exterior cusp-shaped perturbations. Fix $R>8r_0$ and set the cusp tips
\begin{equation}\label{eq:fixed-tips}
 p_+:=\left(\frac R2,\frac{\sqrt3R}{2}\right),
 \qquad
 p_-:=\left(\frac R2,-\frac{\sqrt3R}{2}\right),
 \qquad
 \cR(x_1,x_2)=(x_1,-x_2).
\end{equation}
Thus $|p_\pm|=R$ and $p_-=\cR p_+$.  Define the outgoing and tangential unit vectors, which determine the orientation of the cusps,
\begin{equation}\label{eq:fixed-frames}
 n_+:=\left(-\frac12,\frac{\sqrt3}{2}\right),
 \qquad
 \tau_+:=\left(-\frac{\sqrt3}{2},-\frac12\right),
 \qquad
 n_-:=\cR n_+,
 \qquad
 \tau_-:=\cR\tau_+.
\end{equation}

\begin{figure}[t]
\centering
\includegraphics[width=\textwidth]{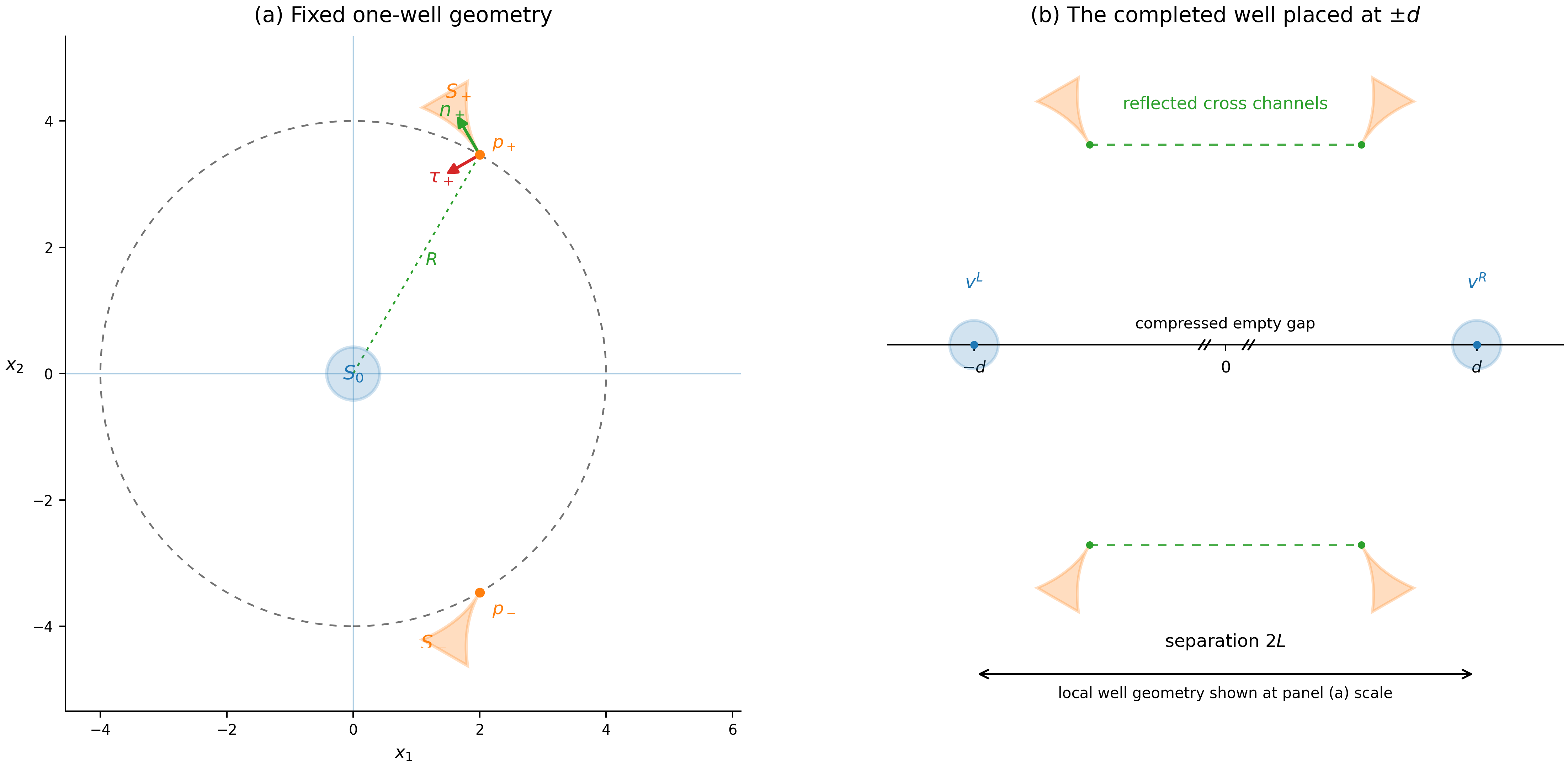}
\caption{Geometry of the construction.  The radial core and the two
one-sided cusps are fixed before the separation is chosen.  Their tips
are the unique rightmost points of the cusp supports, while the outgoing
normal directions have strictly positive radial components.  The completed one-well potential is then placed at the two centers $\pm d$, with separation $2L$.}
\label{fig:fixed-cusp-geometry}
\end{figure}
Note additionally that $p_\pm\cdot n_\pm=R/2$ and
$|\det(n_\pm,\tau_\pm)|=1$.  In particular, the directions $n_\pm$
point radially outward while their first components are negative.

\subsection{Double-well potential}
Let $d=(L,0)$.  We later impose stricter bound on $L > L_0(R)$, but in particular $2L>R$ ensures
\[
 p_++p_--2d=(R-2L,0)
\]
has negative first component.
To encode the cross-cusp channel that will later produce the leading hopping term, introduce its magnetic phase and corresponding total action by
\begin{equation}\label{eq:Phi}
 \Phi(z,w)=b\left[d\wedgep(z-w)+\frac12z\wedgep w\right]
\end{equation}
and
\begin{equation}\label{eq:A0-new}
 g_0(z)=\cJ_1(|z|),
 \qquad
 \cA_0(z,w)=g_0(z)+g_0(w)+\cJ_1(|z+w-2d|).
\end{equation} 
The normal action and phase slopes at the cross pair are
\begin{align}
m_R&:=\partial_{n_+}g_0(p_+),
\label{eq:mR_0}\\
 j&:=\partial_{n_+}\cJ_1(|z+w-2d|)\big|_{(p_+,p_-)},
      \label{eq:j-new_0}\\
\alpha&:=\partial_{n_+}\cA_0(p_+,p_-),
 \label{eq:alpha-new_0}\\
\theta&:=\partial_{n_+}\Phi(p_+,p_-).
 \label{eq:theta-new_0}
\end{align}
The reflected $w$ derivatives have the same values.  The active phase is
\begin{equation}\label{eq:Phi-star}
 \Phi_*:=\Phi(p_+,p_-).
\end{equation}
The
quadratic cusp introduced below makes every
tangential displacement $O(t^2)$, while the outgoing action grows linearly in
$t$.

\begin{proposition}[Geometric certificate]\label{prop:geometry-certificate}
The geometry \Cref{eq:fixed-tips}--\Cref{eq:fixed-frames} has the following
properties.
\begin{enumerate}[label=\textup{(\roman*)},leftmargin=2.5em]
\item The following identities
hold
\begin{align}
m_R&=\frac12\cJ_1'(R)>0,\label{eq:mR}\\
 j&=\frac12\cJ_1'(2L-R)>0,\label{eq:j-new}\\
 \alpha&=m_R+j,\label{eq:alpha-new}\\
 \theta&=\frac{\sqrt3}{2}bL>0.\label{eq:theta-new}
\end{align}
and
$\alpha-i\theta$ lies in the open right half-plane.
\item For every $\cE>0$, the same-cusp bridge gap at the tips satisfies
\begin{equation}\label{eq:same-tip-gap}
 \cJ_\cE\!\left(\sqrt{(2L-R)^2+3R^2}\right)-\cJ_\cE(2L-R)
 \ge \frac{3bR^2}{4}.
\end{equation}
\item Put
\[
 \delta_{\rm cp}:=\frac R2-r_0>0,
 \qquad
 \delta_{\rm cc}:=R-2r_0>0.
\]
The two quantities
\begin{align*}
 &\cJ_1(2L-R+\delta_{\rm cp})-\cJ_1(2L-R)-\cJ_1(R),\\
 &\cJ_1(2L-R+\delta_{\rm cc})-\cJ_1(2L-R)-2\cJ_1(R)
\end{align*}
are strictly increasing for $L>R/2$ and tend to $+\infty$ as
$L\to\infty$.
\end{enumerate}
\end{proposition}

\begin{proof}
The radial identities follow from
$p_+/R=(1/2,\sqrt3/2)$ and
$p_+\cdot n_+=R/2$.  Since the bridge vector at the cross pair is
$(R-2L,0)$, its gradient in the $z$ variable is
$-\cJ_1'(2L-R)e_1$, whose scalar product with $n_+$ is
$\cJ_1'(2L-R)/2$.  This proves \Cref{eq:mR}--\Cref{eq:alpha-new}.
Direct differentiation of \Cref{eq:Phi} gives
\[
 \nabla_z\Phi(p_+,p_-)
 =b\left(-\frac{\sqrt3R}{4},L-\frac R4\right),
\]
and its scalar product with $n_+$ is $\sqrt3bL/2$.  A direct wedge-product
calculation gives \Cref{eq:Phi-star}.

For $s>0$,
\[
 \frac{\dd}{\dd s}\cJ_\cE(\sqrt s)
 =\frac{1}{4}\sqrt{b^2+\frac{4\cE}{s}}
 \ge\frac b4.
\]
The two squared bridge lengths in \Cref{eq:same-tip-gap} differ by
$3R^2$, proving that inequality.

Finally, $\cJ_1'$ is strictly increasing.  Hence, for every fixed
$\delta>0$, the function
$D\mapsto\cJ_1(D+\delta)-\cJ_1(D)$ is strictly increasing.  Also
\[
 \cJ_1(D+\delta)-\cJ_1(D)
 \ge\frac b4\bigl((D+\delta)^2-D^2\bigr)
 \longrightarrow+\infty.
\]
This proves the last assertion.
\end{proof}

\subsection{The fixed log-flat cusp profile and the
\texorpdfstring{$h\log(1/h)$}{h log(1/h)} scale}\label{ssec:construction}

The cusp must be fixed independently of $h$, while its leading contribution
must concentrate near its tip.  The log-flat profile used below has this
property and separates the incoming channel from the exact scattered
response in the geometry considered here.
Fix $t_*>0$ and define
\begin{equation}\label{eq:log-flat-function}
 \ell(t):=\log\frac{t_*}{t},
 \qquad 0<t<t_*.
\end{equation}

\subsubsection{Log-flat balance}

\begin{lemma}[Log-flat minimum]\label{lem:flat-min}
Let $A,\beta>0$ and $0<h<1$.  The function
\[
 t\longmapsto \beta\ell(t)^2+\frac{At}{h}
\]
has a unique critical point in $(0,t_*)$, given for small $h$ by
\begin{equation}\label{eq:log-flat-saddle-real}
 t_{A,h}=\frac{2\beta h}{A}
 W_0\!\left(\frac{At_*}{2\beta h}\right),
\end{equation}
where $W_0$ is the principal Lambert function.  At this point the minimum is
\begin{equation}\label{eq:log-flat-minimum}
 \beta\left(W_0\!\left(\frac{At_*}{2\beta h}\right)^2
 +2W_0\!\left(\frac{At_*}{2\beta h}\right)\right).
\end{equation}
In particular,
\begin{equation}\label{eq:log-flat-scale}
 t_{A,h}\asymp_A h\log(1/h),
 \qquad
 \min\left(\beta\ell(t)^2+\frac{At}{h}\right)
 =\beta\log^2(1/h)+O_A\bigl(\log(1/h)\log\log(1/h)\bigr).
\end{equation}
\end{lemma}

\begin{proof}
The critical-point equation is
\[
 \frac{At}{h}=2\beta\ell(t).
\]
Writing $w=\ell(t)$ and $t=t_*e^{-w}$ gives
$we^w=At_*/(2\beta h)$, which proves
\Cref{eq:log-flat-saddle-real}.  Substitution gives
\Cref{eq:log-flat-minimum}.  The standard expansion
$W_0(x)=\log x-\log\log x+O(1)$ proves
\Cref{eq:log-flat-scale}.
\end{proof}

\begin{lemma}[Real log-flat integral bounds]\label{lem:log-flat-L1}
Fix $k\in\R$, $0<t_0<t_*$, and a compact interval
$I_A\Subset(0,\infty)$.  Let
$\chi\in C_c^\infty((-t_*/2,t_0);[0,1])$ be equal to one near $t=0$.
For every $\eta>0$ there is $h_\eta>0$ such that, uniformly
for $A\in I_A$ and $0<h<h_\eta$,
\begin{equation}\label{eq:log-flat-integral-bounds}
 e^{-(\beta+\eta)\log^2(1/h)}
 \le
 \int_0^{t_0}t^k\chi(t)
 e^{-\beta\ell(t)^2-At/h}\dd t
 \le
 e^{-(\beta-\eta)\log^2(1/h)}.
\end{equation}
The same upper bound holds after multiplication by any fixed power of
$\ell(t)$ or $t^{-1}$.
\end{lemma}

\begin{proof}
The flat factor makes the integral finite for every fixed $k$.  The minimum
in \Cref{lem:flat-min} and the expansion
\Cref{eq:log-flat-scale} give the upper bound after splitting into a fixed
neighborhood of $t_{A,h}$ and its complement.  For the lower bound, integrate
over an interval of length $ct_{A,h}/\sqrt{\log(1/h)}$ centered at
$t_{A,h}$; the second derivative of the exponent there is
$O((h^2\log(1/h))^{-1})$, so the exponent increases by only $O(1)$.
All powers of $t_{A,h}^{\pm1}$ and $\ell(t_{A,h})$ contribute only
$O(\log(1/h))$ to the logarithm and are absorbed by $\eta\log^2(1/h)$.
Uniformity follows from compactness of $I_A$.
\end{proof}

\subsubsection{Definition and admissibility of the exterior cusp components}

Choose $0<t_0<t_*$ and $s_0>0$.  Define
\begin{equation}\label{eq:chart}
 \Psi_+(t,s)=p_++tn_++t^2s\tau_+,
 \qquad
 \Psi_-(t,s)=\cR\Psi_+(t,s),
 \qquad 0< t<t_0,
 \quad |s|<s_0.
\end{equation}
Each map extends continuously to $t=0$, where the entire edge
$\{0\}\times(-s_0,s_0)$ collapses to the corresponding tip.  Thus
``embedding'' below always refers to the open strip $t>0$ and, in particular, the collapsed edge
belongs to the cusp closure, not to the embedding domain.

\begin{figure}[t]
    \centering
    \includegraphics[width=\linewidth]{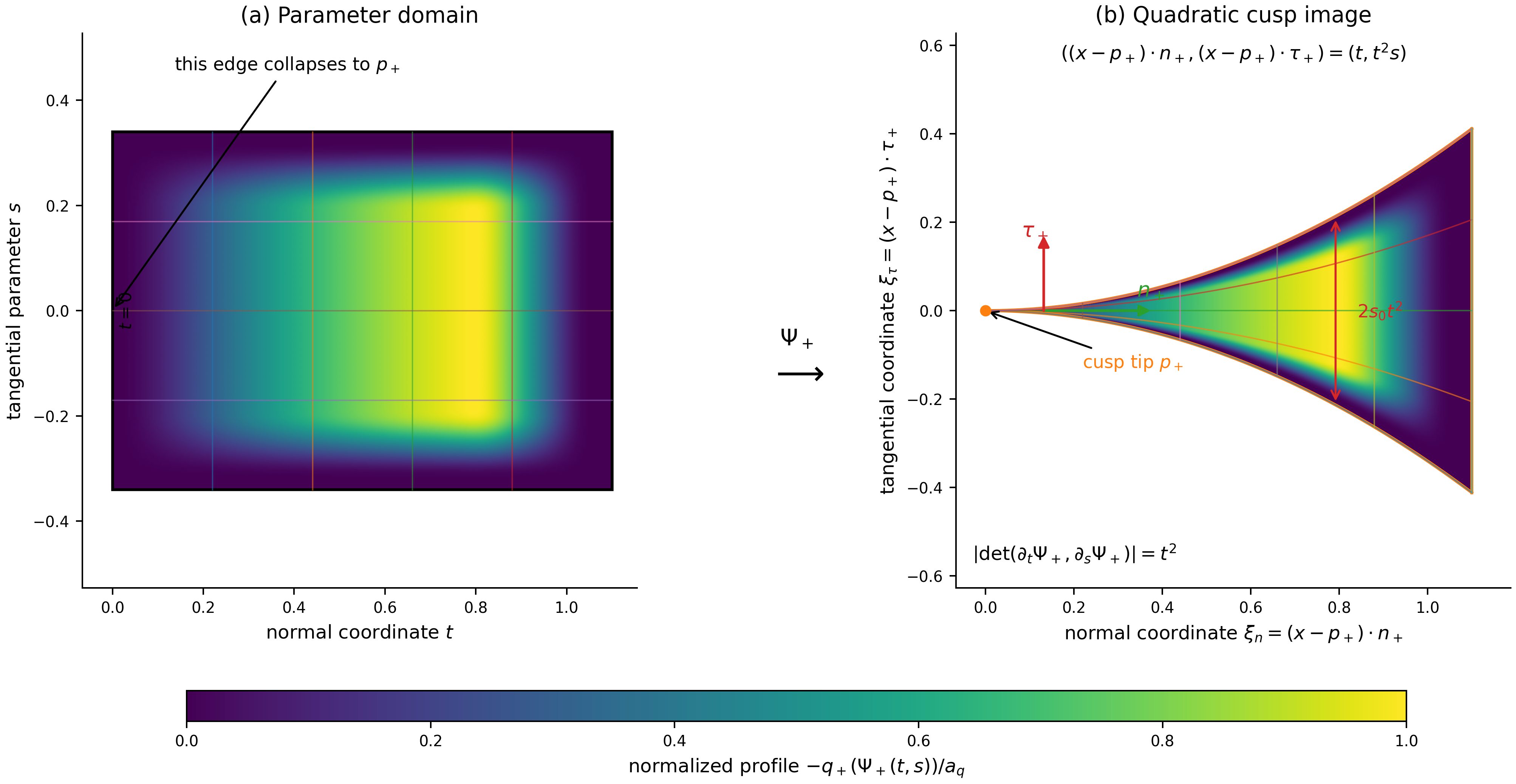}
    \caption{The quadratic cusp chart. The parameter domain $0<t<t_0$, $s<s_0$ is mapped by $\Psi_+(t,s)=p_++tn_++t^2s\tau_+$. The edge $t=0$ collapses to the cusp tip $p_+$, while a section at fixed $t$ has width $2s_0t^2$ in the $\tau_+$ direction. Consequently $|\det(\partial_t\Psi_+, \partial_s\Psi_+)|=t^2$. The shading displays a representative normalized log-flat cusp profile.}
    \label{fig:quadratic-cusp-chart}
\end{figure}

The Jacobian determinant satisfies 
\begin{equation}\label{eq:cusp-jacobian}
 \left|\det(\partial_t\Psi_\pm,\partial_s\Psi_\pm)\right|=t^2.
\end{equation}
Choose $\chi_a\in C_c^\infty((-t_*/2,t_0);[0,1])$ with
$\chi_a\equiv1$ on $[0,t_2]$ for some $0<t_2<t_0$, and choose an even
$\chi_b\in C_c^\infty((-s_0,s_0);[0,1])$, equal to one on
$[-s_0/2,s_0/2]$, with $\int\chi_b>0$.  For $a_q,\beta>0$, define
\begin{equation}\label{eq:cusp-def}
 q_+(\Psi_+(t,s))
 =-a_qe^{-\beta\ell(t)^2}\chi_a(t)\chi_b(s),
 \qquad t>0,
\end{equation}
set $q_+=0$ outside the chart image, and put $q_-=q_+\circ\cR$.  The final
one-well potential is
\begin{equation}\label{eq:final-v}
 v=v^\circ+\eps(q_++q_-).
\end{equation}
We impose
\begin{equation}\label{eq:epsilon-small}
 \eps a_q\le\delta_0<\frac14.
\end{equation}

\begin{lemma}[Smoothness, one-sidedness, and admissibility]
\label{lem:smooth-new}
After decreasing $s_0$ and $t_0$, the restrictions of $\Psi_\pm$ to
$(0,t_0)\times(-s_0,s_0)$ are embeddings with mutually disjoint images
disjoint from $\supp v^\circ$, and they extend continuously to their collapsed
tip edges.  The functions
$q_\pm$ belong to $C_c^\infty(\R^2)$ and are flat at their tips.  Moreover,
there are constants $c_x,c_r>0$ such that on either cusp support
\begin{align}
 (\Psi_\pm(t,s))_1&\le \frac R2-c_xt,\label{eq:one-sided-x}\\
 |\Psi_\pm(t,s)|&\ge R+c_rt.\label{eq:radial-outward}
\end{align}
On the upper cusp $(\Psi_+(t,s))_2\ge\sqrt3R/2$, and on the lower cusp
$(\Psi_-(t,s))_2\le-\sqrt3R/2$.  Consequently the potential
\Cref{eq:final-v} is smooth, compactly supported, nonpositive,
nonradial, invariant under $\cR$, and has the same unique nondegenerate
minimum as $v^\circ$.
\end{lemma}

\begin{proof}
In the orthonormal frame $(n_+,\tau_+)$, the chart coordinates of
$\Psi_+(t,s)-p_+$ are $(t,t^2s)$.  On $t>0$ the map
$(t,s)\mapsto(t,t^2s)$ has inverse $(t,\sigma)\mapsto(t,\sigma/t^2)$, hence
$\Psi_+$ is an embedding there, and \Cref{eq:cusp-jacobian} follows.  The
reflected chart has the same absolute Jacobian.  The continuous extension to
$t=0$ and collapse of the edge follow directly from the formula.  Since the
tips and the core are separated, shortening the charts makes the open images
disjoint.

Every Cartesian derivative of
$e^{-\beta\ell(t)^2}\chi_b(s)$, after it is expressed through the inverse
chart, is bounded by
\[
 C_Nt^{-N}(1+\ell(t))^N e^{-\beta\ell(t)^2}.
\]
This tends to zero faster than every power of $t$.  The extension by zero is
therefore smooth and flat at the tip.  The compact supports of the cutoffs
handle the other chart boundaries.

Using \Cref{eq:fixed-frames},
\[
 (\Psi_+(t,s))_1-\frac R2
 =-\frac t2-\frac{\sqrt3}{2}t^2s.
\]
For $\sqrt3s_0t_0\le1/2$, this is at most $-t/4$.  Moreover, orthonormality of $(n_+,\tau_+)$ and the fixed scalar products
with $p_+$ give the exact identity
\[
 |\Psi_+(t,s)|^2
 =R^2+Rt+(1-\sqrt3Rs)t^2+s^2t^4.
\]
If $s\le0$ the mixed correction is nonnegative.  If $s>0$, the condition
$\sqrt3s_0t_0\le1/2$ gives $\sqrt3st\le1/2$.  In both cases the right side is
at least $(R+t/4)^2$, and hence $|\Psi_+(t,s)|\ge R+t/4$.  Finally,
\[
 (\Psi_+(t,s))_2-\frac{\sqrt3R}{2}
 =\frac{\sqrt3}{2}t-\frac12t^2s\ge0
\]
when $s_0t_0\le\sqrt3$.  Reflection gives the lower-cusp statements.  The two packets are smooth,
nonpositive, nonzero, reflection-related, and supported away from the core and
from one another.  Their depth is at most $\delta_0<1/4$, whereas the unique
core minimum has value $-1$ and is unchanged in a neighborhood.  Thus no new
global minimum is introduced, and the packet geometry makes the sum
nonradial.  The remaining assertions about $v$ follow from the same support
separation and \Cref{eq:final-v}.
\end{proof}

\subsection{Outgoing-cusp control}\label{sec:outgoing-cusp}
To condense notation, we introduce
\[
g_h(x) := \cJ_{\cE^\circ_h}(|x|)
\]
\begin{lemma}[Outgoing-cusp inequality]\label{lem:outgoing-cusp}
After decreasing $s_0$ and $t_0$ once more, there are constants
$0<a_-<a_+<\infty$, independent of $h$, such that
\begin{align}
 g_h(\Psi_\pm(t,s))&\ge g_h(p_\pm)+a_-t,\label{eq:cusp-growth-master}\\
 g_h(\Psi_\pm(t,s))&\le g_h(p_\pm)+a_+t,
 \label{eq:cusp-growth-upper}
\end{align}
for $0<t<t_0$, $|s|\le s_0$, and all sufficiently small $h$, where
$g_h(x)=\cJ_{-E_h^\circ}(|x|)$.  The constant $a_-$ may be chosen
arbitrarily close from below to $m_R$ in \Cref{eq:mR}.
\end{lemma}

\begin{proof}
At the tip,
\[
 \nabla g_0(p_\pm)\cdot n_\pm=m_R>0.
\]
The derivative of the chart is
$\partial_t\Psi_\pm=n_\pm+2ts\tau_\pm$.  Since
$-E_h^\circ\to1$, the gradients $\nabla g_h$ converge to $\nabla g_0$ in
$C^1$ on fixed neighborhoods of the tips.  Uniform continuity therefore
gives positive upper and lower bounds for
$\nabla g_h(\Psi_\pm)\cdot\partial_t\Psi_\pm$ after $t_0$ is chosen small.
Integration in $t$ proves the result.
\end{proof}

For later use, set
\begin{equation}\label{eq:component-supports}
 S_0:=\supp v^\circ,
 \qquad
 S_+:=\supp q_+,
 \qquad
 S_-:=\supp q_-.
\end{equation}
These three compact sets are pairwise disjoint.  The potential and its entire
one-well geometry are now fixed.

\paragraph{Parameter hierarchy and separation threshold.}
First fix $v^\circ$, $b$, and $r_0$, and choose $R>8r_0$.  Next fix the
tips and frames in \Cref{eq:fixed-tips}--\Cref{eq:fixed-frames}; choose the
cusp widths, cutoffs, and cusp parameters $a_q,\beta,\delta_0,\eps$ so that
\Cref{lem:smooth-new}, \Cref{lem:outgoing-cusp}, and
\Cref{eq:epsilon-small} hold.  At this point the single-well potential
$v$ and every parameter entering its definition are fixed, independently of
both $L$ and $h$.  Put
\begin{equation}\label{eq:support-radius-a}
 a:=\sup\{|x|:x\in\supp v\}<\infty
\end{equation}
and
\begin{equation}\label{eq:delta-hop}
 \delta_{\rm hop}:=\frac{bR^2}{64}.
\end{equation}
By \Cref{prop:geometry-certificate}, the same-cusp point gap is
at least $48\delta_{\rm hop}$, while the core--cusp and core--core reserves
tend to infinity with $L$.  Choose $L_0>\max\{a,R\}$ so that both reserves at
$(\cE,\cE^\circ)=(1,1)$ are at least $32\delta_{\rm hop}$ when $L=L_0$.
For $\delta\in\{\delta_{\rm cp},\delta_{\rm cc}\}$ and positive energy
parameters, set
\[
 \Delta_\delta(L;\cE,\cE^\circ)
 :=\cJ_\cE(2L-R+\delta)-\cJ_\cE(2L-R)
   -m_\delta\cJ_{\cE^\circ}(R),
 \qquad
 m_{\delta_{\rm cp}}=1,
 \quad m_{\delta_{\rm cc}}=2.
\]
For every fixed $(\cE,\cE^\circ)$,
\[
 \partial_L\Delta_\delta
 =2\left[
 \cJ_\cE'(2L-R+\delta)-\cJ_\cE'(2L-R)
 \right]>0.
\]
Continuity at $(L_0,1,1)$ gives a compact interval
$I_\cE\Subset(0,\infty)$ containing $1$ such that both reserves are at least
$24\delta_{\rm hop}$ for $L=L_0$ and
$(\cE,\cE^\circ)\in I_\cE^2$; monotonicity propagates the same bounds to
every $L\ge L_0$.  No parameter of the potential is changed in this final
choice.

\paragraph{Standing convention.}
Fix henceforth any $L\ge L_0$ and set $d=(L,0)$.  All constants and small-$h$
thresholds may depend on this fixed choice unless explicitly stated
otherwise.  The potential, cusp geometry, and cusp profile remain the
same for every admissible $L$.

\section{The exact one-well source and control on the cusp supports}\label{sec:one-well-source}

In this section we determine the exact normalized one-well ground state $\phi_h$ of the one-well Hamiltonian $H_h^{\rm atom}(v)$
on the supports of the two cusp components $q_+$ and $q_-$.  

The analysis begins by comparing the exact one-well ground state with the radial reference state through a Feshbach decomposition. Complement coercivity gives a unique low eigenvalue, an exponentially small energy shift, and an order-$h$ gap above the exact ground state. We then decompose the cusp source into the incoming radial contribution and the response generated by the cusp perturbation. Control of the latter is derived by a weighted inverse and local elliptic propagation, yielding the source estimates needed for the analysis of \Cref{sec:active-integral}.

\subsection{The exact final one-well ground state}\label{sec:one-well-feshbach}
Let 
\[
W:=\eps(q_++q_-)
\]
denote the fixed exterior cusp contribution of \Cref{eq:final-v}, so that $v=v^\circ+W$. By
\Cref{eq:cusp-def} and \Cref{eq:epsilon-small}, $W\in C_c^\infty(\R^2;[-\delta_0,0])$
is supported on the two disjoint cusp supports and vanishes on a neighborhood of the
core support.  

For ease of notation set
\[
 H_h^{\rm atom}:=H_h^{\rm atom}(v^\circ+W).
\]
Let $(E_h,\phi_h)$ be a normalized ground-state pair of $H_h^{\rm atom}$. Additionally, this ground state will be shown to be simple. Thus
\begin{equation}\label{eq:final-operator}
 H_h^{\rm atom}\phi_h=E_h\phi_h,
 \qquad \norm{\phi_h}_{L^2}=1.
\end{equation}
Choose the constant phase so that
\begin{equation}\label{eq:positive-overlap}
 c_h:=\ip{\phi_h^\circ}{\phi_h}\geq0.
\end{equation}
Additionally, let
\begin{equation}\label{eq:PQ}
 \cP_h=\phi_h^\circ\otimes(\phi_h^\circ)^*,
 \qquad \cQ_h=1-\cP_h.
\end{equation}
Finally, set
\begin{equation}\label{eq:phi-decomp}
 \eta_h:=\cQ_h\phi_h,
 \qquad
 \phi_h=c_h\phi_h^\circ+\eta_h,
 \qquad \eta_h\in\Ran\cQ_h.
\end{equation}

We aim to derive the exact formula for $\eta_h$, but first need the following lemma. 

\begin{lemma}[Complement coercivity uniform in the fixed cusps]\label{lem:complement-coercivity}
After first choosing $\delta_0$ in \Cref{eq:epsilon-small} sufficiently small and then $h_0$ sufficiently small, there is $c>0$ such that
\begin{equation}\label{eq:compressed-gap}
 \ip{u}{\cQ_h(H_h^{\rm atom}-E_h^\circ)\cQ_hu}
 \ge ch\norm{u}^2,
 \qquad u\in\Ran\cQ_h.
\end{equation}
The same bound, with $c/2$, holds with $E_h^\circ$ replaced by any $E$ satisfying $|E-E_h^\circ|\le ch/2$.
\end{lemma}

\begin{proof}
Choose a fixed IMS partition $\chi_0^2+\chi_1^2=1$ such that $\chi_0=1$ 
on a neighborhood of $\supp v^\circ$, $\chi_1=1$ on the two cusp
supports, and the transition region lies where $v^\circ=W=0$.  For
$u\in\Ran\cQ_h$, the magnetic IMS formula gives
\[
 \ip{u}{(H_h^{\rm atom}-E_h^\circ)u}
 =\sum_{j=0}^1\ip{\chi_ju}{(H_h^{\rm atom}-E_h^\circ)\chi_ju}
 -h^2\sum_{j=0}^1\norm{|\nabla\chi_j|u}^2.
\]
On the outer piece,
$H_h^{\rm atom}-E_h^\circ\ge-\delta_0-E_h^\circ\ge1/2$ for small $h$.  On the inner
piece $W=0$.  Since $u\perp\phi_h^\circ$,
\[
 \ip{\phi_h^\circ}{\chi_0u}
 =-\ip{(1-\chi_0)\phi_h^\circ}{u},
\]
and \Cref{lem:radial-agmon} makes the right side exponentially small.
The radial gap \Cref{eq:base-gap} therefore yields
\[
 \ip{\chi_0u}{(H_h^{\rm atom}(v^\circ)-E_h^\circ)\chi_0u}
 \ge \frac{c_{\rm gap}}2h\norm{\chi_0u}^2
 -Ce^{-c/h}\norm u^2.
\]
The IMS error is $O(h^2)\norm u^2$ and is absorbed by the order-$h$ inner
gap and the order-one outer bound.  This proves \Cref{eq:compressed-gap}.
Replacing $E_h^\circ$ by $E$ changes the form by at most $ch/2$ and gives
the final assertion after decreasing $c$.
\end{proof}

\begin{proposition}[Exact Feshbach equations]\label{prop:Feshbach}
For small $h$, the spectral projection of $H_h^{\rm atom}$ onto
$(-\infty,E_h^\circ+ch/2]$ has rank one.  Its unique eigenvalue $E_h$ is the
ground energy and lies in $[E_h^\circ-ch/4,E_h^\circ+ch/4]$.  Moreover, $\eta_h$ satisfies the exact formula
\begin{equation}\label{eq:eta-Feshbach}
 \eta_h
 =-c_h\Bigl[\cQ_h(H_h^{\rm atom}-E_h)\cQ_h\big|_{\Ran\cQ_h}\Bigr]^{-1}
 \cQ_hW\phi_h^\circ.
\end{equation}
Here and below, every compressed inverse is taken on $\Ran\cQ_h$.
Moreover,
\begin{align}
 \norm{\eta_h}
 &\le Ch^{-1}\norm{W\phi_h^\circ},\label{eq:eta-L2}\\
 |E_h-E_h^\circ|
 &\le C\left(\norm{W\phi_h^\circ}^2h^{-1}
 +\abs{\ip{\phi_h^\circ}{W\phi_h^\circ}}\right),\label{eq:energy-shift}\\
 |c_h-1|&\le C\norm{\eta_h}^2.
 \label{eq:c-close}
\end{align}
\end{proposition}
\begin{proof}
For $u=a\phi_h^\circ+\eta$, $\eta\perp\phi_h^\circ$, \Cref{lem:complement-coercivity} and completion of the square give

$$
 \ip{u}{(H_h^{\rm atom}-E_h^\circ)u}
 \ge -|a|^2\abs{\ip{\phi_h^\circ}{W\phi_h^\circ}}
 -Ch^{-1}|a|^2\norm{W\phi_h^\circ}^2
 +\frac c2h\norm{\eta}^2.
$$

By \Cref{lem:radial-agmon}, the first two terms on the right-hand side are $\cO(e^{-c/h})$ up to powers of $h^{-1}$. Together with the Rayleigh upper bound obtained from $\phi_h^\circ$, this places the ground energy inside the stated window.

To count all spectrum below the upper endpoint, let $V$ be any two-dimensional subspace. Since $\Ran\cQ_h$ has codimension one, $V$ contains a nonzero vector $u\in\Ran\cQ_h$. \Cref{lem:complement-coercivity} gives

$$
 \frac{\ip{u}{H_h^{\rm atom}u}}{\|u\|^2}\ge E_h^\circ+ch.
$$

The min--max principle therefore implies that the spectral projection of $H_h^{\rm atom}$ below $E_h^\circ+ch/2$ has rank at most one, while the Rayleigh bound gives rank at least one. Thus there is exactly one spectral point below that threshold, namely the simple global ground energy $E_h$.

Projecting the eigenvalue equation onto $\Ran\cQ_h$ gives \Cref{eq:eta-Feshbach}; the inverse exists by \Cref{lem:complement-coercivity}. \Cref{eq:eta-L2} and \Cref{eq:energy-shift} follow by taking norms and then projecting onto $\phi_h^\circ$. Finally,

$$
 1=c_h^2+\norm{\eta_h}^2
$$

gives \Cref{eq:c-close}. Since $\norm{W\phi_h^\circ}=\cO(h^{-M}e^{-c/h})$, \Cref{eq:eta-L2} and \Cref{eq:c-close} also give $c_h\to1$ and, after shrinking $h_0$, $c_h\ge\tfrac12$.
\end{proof}

\begin{lemma}[Order-$h$ gap above the exact one-well ground state]
\label{lem:exact-one-well-gap}
There is $g_*>0$ such that, for all small $h$, the exact final one-well
Hamiltonian $H_h^{\rm atom}=H_h^{\rm atom}(v)$ satisfies
\begin{equation}\label{eq:exact-one-well-gap}
 \dist\!\left(E_h,\sigma(H_h^{\rm atom}(v))\setminus\{E_h\}\right)
 \ge g_*h.
\end{equation}
\end{lemma}

\begin{proof}
By \Cref{prop:Feshbach}, $H_h^{\rm atom}$ has exactly one spectral point
$E_h$ below $E_h^\circ+ch/2$, and $E_h$ lies in
\[
 [E_h^\circ-ch/4,E_h^\circ+ch/4].
\]
Moreover, the cusp support lies a fixed positive distance from the core.
\Cref{lem:radial-agmon} and \Cref{prop:Feshbach} therefore give
\[
 |E_h-E_h^\circ|
 \le Ch^{-1}\|W\phi_h^\circ\|^2
 +|\ip{\phi_h^\circ}{W\phi_h^\circ}|
 =\cO(e^{-c_1/h}).
\]
Every other point of $\sigma(H_h^{\rm atom})$ is at least $E_h^\circ+ch/2$, by
the rank-one conclusion of \Cref{prop:Feshbach}.  Hence
\[
 \dist\left(E_h,\sigma(H_h^{\rm atom})\setminus\{E_h\}\right)
 \ge \frac c2h-\cO(e^{-c_1/h})
 \ge \frac c4h.
\]
Taking $g_*=c/4$ proves \Cref{eq:exact-one-well-gap}.
\end{proof}

Since $E_h^\circ=-1+\cO(h)$ and $|E_h-E_h^\circ|$ is exponentially small, we also have
\begin{equation}\label{eq:final-energy-scale}
 E_h=-1+\cO(h).
\end{equation}
For later use, the final sentence of the proof gives
\begin{equation}\label{eq:ch-lower}
 c_h=1+o(1),
 \qquad \frac12\le c_h\le1
\end{equation}
for all sufficiently small $h$.

\subsection{Global source bounds on the cusp supports}\label{sec:global-source}
We now pass from the eigenfunction itself to the source that enters the exact Landau-resolvent representation of the hopping coefficient. Several weighted estimates can be found in \Cref{app:control-scat-response}. For the final one-well ground pair $(E_h,\phi_h)$, define the exact unscaled source
\begin{equation}\label{eq:source-exact-master}
 F_\lambda=h^{-2}v\phi_h,
 \qquad
 (H_{h}^{\mathrm{Lan}}-E_h)\phi_h=-v\phi_h.
\end{equation}
This source will appear exactly in the hopping coefficient $\rho_\lambda$ in \Cref{sec:exact-hopping}. For $\alpha\in\{0,+,-\}$, set
\[
 F_\alpha:=\mathbf 1_{S_\alpha}F_\lambda,
 \qquad
 F_\lambda=F_0+F_++F_-.
\]

Recall $\cE_h^\circ=-E_h^\circ$ and
\begin{equation}\label{eq:radial-tail-again}
 \phi_h^\circ(z)=\Gamma_hK_h^{\cE_h^\circ}(z)
 \qquad (z\notin B_{r_0}(0)).
\end{equation}
Now split the cusp source using the Feshbach decomposition of $\phi_h$ from \Cref{prop:Feshbach},
\begin{equation}\label{eq:source-split}
 F_\lambda^{\rm in}
 :=h^{-2}\eps(q_++q_-)c_h\phi_h^\circ,
 \qquad
 F_\lambda^{\rm sc}
 :=h^{-2}\eps(q_++q_-)\eta_h.
\end{equation}
The superscripts refer to the incoming radial tail from the radial core and the exact scattered
response induced by the cusp perturbations.  Set
\begin{equation}\label{eq:ell-h}
 \ell_h:=\log(1/h),
\end{equation}
and fix a number $\beta_0$ such that
\begin{equation}\label{eq:beta0-choice}
 \frac{2\beta}{3}<\beta_0<\beta.
\end{equation}
The strict first inequality is used in \Cref{sec:active-integral}:
three log-flat costs with coefficient $\beta_0$ dominate the two active costs
with coefficient $\beta$.

Since both cusps lie at tip radius $R$, denote their common radial action by
\begin{equation}
    G_{p, h} := g_h(p_+) = g_h(p_-) = \cJ_{\cE^\circ_h}(R).
\end{equation}

For the scattered response, let
\[
 \overline{S_+\cup S_-}\subset U'_{\rm pkt}\Subset U_{\rm pkt}
\]
and let $T$ be the cusp-adapted weight furnished by \Cref{lem:cusp-weight} in \Cref{app:control-scat-response}; in particular, $T(\Psi_\pm(t,s))=t$ on the cusp supports. Fix also $\varkappa>0$ as in \Cref{eq:kappa-weight}. The same appendix gives the weighted compressed inverse of \Cref{lem:weighted-inverse} and the weighted interior elliptic propagation of \Cref{lem:weighted-interior-elliptic-propagation} used below.

\begin{theorem}[Exact incoming and scattered bounds on the cusp supports]
\label{thm:source-new}
After the cusp charts and the weight $T$ have been fixed, there are
$C,M,h_0>0$ such that, on either cusp and for $0<h<h_0$,
\begin{align}
 |F_\lambda^{\rm in}(\Psi_\pm(t,s))|
 &\le Cc_h\Gamma_hh^{-7/2}e^{-G_{p,h}/h}
 \exp\!\left[-\beta\ell(t)^2-\frac{a_-t}{h}\right],
 \label{eq:source-in-bound}\\
 |F_\lambda^{\rm sc}(\Psi_\pm(t,s))|
 &\le Cc_h\Gamma_hh^{-M}e^{-G_{p,h}/h}
 e^{-\beta_0\ell_h^2}
 \exp\!\left[-\beta_0\ell(t)^2-\frac{\varkappa t}{h}\right].
 \label{eq:source-sc-bound}
\end{align}
More precisely, for every multi-index $\gamma$ and every
$\beta_{\rm in}<\beta$ there are $C_\gamma,M_\gamma>0$ such that
\begin{align}
 |(h\nabla)^\gamma F_\lambda^{\rm in}(\Psi_\pm(t,s))|
 &\le C_\gamma c_h\Gamma_hh^{-M_\gamma}e^{-G_{p,h}/h}
 e^{-\beta_{\rm in}\ell(t)^2-a_-t/h},
 \label{eq:source-in-derivative-bound}\\
 |(h\nabla)^\gamma F_\lambda^{\rm sc}(\Psi_\pm(t,s))|
 &\le C_\gamma c_h\Gamma_hh^{-M_\gamma}e^{-G_{p,h}/h}
 e^{-\beta_0\ell_h^2}
 e^{-\beta_0\ell(t)^2-\varkappa t/h}.
 \label{eq:source-sc-derivative-bound}
\end{align}
\end{theorem}

\begin{proof}
On a cusp, \Cref{prop:K-Laplace},
\Cref{eq:radial-tail-again}, and \Cref{lem:outgoing-cusp} give
\[
 |\phi_h^\circ(\Psi_\pm(t,s))|
 \le C\Gamma_hh^{-3/2}e^{-G_{p,h}/h}e^{-a_-t/h}.
\]
Multiplication by $h^{-2}\eps|q_\pm|$ proves
\Cref{eq:source-in-bound}.

We next estimate the forcing in the exact Feshbach equation.  Since
$T(\Psi_\pm(t,s))=t$, the cusp Jacobian is $t^2$, and
$a_--\varkappa>0$, one has
\begin{align}
 \norm{e^{\varkappa T/h}W\phi_h^\circ}^2
 &\le C\Gamma_h^2h^{-3}e^{-2G_{p,h}/h}
 \int_0^{t_0}t^2
 e^{-2\beta\ell(t)^2-2(a_--\varkappa)t/h}\dd t.
 \label{eq:weighted-source-integral}
\end{align}
Squaring the pointwise envelope replaces $(\beta,a_--\varkappa)$ by
$(2\beta,2(a_--\varkappa))$ exactly.  Apply \Cref{lem:log-flat-L1} with
a loss chosen so that the resulting squared-log coefficient is $2\beta_0$,
and only then take the square root.  The strict inequality
$\beta_0<\beta$ therefore gives
\begin{equation}\label{eq:weighted-source-norm}
 \norm{e^{\varkappa T/h}W\phi_h^\circ}
 \le C\Gamma_hh^{-M}e^{-G_{p,h}/h}e^{-\beta_0\ell_h^2}.
\end{equation}
The rank-one projection in \Cref{eq:eta-Feshbach} does not alter this scale:
by property \textup{(T4)} of \Cref{lem:cusp-weight},
\[
 \norm{e^{\varkappa T/h}\cQ_hW\phi_h^\circ}
 \le C\norm{e^{\varkappa T/h}W\phi_h^\circ}.
\]
The weighted inverse, \Cref{lem:weighted-inverse}, therefore gives
\begin{equation}\label{eq:eta-weighted-L2}
 \norm{e^{\varkappa T/h}\eta_h}
 \le Cc_h\Gamma_hh^{-M}e^{-G_{p,h}/h}e^{-\beta_0\ell_h^2}.
\end{equation}

For pointwise control, use the exact equation
\begin{equation}\label{eq:eta-local-equation}
 (H_h^{\rm atom}-E_h)\eta_h
 =-c_hW\phi_h^\circ+c_h(E_h-E_h^\circ)\phi_h^\circ
 =:f_h.
\end{equation}
We claim that, for every fixed integer $k$,
\begin{equation}\label{eq:eta-forcing-derivatives}
 \sum_{|\gamma|\le k}
 \|e^{\varkappa T/h}(h\nabla)^\gamma f_h\|_{L^2(U_{\rm pkt})}
 \le C_kc_h\Gamma_hh^{-M_k}e^{-G_{p,h}/h}
 e^{-\beta_0\ell_h^2}.
\end{equation}
Indeed, the differentiated radial expansion
\Cref{eq:K-expansion}, the identity $T=t$ on the cusp supports, and
\[
 t^{-N}(1+\ell(t))^Ne^{-\beta\ell(t)^2}
 \le C_{N,\beta_1}e^{-\beta_1\ell(t)^2}
 \qquad(\beta_0<\beta_1<\beta)
\]
reduce every derivative of $W\phi_h^\circ$ to the same real log-flat
integral as in \Cref{eq:weighted-source-integral}.  
\Cref{lem:log-flat-L1} then gives the right side of
\Cref{eq:eta-forcing-derivatives}.  For the second term in
\Cref{eq:eta-local-equation}, \Cref{eq:energy-shift} and
\Cref{eq:weighted-source-norm} give
\[
 |E_h-E_h^\circ|
 \le C\Gamma_hh^{-M}e^{-G_{p,h}/h}e^{-\beta_0\ell_h^2};
\]
property \textup{(T4)}, 
together with semiclassical interior elliptic
estimates on the fixed forbidden set $U_{\rm pkt}$ and Lemma~\ref{lem:radial-agmon} controls all terms of the form
of $e^{\varkappa T/h}(h\nabla) ^{\gamma}\phi_h^\circ$ by a polynomial power of $h^{-1}$.
This proves \Cref{eq:eta-forcing-derivatives}.

Next apply \Cref{lem:weighted-interior-elliptic-propagation} to
$U'_{\rm pkt}\Subset U_{\rm pkt}$, using
\Cref{eq:eta-weighted-L2} and \Cref{eq:eta-forcing-derivatives}.  Since
$U'_{\rm pkt}$ contains the closed cusp supports, including their tips, we
obtain, for every fixed multi-index $\alpha$,
\begin{equation}\label{eq:eta-pointwise}
 e^{\varkappa t/h}
 |(h\nabla)^\alpha\eta_h(\Psi_\pm(t,s))|
 \le C_\alpha c_h\Gamma_hh^{-M_\alpha}
 e^{-G_{p,h}/h}e^{-\beta_0\ell_h^2}.
\end{equation}
Multiplying by $h^{-2}\eps|q_\pm|$ proves
\Cref{eq:source-sc-bound}.  The derivative statement follows from the same
flat-factor inequality above: derivatives of the incoming source factor may use any
coefficient below $\beta$, while the strict inequality $\beta_0<\beta$
allows all derivatives of the scattered source multiplier to retain the
coefficient $\beta_0$.
\end{proof}

We have therefore established the two one-well inputs needed below. An order-$h$ spectral gap for the exact one-well Hamiltonian and sharp incoming/scattered bounds for its cusp source. The former enters the Schur reduction of \Cref{sec:schur-reduction}, while the latter enters the source-cell analysis of \Cref{sec:exact-hopping} and~\Cref{sec:active-integral}.

\section{Parity Schur reduction}\label{sec:schur-reduction}
In this section, we show that the low-energy spectral subspace of the double-well Hamiltonian is controlled by the two translated one-well ground states, $\phi_h^L$ and $\phi_h^R$. We first establish exponential localization and an order-$h$ spectral gap on their orthogonal complement. We then pass to the even and odd combinations and apply a Schur/Feshbach reduction, obtaining one simple low eigenvalue in each parity sector and an exact scalar equation for each.

\subsection{Scaled double-well algebra}\label{sec:scaled-double-well}
The original double-well problem is written in the unscaled energy normalization.  For the
spectral reduction it is convenient to divide the operator by $\lambda^2$.  Let
\begin{equation}\label{eq:magnetic-translation}
 (\mathsf T_a^h f)(x)
 :=\exp\!\left(-\frac{ib}{2h}x\wedgep a\right)f(x-a),
 \qquad a\in\R^2,
\end{equation}
where $\mathsf T_a^h$ denotes magnetic translation by $a$. Denote $(\mathsf Pf)(x)=f(-x)$.  The magnetic translations commute with $H_{h}^{\mathrm{Lan}}$,
and $\mathsf P H_{h}^{\mathrm{Lan}}=H_{h}^{\mathrm{Lan}}\mathsf P$.  Then we can define the translated wells and ground states by
\begin{align}
 v^L(x)&:=v(x+d),& v^R(x)&:=v(-x+d),\label{eq:translated-potentials}\\
 \phi_h^L&:=\mathsf T_{-d}\phi_h,&
 \phi_h^R&:=\mathsf P\phi_h^L.\label{eq:translated-states}
\end{align}
This implies
\begin{equation}\label{eq:one-well-equations}
 (H_{h}^{\mathrm{Lan}}+v^L)\phi_h^L=E_h\phi_h^L,
 \qquad
 (H_{h}^{\mathrm{Lan}}+v^R)\phi_h^R=E_h\phi_h^R,
 \qquad
 \norm{\phi_h^L}=\norm{\phi_h^R}=1.
\end{equation}
The scaled double-well operator can then be written
\begin{equation}\label{eq:Hh}
 H_h:=H_{h}^{\mathrm{Lan}}+v^L+v^R,
\end{equation}
and the original operator is $H_v(\lambda)=h^{-2}H_h$.  The two single-well supports are disjoint because $L>a$ in the standing convention after \Cref{eq:support-radius-a}.
Moreover $[H_h,\mathsf P]=0$.

Set
\begin{align}
 s_h&:=\ip{\phi_h^L}{\phi_h^R},\label{eq:overlap-def}\\
 \widehat\delta_h
 &:=\ip{\phi_h^L}{(H_h-E_h)\phi_h^L}
 =\ip{\phi_h^L}{v^R\phi_h^L},\label{eq:delta-scaled}\\
 \widehat\rho_h
 &:=\ip{\phi_h^L}{(H_h-E_h)\phi_h^R}
 =\ip{\phi_h^L}{v^L\phi_h^R}
 =:h^2\rho_\lambda.\label{eq:rho-scaled}
\end{align}
All three quantities are real. In fact, $s_h=\ip{\phi_h^L}{\mathsf P\phi_h^L}$,
while $\mathsf P$ preserves $\mathcal D(H_h)$ and commutes with $H_h$, so
$(H_h-E_h)\mathsf P$ is self-adjoint on $\mathcal D(H_h)$. By reflection,
the right state has the same diagonal defect as the left. We also denote
\begin{equation}\label{eq:Vspan}
 \mathcal V_h:=\operatorname{span}\{\phi_h^L,\phi_h^R\}.
\end{equation}

\subsection{Two-well complement coercivity}
\label{sec:complement-coercivity}

Choose fixed smooth functions $\chi_L,\chi_R,\chi_0$ such that
\begin{equation}\label{eq:IMS-partition}
 \chi_L^2+\chi_R^2+\chi_0^2=1,
\end{equation}
which will form the basis for an  Ismagilov-Morgan-Simon (IMS) decomposition. 

Additionally choose the functions such that $\chi_L$ equals one on a closed neighborhood of $\supp v^L$,
and $\chi_R$ equals one on a closed neighborhood of $\supp v^R$; these
neighborhoods and the supports of $\chi_L,\chi_R$ have disjoint closures.
The transition regions are a fixed positive distance from the corresponding
potential supports, and $v^L=v^R=0$ on $\supp\chi_0$.  Such a partition exists because the shifted supports are disjoint.

\begin{lemma}[Exponential one-well tail outside the IMS core]
\label{lem:IMS-tail}
There are $c,C>0$ such that
\begin{equation}\label{eq:IMS-tail-bound}
 \norm{(1-\chi_L)\phi_h^L}
 +\norm{(1-\chi_R)\phi_h^R}
 \le Ce^{-c/h}.
\end{equation}
\end{lemma}

\begin{proof}
Each cutoff equals one on a fixed neighborhood of the full corresponding one-well
support.  Outside that neighborhood the potential vanishes and hence
\[
H_{h}^{\mathrm{Lan}}-E_h\ge bh-E_h\ge1/2
\]
for small $h$.  

Choose a fixed Lipschitz
Agmon weight $\varphi\ge0$ which vanishes near the potential support, equals
a fixed positive constant on $\supp(1-\chi_{L,R})$, and satisfies
\[
|\nabla\varphi|^2\le1/4.  
\]
The magnetic Agmon identity then gives
\[
 \norm{e^{\varphi/h}\phi_h^{L,R}}_{H^1_{A,h}}
 \le Ch^{-M}.
\]
Since the distance from $\supp(1-\chi_{L,R})$ to the corresponding potential support
is positive, the polynomial factor is absorbed into a slightly smaller exponential.
\end{proof}

\begin{theorem}[Two-well complement gap]\label{thm:two-well-gap}
There are $c_{\rm dw},h_0>0$ such that
\begin{equation}\label{eq:Vperp-gap}
 \ip{u}{(H_h-E_h)u}
 \ge c_{\rm dw}h\norm{u}^2,
 \qquad
 u\in\mathcal V_h^\perp,
 \qquad 0<h<h_0.
\end{equation}

\end{theorem}

\begin{proof}
The exact one-well operator has the order-$h$ spectral gap of \Cref{lem:exact-one-well-gap}, which we now propagate to the two-well operator by an IMS decomposition.

On $\supp\chi_L$ the right potential vanishes, so
$H_h=H_{h}^{\mathrm{Lan}}+v^L$.  By \Cref{eq:exact-one-well-gap},
\begin{equation}\label{eq:left-gap-local}
 \ip{\chi_Lu}{(H_h-E_h)\chi_Lu}
 \ge g_*h\left(
 \norm{\chi_Lu}^2-|\ip{\phi_h^L}{\chi_Lu}|^2
 \right).
\end{equation}
Since $u\perp\phi_h^L$,
\[
 |\ip{\phi_h^L}{\chi_Lu}|
 =|\ip{(1-\chi_L)\phi_h^L}{u}|
 \le Ce^{-c/h}\norm{u}
\]
by \Cref{lem:IMS-tail}.  The same estimate holds on the right.
On $\supp\chi_0$ both potentials vanish, and the bottom Landau level gives
\begin{equation}\label{eq:exterior-gap}
 \ip{\chi_0u}{(H_h-E_h)\chi_0u}
 =\ip{\chi_0u}{(H_{h}^{\mathrm{Lan}}-E_h)\chi_0u}
 \ge(bh-E_h)\norm{\chi_0u}^2
 \ge\frac12\norm{\chi_0u}^2.
\end{equation}
The IMS error is $\cO(h^2)\norm{u}^2$.  Combining
\Cref{eq:left-gap-local}--\Cref{eq:exterior-gap} and using
\Cref{eq:IMS-partition} proves \Cref{eq:Vperp-gap} for small $h$.
\end{proof}

\subsection{Coarse localization of the two-well data}\label{sec:overlap-ledger}\label{sec:coarse-defect-residual}

\begin{theorem}[Overlap and off-diagonal localization]\label{thm:overlap-ledger}
There are $c,C>0$ such that, for all sufficiently small $h$,
\begin{equation}\label{eq:overlap-small}
 |s_h|+|\widehat\rho_h|\le Ce^{-c/h}.
\end{equation}
In particular $s_h=o(1)$ and $\widehat\rho_h=o(h)$.
\end{theorem}
\begin{proof}
Choose the cutoffs in \Cref{lem:IMS-tail} with disjoint left and
right cores.  Splitting $\phi_h^L=\chi_L\phi_h^L+(1-\chi_L)\phi_h^L$ gives
\[
 |\ip{\phi_h^L}{\phi_h^R}|
 \le \|(1-\chi_L)\phi_h^L\|+\|\chi_L\phi_h^R\|.
\]
The first term is exponentially small by \Cref{lem:IMS-tail};
on $\supp\chi_L$ one has $1-\chi_R=1$, so the second term is controlled by
the corresponding right-well estimate.  Likewise, $v^L$ is supported inside
the left core, where the right orbital is exponentially small, and therefore
$|\widehat\rho_h|=|\ip{\phi_h^L}{v^L\phi_h^R}|\le Ce^{-c/h}$.
\end{proof}

Define the scaled left and right residuals
\begin{equation}\label{eq:residuals}
 r_h^L:=(H_h-E_h)\phi_h^L=v^R\phi_h^L,
 \qquad
 r_h^R:=(H_h-E_h)\phi_h^R=v^L\phi_h^R.
\end{equation}
\begin{lemma}[Coarse defect and residual bounds]\label{lem:coarse-defect-residual}
There are $c,C>0$ such that, for all sufficiently small $h$,
\begin{equation}\label{eq:coarse-defect-residual}
 |\widehat\delta_h|+\|r_h^L\|+\|r_h^R\|\le Ce^{-c/h}.
\end{equation}
\end{lemma}
\begin{proof}
The support of the right potential lies a fixed positive distance from the full
left one-well support.  Hence \Cref{lem:IMS-tail} gives
$\|\phi_h^L\|_{L^2(\supp v^R)}\le Ce^{-c/h}$ after decreasing $c$ if
necessary.  Since $v^R$ is a fixed bounded multiplier, this proves the left
residual bound and, by
$\widehat\delta_h=\ip{\phi_h^L}{v^R\phi_h^L}$, the diagonal bound.  Reflection
gives the right residual estimate.
\end{proof}

\subsection{Parity trial states and coarse Schur data}\label{sec:parity-trial-data}
For $\sigma\in\{+1,-1\}$, where $+1$ denotes the even sector and $-1$ the odd
sector, let
\begin{equation}\label{eq:parity-space}
 \mathcal H_\sigma:=\{u\in L^2(\R^2):\mathsf Pu=\sigma u\}.
\end{equation}
Define
\begin{equation}\label{eq:psi-sigma}
 \psi_{h,\sigma}
 :=\frac{\phi_h^L+\sigma\phi_h^R}
 {\sqrt{2(1+\sigma s_h)}}.
\end{equation}
By \Cref{thm:overlap-ledger}, $s_h=o(1)$, so the denominator is positive for small $h$. Moreover,
$\mathsf P\psi_{h,\sigma}=\sigma\psi_{h,\sigma}$. Let $\Pi_{h,\sigma}$ denote the orthogonal projection onto
$\C\psi_{h,\sigma}$ in $\mathcal H_\sigma$, and write $Q_{h,\sigma}=1-\Pi_{h,\sigma}$ on $\mathcal H_\sigma$.
Whenever a compressed inverse
$[Q_{h,\sigma}(H_h-E)Q_{h,\sigma}]^{-1}$ appears below, it denotes the inverse
of the restriction of the compressed operator to $\Ran Q_{h,\sigma}$.

For real $E$, define the low-energy window by
\begin{equation}\label{eq:low-window}
 |E-E_h|\le\frac14c_{\rm dw}h.
\end{equation}
\begin{corollary}[Parity complement inverse]
For each parity $\sigma$ and every real $E$ in \Cref{eq:low-window},
\begin{equation}\label{eq:parity-compressed-inverse}
 \left\|
 \left[
 Q_{h,\sigma}(H_h-E)Q_{h,\sigma}
 \big|_{\Ran Q_{h,\sigma}}
 \right]^{-1}
 \right\|
 \le\frac{2}{c_{\rm dw}h}.
\end{equation}
\end{corollary}

\begin{proof}
Take $u\in\Ran Q_{h,\sigma}\subset\mathcal H_\sigma$.  Parity gives
\[
 \ip{u}{\phi_h^R}=\sigma\ip{u}{\phi_h^L}.
\]
Orthogonality to $\psi_{h,\sigma}$ therefore implies
\[
 0=\ip{u}{\phi_h^L+\sigma\phi_h^R}
 =2\ip{u}{\phi_h^L},
\]
so $u\in\mathcal V_h^\perp$.  Subtracting
$|E-E_h|\norm{u}^2$ from \Cref{eq:Vperp-gap} yields a lower bound
$\frac12c_{\rm dw}h\norm{u}^2$ under \Cref{eq:low-window}.  The spectral
theorem gives \Cref{eq:parity-compressed-inverse}.
\end{proof}

The exact Rayleigh quotient of the parity trial state is
\begin{equation}\label{eq:exact-rayleigh}
 a_{h,\sigma}
 :=\ip{\psi_{h,\sigma}}{H_h\psi_{h,\sigma}}
 =E_h+\frac{\widehat\delta_h+\sigma\widehat\rho_h}
 {1+\sigma s_h}.
\end{equation}

Define the complement coupling
\begin{equation}\label{eq:B-sigma}
 B_{h,\sigma}
 :=Q_{h,\sigma}H_h\psi_{h,\sigma}
 =Q_{h,\sigma}(H_h-E_h)\psi_{h,\sigma}.
\end{equation}

\begin{corollary}[Coarse Schur correction bound]\label{cor:finite-certificate}
Uniformly in $\sigma\in\{+1,-1\}$ and in the window
\Cref{eq:low-window},
\begin{equation}\label{eq:finite-certificate}
 \norm{B_{h,\sigma}}^2
 \left\|
 \left[
 Q_{h,\sigma}(H_h-E)Q_{h,\sigma}
 \big|_{\Ran Q_{h,\sigma}}
 \right]^{-1}
 \right\|
 \le Ce^{-c/h}=o(h).
\end{equation}
\end{corollary}

\begin{proof}
By \Cref{eq:residuals} and \Cref{eq:psi-sigma},
\[
 (H_h-E_h)\psi_{h,\sigma}
 =\frac{r_h^L+\sigma r_h^R}{\sqrt{2(1+\sigma s_h)}}.
\]
\Cref{thm:overlap-ledger} makes the denominator uniformly bounded away
from zero, while \Cref{lem:coarse-defect-residual} gives an exponential
bound on the numerator.  Combine this with
\Cref{eq:parity-compressed-inverse}; the polynomial factor $h^{-1}$ is absorbed
by decreasing the exponential constant.
\end{proof}

\subsection{Parity Schur reduction}\label{sec:parity-schur-reduction}
For real $E$ in the low-energy window, set
\begin{equation}\label{eq:A-compressed}
 \mathcal A_{h,\sigma}(E)
 :=Q_{h,\sigma}(H_h-E)Q_{h,\sigma}
 \big|_{\Ran Q_{h,\sigma}}.
\end{equation}
By \Cref{eq:parity-compressed-inverse},
$\mathcal A_{h,\sigma}(E)$ is positive and invertible.  Define
\begin{equation}\label{eq:Sigma}
 \Sigma_{h,\sigma}(E)
 :=\ip{B_{h,\sigma}}
 {\mathcal A_{h,\sigma}(E)^{-1}B_{h,\sigma}}.
\end{equation}
The function is nonnegative and real analytic in the window.  Since
$\partial_E\mathcal A_{h,\sigma}(E)=-I$ on $\Ran Q_{h,\sigma}$,
\begin{equation}\label{eq:Sigma-derivative}
 \Sigma_{h,\sigma}'(E)
 =\left\|
 \mathcal A_{h,\sigma}(E)^{-1}B_{h,\sigma}
 \right\|^2\ge0.
\end{equation}

\begin{theorem}[Parity Schur reduction and the two low eigenvalues]
\label{thm:parity-schur}
For all sufficiently small $h$, each parity restriction
$H_h|_{\mathcal H_\sigma}$ has exactly one eigenvalue
$E_{h,\sigma}$ in the window \Cref{eq:low-window}.  It is simple and
satisfies the exact scalar equation
\begin{equation}\label{eq:schur-equation}
 E_{h,\sigma}
 =a_{h,\sigma}-\Sigma_{h,\sigma}(E_{h,\sigma}).
\end{equation}
More generally, for every $E$ in the window, the Feshbach map is isospectral:
$E$ is an eigenvalue of $H_h|_{\mathcal H_\sigma}$ if and only if
\begin{equation}\label{eq:scalar-feshbach-function}
 F_{h,\sigma}(E)
 :=a_{h,\sigma}-E-\Sigma_{h,\sigma}(E)=0,
\end{equation}
and the kernel dimensions agree.  At a zero, an associated eigenvector is
\begin{equation}\label{eq:feshbach-eigenvector}
 \psi_{h,\sigma}
 -\mathcal A_{h,\sigma}(E)^{-1}B_{h,\sigma}.
\end{equation}
Moreover, uniformly throughout the window,
\begin{equation}\label{eq:feshbach-envelope-certificate}
 0\leq\Sigma_{h,\sigma}(E)
 \leq C e^{-c/h}.
\end{equation}
Consequently,
\begin{equation}\label{eq:E-sigma-expansion}
 E_{h,\sigma}
 =E_h+\frac{\widehat\delta_h+\sigma\widehat\rho_h}
 {1+\sigma s_h}
 +\cO\!\left(e^{-c/h}\right),
\end{equation}
and $E_{h,+},E_{h,-}$ are the two lowest eigenvalues of $H_h$ on
$L^2(\R^2)$.
\end{theorem}

\begin{proof}
Relative to
$\mathcal H_\sigma=\C\psi_{h,\sigma}\oplus\Ran Q_{h,\sigma}$, the operator
$H_h-E$ has the block form
\[
 \begin{pmatrix}
 a_{h,\sigma}-E & B_{h,\sigma}^*\\
 B_{h,\sigma} & \mathcal A_{h,\sigma}(E)
 \end{pmatrix}.
\]
The lower-right block is positive and invertible by
\Cref{eq:parity-compressed-inverse}.  If a vector in the kernel is written as
$c\psi_{h,\sigma}+\eta$, the lower block equation gives
\[
 \eta=-c\,\mathcal A_{h,\sigma}(E)^{-1}B_{h,\sigma}.
\]
Substitution into the upper equation gives
$F_{h,\sigma}(E)c=0$.  Conversely, a zero of
$F_{h,\sigma}$ produces the vector \Cref{eq:feshbach-eigenvector}.
This proves the isospectral statement and equality of kernel dimensions.

\Cref{cor:finite-certificate} gives, uniformly in the window,
\[
 0\le\Sigma_{h,\sigma}(E)
 \le \norm{B_{h,\sigma}}^2
 \norm{\mathcal A_{h,\sigma}(E)^{-1}}
 \le Ce^{-c/h},
\]
which proves \Cref{eq:feshbach-envelope-certificate}.  Moreover, \Cref{lem:coarse-defect-residual} and \Cref{thm:overlap-ledger} imply
\[
 a_{h,\sigma}=E_h+o(h),
 \qquad
 \Sigma_{h,\sigma}(E)=o(h)
\]
uniformly in the window.  Put
\[
 g_{h,\sigma}(E):=E-a_{h,\sigma}+\Sigma_{h,\sigma}(E).
\]
At the two endpoints of \Cref{eq:low-window},
\begin{align*}
 g_{h,\sigma}\!\left(E_h-\frac14c_{\rm dw}h\right)
 &=-\frac14c_{\rm dw}h+o(h)<0,\\
 g_{h,\sigma}\!\left(E_h+\frac14c_{\rm dw}h\right)
 &=\frac14c_{\rm dw}h+o(h)>0.
\end{align*}
By \Cref{eq:Sigma-derivative},
\[
g_{h,\sigma}'(E)=1+\Sigma_{h,\sigma}'(E)\ge1.
\]
Thus there is exactly one zero $E_{h,\sigma}$ in the window.  The kernel
corresponding to that zero is one-dimensional, so the eigenvalue is simple.
\Cref{eq:E-sigma-expansion} follows from
\Cref{eq:schur-equation}, \Cref{eq:exact-rayleigh}, and
\Cref{eq:feshbach-envelope-certificate}.

Finally, $E_{h,\sigma}=E_h+o(h)$.  For every two-dimensional subspace of
$\mathcal H_\sigma$, there is a nonzero vector orthogonal to
$\psi_{h,\sigma}$; as above, such a vector lies in $\mathcal V_h^\perp$, so \Cref{thm:two-well-gap} gives its Rayleigh quotient
at least $E_h+c_{\rm dw}h$.  Hence the second min--max value in each parity
sector is at least $E_h+c_{\rm dw}h$, whereas $E_{h,\sigma}=E_h+o(h)$.
Therefore $E_{h,\sigma}$ is the simple ground eigenvalue in its parity sector.
Since $L^2(\R^2)=\mathcal H_+\oplus\mathcal H_-$, the two parity ground
values are exactly the two lowest eigenvalues of the full operator.
\end{proof}

\Cref{sec:active-integral} will construct a positive unscaled active
envelope $\mathfrak a_h$; set
$\widehat{\mathfrak a}_h:=h^2\mathfrak a_h$.  
To turn the exact Schur equations
into a signed splitting asymptotic, it remains only to prove
\begin{equation}\label{eq:schur-sharp-interface}
 |\widehat\delta_h|
 +h^{-1}\bigl(\|r_h^L\|^2+\|r_h^R\|^2\bigr)
 +\sup_{\sigma,E}\Sigma_{h,\sigma}(E)
 =o(\widehat{\mathfrak a}_h),
\end{equation}
where the supremum is over the low-energy window
\Cref{eq:low-window}. That is, the errors incurred by the reduction must be small relative to the size of the hopping coefficient. This will be addressed in \Cref{sec:spectral-transfer} after the full derivation of the envelope in \Cref{sec:active-integral}.

\section{Exact hopping formula and the nine-cell ledger}\label{sec:exact-hopping}

The previous section treated $\rho_\lambda$ only as an algebraic off-diagonal
matrix element.  We now derive its exact two-source Landau-resolvent representation.  This
derivation identifies the gauge-invariant phase $\Phi$ used earlier to design the
cusp geometry and converts the three support components of the one-well
source into a $3\times3$ family of ordered cells. Seven of the nine cells are shown to have a fixed positive action gap relative to the two cross cells $I_{+-}$ and $I_{-+}$.

\subsection{Exact source--source formula}\label{sec:framework-exact}
In one-well coordinates, the exact formula is given by the following proposition; 
(see also \cite[eqn (2.29)]{FSW-absence}).

\begin{proposition}[Exact source--source resolvent representation]
\label{prop:hopping-derivation}
The hopping coefficient admits the exact representation
\begin{equation}\label{eq:exact-source-cell-total}
 \rho_\lambda
 =-h^2\iint_{\R^2\times\R^2}
 \overline{F_\lambda(z)}
 K_h^{-E_h}(z+w-2d)
 e^{i\Phi(z,w)/h}
 F_\lambda(w)\dd z\dd w,
\end{equation}
where $F_\lambda=h^{-2}v\phi_h$ is the one-well source
from \Cref{eq:source-exact-master} and $\Phi$ is given by
\Cref{eq:Phi}.
\end{proposition}

\begin{proof}
We first express the right one-well state through the free Landau resolvent,
and then undo the magnetic translations to return to the common one-well
coordinates.

By \Cref{eq:one-well-equations},
\[
 (H_{h}^{\mathrm{Lan}}-E_h)\phi_h^R=-v^R\phi_h^R.
\]
Since $E_h<0$, $H_{h}^{\mathrm{Lan}}-E_h$ is invertible, and therefore
\begin{equation}\label{eq:right-source-eq}
 \phi_h^R
 =-h^2(H_{h}^{\mathrm{Lan}}-E_h)^{-1}
 \bigl(h^{-2}v^R\phi_h^R\bigr).
\end{equation}
Substituting this identity into
$\rho_\lambda=h^{-2}\ip{v^L\phi_h^L}{\phi_h^R}$ gives
\begin{equation}\label{eq:rho-resolvent-intermediate}
 \rho_\lambda
 =-h^2
 \ip{h^{-2}v^L\phi_h^L}
 {(H_{h}^{\mathrm{Lan}}-E_h)^{-1}(h^{-2}v^R\phi_h^R)}.
\end{equation}
By the definition of the magnetic translations,
\begin{align}
 h^{-2}v^L(x)\phi_h^L(x)
 &=e^{\frac{ib}{2h}x\wedgep d}F_\lambda(x+d),\\
 h^{-2}v^R(y)\phi_h^R(y)
 &=e^{-\frac{ib}{2h}y\wedgep d}F_\lambda(d-y).
\end{align}
Using the symmetric-gauge kernel \Cref{eq:landau-kernel-master},
\Cref{eq:rho-resolvent-intermediate} becomes
\[
 \rho_\lambda
 =-h^2\iint
 \overline{F_\lambda(x+d)}
 K_h^{-E_h}(x-y)
 F_\lambda(d-y)
 \exp\!\left[
 -\frac{ib}{2h}
 \bigl(x\wedgep d+y\wedgep d+x\wedgep y\bigr)
 \right]
 \dd x\dd y .
\]

Now set
\[
 z=x+d,\qquad w=d-y.
\]
Then $x-y=z+w-2d$, and
\[
 x\wedgep d+y\wedgep d+x\wedgep y
 =-2d\wedgep(z-w)-z\wedgep w.
\]
Hence
\[
 -\frac b2
 \bigl(x\wedgep d+y\wedgep d+x\wedgep y\bigr)
 =
 b\left[d\wedgep(z-w)+\frac12z\wedgep w\right]
 =\Phi(z,w).
\]
Thus the phase introduced in the geometric construction of \Cref{sec:construction} is exactly the magnetic phase produced by the two-source resolvent representation. Since the change of variables has unit Jacobian,
\Cref{eq:exact-source-cell-total} follows.
\end{proof}

\subsection{Three source components and nine ordered cells}
\label{ssec:global}
Recall from \Cref{sec:global-source} that
\[
 F_\alpha:=\mathbf 1_{S_\alpha}F_\lambda,
 \qquad \alpha\in\{0,+,-\},
 \qquad F_\lambda=F_0+F_++F_-.
\]
Decomposing both source variables in \Cref{eq:exact-source-cell-total}, define the ordered source cell
\begin{equation}\label{eq:cell-ab}
 I_{\alpha\beta}(h)
 :=-h^2\iint_{S_\alpha\times S_\beta}
 \overline{F_\lambda(z)}K_h^{-E_h}(z+w-2d)
 e^{i\Phi(z,w)/h}F_\lambda(w)\dd z\dd w.
\end{equation}
Then the hopping coefficient has the exact
nine-cell decomposition
\begin{equation}\label{eq:nine-cell-decomp}
 \rho_\lambda=\sum_{\alpha,\beta\in\{0,+,-\}}I_{\alpha\beta}(h).
\end{equation}
The cells $I_{+-}$ and $I_{-+}$ will be evaluated in \Cref{sec:active-integral}.  The purpose of this construction is to control the
remaining seven cells.

Recall $\cR(x_1,x_2)=(x_1,-x_2)$ and let $(\mathfrak C f)(x)=\overline{f(\cR x)}$.
\begin{lemma}[Antiunitary reflection symmetry]\label{lem:antiunitary-master}
If $v\circ\cR=v$, then $\mathfrak C H_h^{\rm atom}(v)=H_h^{\rm atom}(v)\mathfrak C$.  If the one-well ground state is simple, its constant phase may be chosen so that
\[
 \phi_h(\cR x)=\overline{\phi_h(x)},\qquad F_\lambda(\cR x)=\overline{F_\lambda(x)}.
\]
\end{lemma}
\begin{proof}
Complex conjugation combined with $x_2\mapsto-x_2$ changes the sign of the first magnetic momentum component and preserves the second.  Hence it commutes with $H_{h}^{\mathrm{Lan}}$ and with the real reflection-invariant potential.  Simplicity fixes the phase by checking that $\mathfrak C\phi_{h}$ is also an eigenmode of $H_h^{\rm atom}$ associated to $E_{h}$.
\end{proof}

\begin{lemma}[Exact cell symmetries]\label{lem:cell-symmetries}
For every $h$ for which the cells are defined,
\begin{equation}\label{eq:swap-symmetry}
 I_{\beta\alpha}=\overline{I_{\alpha\beta}}.
\end{equation}
Let $0^\sharp=0$, $+^\sharp=-$, and $-^\sharp=+$ denote the reflection
involution on support labels.  Then
\begin{equation}\label{eq:reflection-cell-symmetry}
 I_{\alpha^\sharp\beta^\sharp}
 =\overline{I_{\alpha\beta}}
 =I_{\beta\alpha}.
\end{equation}
Consequently,
\begin{equation}\label{eq:cell-symmetry-list}
 I_{00}\in\R,
 \quad I_{++}=I_{--}\in\R,
 \quad I_{0-}=I_{+0}=\overline{I_{0+}},
 \quad I_{-0}=I_{0+},
 \quad I_{-+}=\overline{I_{+-}}.
\end{equation}
\end{lemma}

\begin{proof}
The kernel is real and symmetric, and $\Phi(w,z)=-\Phi(z,w)$.  Taking the
complex conjugate of \Cref{eq:cell-ab} and interchanging $z,w$ gives
\Cref{eq:swap-symmetry}.  Under the reflection $\cR$ one has
$K_h^{-E_h}(\cR\zeta)=K_h^{-E_h}(\zeta)$,
$\Phi(\cR z,\cR w)=-\Phi(z,w)$, and, componentwise,
\begin{equation}\label{eq:component-reflection}
 F_{\alpha^\sharp}(\cR x)=\overline{F_\alpha(x)}
 \qquad(\alpha\in\{0,+,-\})
\end{equation}
by \Cref{lem:antiunitary-master} and the reflected cusp construction.
Changing variables $(z,w)=(\cR z',\cR w')$ gives
\Cref{eq:reflection-cell-symmetry}.  The listed relations follow immediately.
\end{proof}

\subsection{Action geometry of the nine cells}
The exact decomposition \Cref{eq:nine-cell-decomp} contains two cross cells
and seven competitors.  The competitors are separated by elementary support
inequalities: the cusp tips are the unique rightmost points of the cusp
supports, and a same-cusp pair retains a fixed vertical separation from the
opposite cusp.

Set $\cE_h:=-E_h>0$, and recall that
$\cE_h^\circ=-E_h^\circ$.
Define the active reference action
\begin{equation}\label{eq:active-reference-action}
 A_{*,h}:=2G_{p,h}+\cJ_{\cE_h}(2L-R).
\end{equation}
For a pair of cusp variables, the real exponential cost in the integrand is
the sum of the two incoming radial-tail actions and the bridge action.  Set
\begin{equation}\label{eq:cusp-full-action}
 \cA_h(z,w):=g_h(z)+g_h(w)+\cJ_{\cE_h}(|z+w-2d|).
\end{equation}
Since $|p_\pm|=R$, one has
\[
 A_{*,h}=\cA_h(p_+,p_-)=\cA_h(p_-,p_+).
\]

\begin{theorem}[Linear growth on the cross supports]
\label{lem:cross-support-growth}
There are $b_{\rm br},c_{\rm cross},h_0>0$ such that the following holds.
For $\sigma\in\{+,-\}$, let $-\sigma$ denote the opposite sign and write
\[
 z=\Psi_\sigma(t,s),\qquad w=\Psi_{-\sigma}(u,r).
\]
Then, for all $0<h<h_0$,
\begin{equation}\label{eq:cross-bridge-distance-growth}
 |z+w-2d|\ge 2L-R+c_x(t+u),
\end{equation}
and hence
\begin{equation}\label{eq:bridge-linear-response}
 \cJ_{\cE_h}(|z+w-2d|)
 \ge \cJ_{\cE_h}(2L-R)+b_{\rm br}(t+u).
\end{equation}
Moreover,
\begin{equation}\label{eq:cross-action-growth}
 \cA_h(z,w)
 \ge A_{*,h}+c_{\rm cross}(t+u).
\end{equation}
One may take $c_{\rm cross}=a_-+b_{\rm br}$.
\end{theorem}

\begin{proof}
\Cref{lem:smooth-new} gives
\[
 z_1\le R/2-c_xt,
 \qquad
 w_1\le R/2-c_xu.
\]
Since $2L>R$, the first component of $z+w-2d$ is negative, and therefore
\[
 |z+w-2d|\ge 2L-z_1-w_1
 \ge 2L-R+c_x(t+u),
\]
which proves \Cref{eq:cross-bridge-distance-growth}.  Since the energy
parameters remain in the fixed compact neighborhood chosen after
\Cref{eq:delta-hop}, the mean-value theorem gives $b_{\rm br}>0$ such that
\Cref{eq:bridge-linear-response} holds uniformly for both cross orders and
all small $h$.

The outgoing-cusp inequality gives
\[
 g_h(z)+g_h(w)\ge2G_{p,h}+a_-(t+u).
\]
Combining this with \Cref{eq:bridge-linear-response} and
\Cref{eq:active-reference-action} proves \Cref{eq:cross-action-growth}
with $c_{\rm cross}=a_-+b_{\rm br}$.
\end{proof}

At the exponential level, each cusp source contributes the baseline incoming-tail
cost $G_{p,h}$, whereas the core source carries no such cost.  Thus the
reference costs for the core--core, core--cusp, and same-cusp cells contain,
respectively, zero, one, and two factors of $G_{p,h}$ in addition to the
bridge action.  The following theorem computes lower bounds on the action
gaps of the inactive support pairs. This will be used in
\Cref{prop:inactive-cell-ledger} to show that these channels have
exponentially smaller contributions than the cross-channels.

\begin{theorem}[Seven inactive support gaps]
\label{thm:support-action-separation}
There is $h_0>0$ such that, for $0<h<h_0$, the following estimates hold.
For each $\sigma\in\{+,-\}$,
\begin{align}
 \inf_{z,w\in S_0}\cJ_{\cE_h}(|z+w-2d|)
 &\ge A_{*,h}+8\delta_{\rm hop},
 \label{eq:core-core-gap}\\
 G_{p,h}+\inf_{\substack{z\in S_0\\w\in S_\sigma}}
 \cJ_{\cE_h}(|z+w-2d|)
 &\ge A_{*,h}+8\delta_{\rm hop},
 \label{eq:core-cusp-gap}\\
 2G_{p,h}+\inf_{z,w\in S_\sigma}
 \cJ_{\cE_h}(|z+w-2d|)
 &\ge A_{*,h}+8\delta_{\rm hop}.
 \label{eq:same-cusp-gap}
\end{align}
The core--cusp estimate also holds with the order of the variables reversed.
Thus the seven cells covered are the core--core cell, the four mixed cells,
and the two same-cusp cells.  The only cells not covered are the cross cells
\end{theorem}

\begin{proof}
If $z,w\in S_0$, then
\[
|z+w-2d|\ge 2L-2r_0=2L-R+\delta_{\rm cc}.
\]  
Since $\cJ_{\cE_h}$ is increasing, the difference between the left side of \Cref{eq:core-core-gap} and $A_{*,h}$ is bounded below by
\begin{align*}
&\cJ_{\cE_h}(2L-R+\delta_{\rm cc})
-\cJ_{\cE_h}(2L-R)-2G_{p,h}\\
&\qquad=
\cJ_{\cE_h}(2L-R+\delta_{\rm cc})
-\cJ_{\cE_h}(2L-R)-2\cJ_{\cE_h^\circ}(R).
\end{align*}
This is precisely
$\Delta_{\delta_{\rm cc}}(L;\cE_h,\cE_h^\circ)$. By the separation choice
made after \Cref{eq:delta-hop}, this reserve is at least
$24\delta_{\rm hop}$ whenever
$(\cE_h,\cE_h^\circ)\in I_\cE^2$. Since
$\cE_h,\cE_h^\circ\to1$, this holds for all sufficiently small $h$. In particular, \Cref{eq:core-core-gap} holds.

Fix $\sigma\in\{+,-\}$. If $z\in S_0$ and $w\in S_\sigma$, then
$z_1\le r_0$ and $w_1\le R/2$.  Therefore
\[
 |z+w-2d|\ge 2L-\frac R2-r_0
 =2L-R+\delta_{\rm cp}.
\]
Again using the monotonicity of $\cJ_{\cE_h}$, the difference between the
left-hand side of \Cref{eq:core-cusp-gap} and $A_{*,h}$ is bounded below by
\begin{align*}
&G_{p,h}
+\cJ_{\cE_h}(2L-R+\delta_{\rm cp})
-A_{*,h}\\
&\qquad=
\cJ_{\cE_h}(2L-R+\delta_{\rm cp})
-\cJ_{\cE_h}(2L-R)-\cJ_{\cE_h^\circ}(R).
\end{align*}
This is precisely
$\Delta_{\delta_{\rm cp}}(L;\cE_h,\cE_h^\circ)$, which is at least
$24\delta_{\rm hop}$ for all sufficiently small $h$ by the same separation
choice and the convergence $\cE_h,\cE_h^\circ\to1$. In particular, it is at
least $8\delta_{\rm hop}$.  The reversed order is identical, proving \Cref{eq:core-cusp-gap}.

Finally, if $z,w\in S_+$, \Cref{lem:smooth-new} gives
$z_1+w_1\le R$ and $z_2+w_2\ge\sqrt3R$.  Hence
\[
 |z+w-2d|^2\ge (2L-R)^2+3R^2.
\]
Reflection gives the same inequality on $S_-\times S_-$.  Since
$\frac{\dd}{\dd s}\cJ_\cE(\sqrt s)\ge b/4$, it follows that
\begin{align*}
\cJ_{\cE_h}(|z+w-2d|)-\cJ_{\cE_h}(2L-R)
&\ge
\frac b4\Bigl(|z+w-2d|^2-(2L-R)^2\Bigr)\\
&\ge\frac{3bR^2}{4}
=48\delta_{\rm hop}.
\end{align*}
Since the two $G_{p,h}$ terms cancel upon subtracting $A_{*,h}$,
this proves \Cref{eq:same-cusp-gap}, thus finishing the theorem.

\end{proof}

\subsection{Absolute bounds for the inactive cells}
Next we combine the action separation above with the global source and kernel bounds to derive the exponentially small bounds of the inactive channels.
\begin{lemma}[Componentwise $L^1$ source bounds]
\label{lem:component-source-L1}
For every $\eta>0$ there are $C_\eta,M>0$ such that
\begin{align}
 \norm{F_0}_{L^1}&\le Ch^{-2},\label{eq:F0-L1}\\
 \norm{F_+}_{L^1}+\norm{F_-}_{L^1}
 &\le C_\eta h^{-M}\Gamma_he^{-G_{p,h}/h}
 e^{-(\beta_0-\eta)\ell_h^2}.
 \label{eq:Fcusp-L1}
\end{align}
\end{lemma}

\begin{proof}
The first estimate follows from compact support, boundedness of $v^\circ$,
and $\norm{\phi_h}=1$.  For a cusp component, use the decomposition
\Cref{eq:source-split}, \Cref{thm:source-new}, and the cusp Jacobian
$t^2$.  \Cref{lem:log-flat-L1} bounds the incoming integral.  The scattered
part contains the additional factor $e^{-\beta_0\ell_h^2}$ and is smaller
after polynomial losses are absorbed into $h^{-M}$.
\end{proof}

\begin{lemma}[Landau kernel bounds on the cell supports]
\label{lem:global-kernel-bound}
For every pair $\alpha,\beta\in\{0,+,-\}$ there is $C>0$ such that
\begin{equation}\label{eq:cell-kernel-bound}
 \sup_{z\in S_\alpha,w\in S_\beta}
 |K_h^{\cE_h}(z+w-2d)|
 \le Ch^{-3/2}
 \exp\!\left[-\frac1h
 \inf_{S_\alpha\times S_\beta}
 \cJ_{\cE_h}(|z+w-2d|)\right].
\end{equation}
The same estimate holds after any fixed number of semiclassical spatial or
energy derivatives, with a larger polynomial power of $h^{-1}$.
\end{lemma}

\begin{proof}
Because the two shifted supports are disjoint and compact, every vector
$z+w-2d$ occurring in a cell remains in a compact annulus bounded away from
the origin.  \Cref{prop:K-Laplace}
and compactness give the assertion.
\end{proof}

\begin{proposition}[Inactive-cell action ledger]
\label{prop:inactive-cell-ledger}
For the seven inactive ordered cells there are $C,M>0$ such that
\begin{equation}\label{eq:inactive-cell-absolute}
 \sum_{(\alpha,\beta)\notin\{(+,-),(-,+)\}}
 |I_{\alpha\beta}(h)|
 \le Ch^{-M}(\Gamma_h^2+1)
 \exp\!\left[-\frac{A_{*,h}+7\delta_{\rm hop}}h\right].
\end{equation}
The three possible normalization factors $1,\Gamma_h,\Gamma_h^2$ are all bounded by a constant multiple of $\Gamma_h^2+1$.
\end{proposition}

\begin{proof}
Use the exact cell formula \Cref{eq:cell-ab}, the $L^1$ source bounds, and
\Cref{lem:global-kernel-bound}.  \Cref{thm:support-action-separation}
provides an $8\delta_{\rm hop}$ action gap.  Weakening it to
$7\delta_{\rm hop}$ absorbs all polynomial and subexponential prefactors.  The core--core, mixed, and same-packet source factors are all covered
by the common envelope $\Gamma_h^2+1$, using the elementary inequality
$\Gamma_h\le \Gamma_h^2+1$.  Thus no lower bound on $\Gamma_h$ is needed for
this absolute estimate.  Summing the seven cells proves
\Cref{eq:inactive-cell-absolute}.
\end{proof}

Combining the exact decomposition with the proposition gives
\begin{equation}\label{eq:rho-active-reduction}
 \rho_\lambda=I_{+-}(h)+I_{-+}(h)
 +\cO\!\left(h^{-M}(\Gamma_h^2+1)
 e^{-(A_{*,h}+7\delta_{\rm hop})/h}\right).
\end{equation}

By \Cref{lem:cell-symmetries}, $I_{-+}=\overline{I_{+-}}$, so equivalently
\[
 \rho_\lambda=2\Rea I_{+-}(h)
 +\cO\!\left(h^{-M}(\Gamma_h^2+1)
 e^{-(A_{*,h}+7\delta_{\rm hop})/h}\right).
\]
We have therefore reduced the hopping coefficient to the two cross cells,
which are exact complex conjugates, while every other cell carries a fixed
positive $h^{-1}$ action reserve.  The next section evaluates $I_{+-}$.

\section{Evaluation of the two cross channels}\label{sec:active-integral}
In this section we evaluate
\[
 I_{+-}(h)
 :=-h^2\iint_{S_+\times S_-}
 \overline{F_\lambda(z)}K_h^{-E_h}(z+w-2d)
 e^{i\Phi(z,w)/h}F_\lambda(w)\dd z\dd w.
\]
Since $I_{-+}=\overline{I_{+-}}$ exactly by \Cref{lem:cell-symmetries},
this determines both cross channels; \Cref{sec:full-hopping-section}
combines them with the seven inactive cells.

Recall from \Cref{eq:source-split} that on $S_+\cup S_-$,
\[
 F_\lambda=F_\lambda^{\rm in}+F_\lambda^{\rm sc},
\]
where
\[
F_\lambda^{\rm in} = h^{-2}\varepsilon(q_++q_-)c_h\phi^\circ_h
\]
is the direct radial core tail arriving at the cusps and
\[
F_\lambda^{\rm sc} = h^{-2}\varepsilon(q_++q_-)\eta_h
\]
is the exact scattered response generated by the perturbation.  The incoming
source is explicit and admits the local analytic continuation used below;
the scattered source is never analytically continued and will instead be
controlled by the global bounds of \Cref{sec:global-source}.  Accordingly,
\begin{equation}\label{eq:cross-source-four}
 I_{+-}=I_{+-}^{\rm in,in}+I_{+-}^{\rm in,sc}
 +I_{+-}^{\rm sc,in}+I_{+-}^{\rm sc,sc}.
\end{equation}
\Cref{sec:boundary-layer}--\Cref{sec:complete-incoming-cross} evaluate
$I_{+-}^{\rm in,in}$, while \Cref{sec:scat-exclusion} proves that the
other three terms are smaller by a strict log-flat margin

\subsection{Complex continuation of the Landau kernel}
For a complex vector $\zeta\in\C^2$, put
$\mathfrak r(\zeta)=(\zeta_1^2+\zeta_2^2)^{1/2}$, using the branch positive on the
relevant real annulus.

\begin{theorem}[Complex-parameter Landau kernel continuation]
\label{thm:complex-kernel}
Fix compact intervals $K_r\Subset(0,\infty)$ and
$K_\cE\Subset(0,\infty)$.  There are simply connected complex
neighborhoods $\mathcal N_r\supset K_r$ and
$\mathcal N_\cE\supset K_\cE$ on which $\Rea(r^2)>0$,
$\Rea r>0$, and $\Rea\cE>0$, such that the heat-kernel integral
\Cref{eq:K-integral}, with $|z|^2$ replaced by $r^2$, is holomorphic in
$(r,\cE)$ and admits the expansion
\begin{equation}\label{eq:complex-K-expansion-r}
 K_h^\cE(r)=h^{-3/2}e^{-\cJ_\cE(r)/h}
 \left(\sum_{j=0}^{N-1}h^jk_j(r,\cE)
       +h^NR_{N,h}(r,\cE)\right).
\end{equation}
Every fixed derivative of $R_{N,h}$ is uniformly bounded on smaller fixed
neighborhoods, and $k_0$ has no zero there.  The same conclusions hold
after a fixed number of spatial and energy derivatives.
\end{theorem}

\begin{proof}
The complete contour argument is given in
\Cref{app:complex-landau-kernel}.
\end{proof}

\begin{corollary}[Complex action expansion on the active scale]
\label{cor:complex-action-active}
Choose the neighborhoods in \Cref{thm:complex-kernel} to contain the
radial and bridge arguments arising from the cross cell.  If
$|\Ima\zeta|\le Ch\ell_h$ and $\Rea\zeta$ lies in the corresponding real
annulus, then $\mathfrak r(\zeta)\in\mathcal N_r$ and
\begin{equation}\label{eq:complex-action-split-new}
 \frac{\cJ_\cE(\mathfrak r(\zeta))}{h}
 =\frac{\cJ_\cE(|\Rea\zeta|)}h
 +\frac{\partial_r\cJ_\cE(|\Rea\zeta|)}h
  \frac{\Rea\zeta}{|\Rea\zeta|}\cdot
  (\zeta-\Rea\zeta)
 +\cO(h\ell_h^2).
\end{equation}
\end{corollary}

\begin{proof}
This is the first-order Taylor expansion of the holomorphic action on the
fixed neighborhoods supplied by \Cref{thm:complex-kernel}; the
quadratic remainder is $O(|\zeta-\Rea\zeta|^2/h)=O(h\ell_h^2)$.
\end{proof}

\subsection{Boundary-layer form of the incoming source}\label{sec:boundary-layer}
Set
\[
 k_p:=k_0(R,1)>0,
\]
and
\begin{equation}\label{eq:mRh}
 m_{R,h}:=\frac12\cJ_{\cE_h^\circ}'(R)=m_R+\cO(h).
\end{equation}
The following expansion is only asserted on the active boundary layer; the
global bounds of \Cref{thm:source-new} are used outside it.

\begin{lemma}[Incoming source expansion on the active scale]
\label{lem:incoming-boundary-layer}
Fix $M_0>0$.  Let $\Psi_\sigma^\C(t,s)$ denote the polynomial
complexification of the cusp chart for $\sigma\in\{+,-\}$.  On the local
complex deformations with $|s|\le s_0$ and $|t|\le M_0h\ell_h$, there are
holomorphic remainders $\rho_{\sigma,h}(t,s)$ satisfying
\begin{equation}\label{eq:incoming-action-remainder}
 |\rho_{\sigma,h}(t,s)|\le C|t|^2
\end{equation}
and
\begin{align}
 F_\lambda^{\rm in}(\Psi_\sigma^\C(t,s))
 &=-\eps a_qc_h\Gamma_hh^{-7/2}e^{-G_{p,h}/h}
 e^{-\beta\ell(t)^2}
 e^{-[m_{R,h}t+\rho_{\sigma,h}(t,s)]/h}\chi_b(s)
 \left(k_p+\cO(h+|t|)\right).
 \label{eq:incoming-complex-expansion}
\end{align}
In particular, uniformly for real $0<t\le M_0h\ell_h$,
\begin{equation}\label{eq:incoming-local-expansion}
 F_\lambda^{\rm in}(\Psi_\sigma(t,s))
 =-\eps a_qc_h\Gamma_hh^{-7/2}e^{-G_{p,h}/h}
 e^{-\beta\ell(t)^2-m_{R,h}t/h}\chi_b(s)
 \left(k_p+\cO(h\ell_h^2)\right).
\end{equation}
The principal branch of $\log(t_*/t)$ is used on the complex deformations.
\end{lemma}

\begin{proof}
On the cusp support, $\chi_a=1$ on the stated scale for small $h$.  Insert
\Cref{eq:radial-tail-again} and the differentiated expansion
\Cref{eq:K-expansion}.  For the complex statement, \Cref{thm:complex-kernel}
applies after decreasing $t_0$ so that one common complex disc in $t$ remains
inside the spatial-radius neighborhood for both cusp charts and the fixed
energy window \Cref{eq:radial-energy-window}.  Taylor expansion of the
holomorphic action gives
\[
 \cJ_{\cE_h^\circ}\!\left(\mathfrak r(\Psi_\sigma^\C(t,s))\right)
 =G_{p,h}+m_{R,h}t+\rho_{\sigma,h}(t,s),
 \qquad |\rho_{\sigma,h}(t,s)|\le C|t|^2.
\]
The leading kernel coefficient equals $k_p+\cO(h+|t|)$ uniformly on this
fixed neighborhood, which proves \Cref{eq:incoming-complex-expansion}.
On the real active scale,
$\rho_{\sigma,h}/h=\cO(h\ell_h^2)$, while $h+|t|=\cO(h\ell_h)$; expanding
the quadratic action remainder therefore gives
\Cref{eq:incoming-local-expansion}.
\end{proof}

\subsection{A log-flat Lambert-\texorpdfstring{$W$}{W} saddle}
For $k>-1$ and $\Rea c>0$, define
\begin{equation}\label{eq:Ik-def}
 \mathcal I_k(c,h)
 :=\int_0^{t_0}t^k\chi_a(t)
 \exp\!\left[-\beta\ell(t)^2-\frac{ct}{h}\right]\dd t.
\end{equation}

\begin{theorem}[Complex log-flat saddle]
\label{thm:log-flat-complex-saddle}
Let $c$ range in a compact subset $K$ of the open right half-plane and let
$k>-1$ be fixed.  Put
\begin{align}
 w_k(c,h)&:=W_0\!\left(
 \frac{ct_*}{2\beta h}e^{(k+1)/(2\beta)}\right),
 \label{eq:wk-def}\\
 \mathcal S_k(c,h)&:=t_*^{k+1}
 \sqrt{\frac{\pi}{\beta(1+w_k)}}
 \exp\!\left[-\beta(w_k^2+2w_k)
 +\frac{(k+1)^2}{4\beta}\right].
 \label{eq:Qk-def}
\end{align}
Here $W_0$ is the principal Lambert function and the square root is continued
from positive $c$.  Then, uniformly for $c\in K$,
\begin{equation}\label{eq:Ik-asymptotic}
 \mathcal I_k(c,h)=\mathcal S_k(c,h)
 \left(1+\cO\!\left(|w_k|^{-1}\right)\right).
\end{equation}
The expansion may be differentiated a fixed number of times in $c$.
Moreover, the part on which $\chi_a=1$ can be deformed to a contour through
the principal saddle on which the integral of the absolute value is bounded
by $C|\mathcal S_k(c,h)|$.
\end{theorem}

\begin{proof}
The complete contour argument is given in
\Cref{app:lambert-saddle}.
\end{proof}

The principal-branch expansion of $W_0$ gives, uniformly for $c\in K$,
\begin{align}
 w_k(c,h)&=\ell_h+\cO(\log\ell_h),\label{eq:wk-size}\\
 \log|\mathcal S_k(c,h)|
 &=-\beta\ell_h^2+2\beta\ell_h\log\ell_h+\cO(\ell_h).
 \label{eq:Sk-log-size}
\end{align}
In particular, for some $C_K>0$,
\begin{equation}\label{eq:Qk-lower-rough}
 |\mathcal S_k(c,h)|
 \ge\exp\!\left[-\beta\ell_h^2
 +2\beta\ell_h\log\ell_h-C_K\ell_h\right].
\end{equation}

\begin{corollary}[Comparison of fixed log-flat slopes]
\label{cor:log-flat-slope-comparison}
Let $K\Subset\{c:\Rea c>0\}$ and fix $k>-1$.  There is $C_K>0$ such that,
for $c_1,c_2\in K$ and small $h$,
\begin{equation}\label{eq:Q-slope-comparison}
 h^{C_K}
 \le\frac{|\mathcal S_k(c_1,h)|}{|\mathcal S_k(c_2,h)|}
 \le h^{-C_K}.
\end{equation}
The same comparison holds between
$\int_0^{t_0}t^ke^{-\beta\ell(t)^2-At/h}\dd t$ and
$|\mathcal S_k(c,h)|$ when $A$ ranges in a fixed compact subset of
$(0,\infty)$ and $c\in K$.
\end{corollary}

\begin{proof}
For $c_1,c_2$ in a fixed compact set, the principal-branch expansion of
$W_0$ gives $w_k(c_1,h)-w_k(c_2,h)=O(1)$ and
$\Rea w_k(c_j,h)=\ell_h+O(\log\ell_h)$.  Substitution in
\Cref{eq:Qk-def} shows that the difference of the logarithms of the two
moduli is $O(\ell_h)$, which is equivalent to
\Cref{eq:Q-slope-comparison}.  The final assertion follows from
\Cref{thm:log-flat-complex-saddle} with the positive real value
$c=A$.
\end{proof}

\begin{lemma}[Active normal truncation]\label{lem:active-normal-tail}
Let $K\Subset\{c:\Rea c>0\}$, let $A_0>0$, and fix $k>-1$.  For every
$N>0$ there is $M_0>0$ such that, with
\begin{equation}\label{eq:active-normal-scale}
 t_h^{\rm act}:=M_0h\ell_h,
\end{equation}
one has, uniformly for $c\in K$, $A\ge A_0$, and small $h$,
\begin{equation}\label{eq:active-normal-tail}
 \int_{t_h^{\rm act}}^{t_0}t^k
 e^{-\beta\ell(t)^2-At/h}\dd t
 \le h^N|\mathcal S_k(c,h)|.
\end{equation}
The principal saddle lies inside $|t|<t_h^{\rm act}/2$.  In the contour proof of
\Cref{thm:log-flat-complex-saddle}, the real segment
$[0,t_h^{\rm act}]$ may be deformed to a contour contained in $|t|\le t_h^{\rm act}$. Its outer
connector also contributes at most $h^N|\mathcal S_k(c,h)|$.
\end{lemma}

\begin{proof}
The proof is given in \Cref{app:lambert-saddle}.
\end{proof}

\subsection{The complete incoming cross cell}\label{sec:complete-incoming-cross}

Set
\begin{equation}\label{eq:c-star-def}
 c_*:=\alpha-i\theta,
\end{equation}
where $\alpha$ and $\theta$ are given by
\Cref{eq:alpha-new} and \Cref{eq:theta-new}; in particular $\Rea c_*=\alpha>0$.
Specializing the scalar saddle to the cusp Jacobian power $k=2$, define
\begin{align}
 w_*(h)&:=w_2(c_*,h)
 =W_0\!\left(
 \frac{c_*t_*}{2\beta h}e^{3/(2\beta)}\right),
 \label{eq:w-star}\\
 \mathcal S_*(h)&:=\mathcal S_2(c_*,h)
 =t_*^3\sqrt{\frac{\pi}{\beta(1+w_*(h))}}
 \exp\!\left[-\beta(w_*(h)^2+2w_*(h))
 +\frac{9}{4\beta}\right].
 \label{eq:Q-star}
\end{align}
Also set
\begin{equation}\label{eq:kb-Bs}
 k_b:=k_0(2L-R,1)>0,
 \qquad
 B_s:=\int_{\R}\chi_b(s)\dd s>0.
\end{equation}

\begin{proposition}[Incoming--incoming cross cell]
\label{prop:incoming-cross-cell}
Define the complex active amplitude and its positive envelope by
\begin{align}
 \mathcal M_h
 &:=\eps^2a_q^2c_h^2\Gamma_h^2h^{-13/2}
 e^{-A_{*,h}/h}\mathcal S_*(h)^2,
 \label{eq:active-complex-amplitude}\\
 \mathfrak a_h&:=|\mathcal M_h|
 =\eps^2a_q^2c_h^2\Gamma_h^2h^{-13/2}
 e^{-A_{*,h}/h}|\mathcal S_*(h)|^2>0.
 \label{eq:active-envelope}
\end{align}
Thus
\begin{equation}\label{eq:M-phase}
 \mathcal M_h=\mathfrak a_h e^{2i\arg\mathcal S_*(h)}.
\end{equation}
Define also the nonzero real constant
\begin{equation}\label{eq:C-star}
 \mathscr C_*:=-k_p^2k_bB_s^2\ne0.
\end{equation}
The power $h^{-13/2}$ in \Cref{eq:active-complex-amplitude} comes from
$h^2(h^{-7/2})^2h^{-3/2}$.  As $h\to0$,
\begin{equation}\label{eq:incoming-cell-asymptotic}
 I_{+-}^{\rm in,in}(h)
 =\mathscr C_*e^{i\Phi_*/h}\mathcal M_h
 \left(1+\cO(\ell_h^{-1})\right).
\end{equation}
\end{proposition}

\begin{proof}
We first truncate to a normal box.  Choose $M_0$ in
\Cref{lem:active-normal-tail} with $k=2$ and $c=c_*$, and use the
corresponding scale $t_h^{\rm act}$ from \Cref{eq:active-normal-scale}.
The principal saddle lies in $|t|<t_h^{\rm act}/2$.  Split the two real
normal integrations into the active box
\[
0<t,u<t_h^{\rm act}
\] 
and its complement.  On the complement,
\Cref{thm:source-new} and
\Cref{eq:cross-action-growth} give a product of real log-flat integrals with
strictly positive linear slopes.  \Cref{lem:active-normal-tail} controls
the variable outside the active interval, while \Cref{cor:log-flat-slope-comparison} shows that changing one fixed
positive slope to another changes the corresponding $\mathcal S_2$ by at
most a power of $h^{-1}$.
Choosing $N$ to dominate these polynomial losses
therefore gives
\begin{equation}\label{eq:active-box-tail}
 I_{+-}^{\rm in,in}
 -I_{+-}^{\rm in,in}\big|_{0<t,u<t_h^{\rm act}}
 =O(h^N\mathfrak a_h).
\end{equation}
In particular, this controls both the interval
$t_h^{\rm act}<t<t_2$, where $\chi_a=1$, and the fixed region where the
cutoff changes.

We next expand the joint phase and amplitude on the active box.  On the real
cusp support $F_\lambda^{\rm in}$ is real, so
$\overline{F_\lambda^{\rm in}}=F_\lambda^{\rm in}$ there; we analytically
continue this common real restriction.  Put
\[
 z=\Psi_+^\C(t,s),\qquad w=\Psi_-^\C(u,r),
\]
and define on the swept product domain
\begin{equation}\label{eq:complex-total-action}
 \cA_h^\C(z,w)
 :=\cJ_{\cE_h^\circ}(\mathfrak r(z))
 +\cJ_{\cE_h^\circ}(\mathfrak r(w))
 +\cJ_{\cE_h}(\mathfrak r(z+w-2d)),
\end{equation}
using the branches fixed by \Cref{thm:complex-kernel}.  On the real
cross support, $\cA_h^\C=\cA_h$.

\Cref{lem:incoming-boundary-layer} applies to both sources, and
\Cref{thm:complex-kernel} applies to the bridge because the cusp
charts are polynomial and the complex radial and bridge arguments remain in
fixed neighborhoods.  Put
\begin{equation}\label{eq:jh}
 j_h:=\frac12\cJ_{\cE_h}'(2L-R)=j+\cO(h).
\end{equation}
Then \Cref{eq:mRh} gives
$m_{R,h}+j_h=\alpha+\cO(h)$.  The normal phase slope is $\theta$ by
\Cref{eq:theta-new}.  Uniformly for $|s|,|r|\le s_0$ and
$|t|,|u|\le t_h^{\rm act}$ on the swept contours, Taylor expansion gives
\begin{equation}\label{eq:coupled-active-phase}
 \cA_h^\C(z,w)-i\Phi(z,w)
 =A_{*,h}-i\Phi_*+c_*(t+u)+\mathcal R_h(t,u,s,r),
\end{equation}
where
\begin{equation}\label{eq:coupled-active-remainder}
 |\mathcal R_h(t,u,s,r)|
 \le C\bigl(|t|^2+|u|^2+|tu|+h(|t|+|u|)\bigr).
\end{equation}
The quadratic terms arise from the Taylor remainders of the two
radial actions, the bridge action, and the magnetic phase.  The final term accounts for
$m_{R,h}+j_h$ and its limit $\alpha$.  Since
$|t|,|u|\le M_0h\ell_h$,
\begin{equation}\label{eq:coupled-remainder-small}
 \sup\frac{|\mathcal R_h|}{h}=O(h\ell_h^2)=o(\ell_h^{-1}).
\end{equation}

After extracting the common prefactor
\[
\eps^2a_q^2c_h^2\Gamma_h^2h^{-13/2}e^{-A_{*,h}/h}
\] 
and the phase $e^{i\Phi_*/h}$, the remaining amplitude is jointly
holomorphic in $(t,u)$ on the swept product domain and satisfies
\begin{equation}\label{eq:active-amplitude-limit}
 \mathcal B_h(t,u,s,r)
 =-k_p^2k_b\chi_b(s)\chi_b(r)+O(h\ell_h)
\end{equation}
uniformly there.  Here the exact quadratic action remainders have already
been placed in $\mathcal R_h$; the remaining kernel coefficients vary by
$O(h+|t|+|u|)=O(h\ell_h)$.  The two cusp Jacobians supply the factors
$t^2u^2$.  Thus the remaining normal dependence has the form
\[
t^2u^2
e^{-\beta\ell(t)^2-c_*t/h}
e^{-\beta\ell(u)^2-c_*u/h}
e^{-\mathcal R_h(t,u,s,r)/h}
\mathcal B_h(t,u,s,r).
\]

Now deform the two normal contours one at a time, using the truncated contour from
\Cref{lem:active-normal-tail}, denoted now by $\gamma_h^{\rm sd}$.  Joint holomorphy permits the first
deformation uniformly in the second variable and then the second deformation
uniformly on the first contour.  The same lemma controls the
outer connectors on the scale of the discarded normal tails; together with
the slope comparison used in \Cref{eq:active-box-tail}, their total
contribution is absorbed into the error there. The flat factor makes the
endpoint at zero harmless, so no additional side contribution remains.
Moreover, the absolute-contour bound in
\Cref{thm:log-flat-complex-saddle} gives
\[
 \int_{\gamma_h^{\rm sd}}|t|^2
 \left|e^{-\beta\ell(t)^2-c_*t/h}\right||\dd t|
 \le C|\mathcal S_*(h)|.
\]

The absolute-contour estimate allows \Cref{eq:coupled-remainder-small} and
\Cref{eq:active-amplitude-limit} to be inserted under the product integral
with a relative error $O(h\ell_h^2)$.  The two scalar normal integrals are
\[
    \mathcal S_*(h)(1+O(\ell_h^{-1})).
\] 
The tangential variables remain real and compactly supported, and the
leading tangential factor is
\[
 \iint\chi_b(s)\chi_b(r)\dd s\dd r=B_s^2.
\]
The limiting amplitude is therefore
\[
    -k_p^2k_bB_s^2=\mathscr C_*\neq0.
\] 
Since $h\ell_h^2=o(\ell_h^{-1})$, the scalar saddle error dominates the
local phase and amplitude errors. Combining this calculation with \Cref{eq:active-box-tail} proves
\Cref{eq:incoming-cell-asymptotic}.
\end{proof}

\subsection{Exclusion of every cusp-induced response term}\label{sec:scat-exclusion}
It remains to show that the three response terms in
\Cref{eq:cross-source-four} are smaller than the incoming--incoming term by
a strict log-flat margin.

\begin{table}[H]
\centering
\small
\begin{tabular}{lcc}
\toprule
Cross-cell contribution & Leading log-flat costs & Leading exponent\\
\midrule
incoming--incoming & two local cusp integrals & $2\beta\ell_h^2$\\
incoming--scattered & two local integrals plus one global response cost
& $(\beta+2\beta_0)\ell_h^2\ge3\beta_0\ell_h^2$\\
scattered--scattered & two local integrals plus two global response costs
& $4\beta_0\ell_h^2$\\
\bottomrule
\end{tabular}
\caption{The response margin.  The choice
$2\beta/3<\beta_0<\beta$ makes the mixed and double-scattered costs strictly
larger than the two-cost active envelope.  The proof below derives these
entries from the actual product integrals.}
\label{tab:response-costs}
\end{table}

\begin{proposition}[Three-versus-two log-flat margin]
\label{prop:response-exclusion}
There is $c_\beta>0$ such that
\begin{align}
 |I_{+-}^{\rm in,sc}|+|I_{+-}^{\rm sc,in}|
 &\le C\mathfrak a_he^{-c_\beta\ell_h^2},
 \label{eq:mixed-response-small}\\
 |I_{+-}^{\rm sc,sc}|
 &\le C\mathfrak a_he^{-2c_\beta\ell_h^2}.
 \label{eq:double-response-small}
\end{align}
Consequently,
\begin{equation}\label{eq:exact-cross-asymptotic}
 I_{+-}(h)
 =\mathscr C_*\mathfrak a_h
 e^{i(\Phi_*/h+2\arg \mathcal S_*(h))}
 \left(1+\cO(\ell_h^{-1})\right).
\end{equation}
\end{proposition}

\begin{proof}
Write $z=\Psi_+(t,s)$ and $w=\Psi_-(u,r)$.  The bridge estimate
\Cref{eq:bridge-linear-response} from \Cref{lem:cross-support-growth}
gives a positive bridge slope $b_{\rm br}$ for both variables.  Thus an
incoming variable has the positive total slope $a_-+b_{\rm br}$, while a
scattered variable has the positive total slope $\varkappa+b_{\rm br}$.

Apply the pointwise estimates of \Cref{thm:source-new}, the kernel
bound from \Cref{prop:K-Laplace}, the cusp Jacobians $t^2u^2$,
and \Cref{eq:bridge-linear-response}.  All fixed amplitudes and powers of
$h^{-1}$ are absorbed in $Ch^{-M}$.  For a mixed cell this gives
\begin{align}
 |I_{+-}^{\rm in,sc}|
 &\le Ch^{-M}\Gamma_h^2e^{-A_{*,h}/h}e^{-\beta_0\ell_h^2}
 \int_0^{t_0}t^2e^{-\beta\ell(t)^2-(a_-+b_{\rm br})t/h}\dd t\notag\\
 &\hspace{3.3cm}\times
 \int_0^{t_0}u^2e^{-\beta_0\ell(u)^2-(\varkappa+b_{\rm br})u/h}\dd u.
 \label{eq:mixed-response-product}
\end{align}
The same bound holds with the two source labels interchanged.  The product
factorizes, so applying \Cref{lem:log-flat-L1} to the two factors with
total loss $\eta>0$ gives
\begin{equation}\label{eq:mixed-response-absolute}
 |I_{+-}^{\rm in,sc}|+|I_{+-}^{\rm sc,in}|
 \le Ch^{-M}\Gamma_h^2e^{-A_{*,h}/h}
 e^{-(\beta+2\beta_0-\eta)\ell_h^2}
 \le Ch^{-M}\Gamma_h^2e^{-A_{*,h}/h}
 e^{-(3\beta_0-\eta)\ell_h^2}.
\end{equation}
For the double-scattered cell there are two global response factors and two
local cusp integrals.  Applying the same one-dimensional log-flat estimate
to each factor gives
\begin{equation}\label{eq:double-response-absolute}
 |I_{+-}^{\rm sc,sc}|
 \le Ch^{-M}\Gamma_h^2e^{-A_{*,h}/h}
 e^{-(4\beta_0-\eta)\ell_h^2}.
\end{equation}
These are estimates of the exact two-variable integrals; no cancellation is
used.

On the other hand, \Cref{eq:Qk-lower-rough} with $k=2$ implies, after
weakening the bound by a power of $h$,
\begin{equation}\label{eq:Qstar-lower-response}
 |\mathcal S_*(h)|^2\ge h^M e^{-(2\beta+\eta)\ell_h^2}.
\end{equation}
Choose $\eta>0$ so small that
$3\beta_0-2\beta-2\eta>0$, which is possible by
\Cref{eq:beta0-choice}.  Dividing
\Cref{eq:mixed-response-absolute} and
\Cref{eq:double-response-absolute} by the active envelope
\Cref{eq:active-envelope}, the remaining polynomial powers of $h^{-1}$ are
absorbed by a smaller fixed fraction of the positive $\ell_h^2$ margins.
This proves \Cref{eq:mixed-response-small} and
\Cref{eq:double-response-small} for a sufficiently small $c_\beta>0$.
Combining them with \Cref{prop:incoming-cross-cell} and
\Cref{eq:M-phase} proves \Cref{eq:exact-cross-asymptotic}.
\end{proof}

By \Cref{lem:cell-symmetries},
$I_{-+}=\overline{I_{+-}}$ exactly, so \Cref{eq:exact-cross-asymptotic}
evaluates both cross channels.  \Cref{sec:full-hopping-section} now
combines them with the inactive remainder from \Cref{sec:exact-hopping}.

\section{Full hopping asymptotic and sign changes}
\label{sec:full-hopping-section}
In this section we combine the inactive-cell reduction of
\Cref{sec:exact-hopping} with the cross-channel asymptotic of
\Cref{sec:active-integral} to obtain the full hopping asymptotic.  We
then analyze the Lambert-$W$ phase correction and prove that the hopping
coefficient changes sign infinitely often as $\lambda\to\infty$.

\subsection{Conjugate cross channels and the full asymptotic}\label{sec:full-hopping-asymptotic}
By \Cref{lem:cell-symmetries},
$I_{-+}=\overline{I_{+-}}$ exactly.  Since
$\mathscr C_*=-k_p^2k_bB_s^2<0$, \Cref{prop:response-exclusion}
therefore gives
\begin{equation}\label{eq:two-active-cosine}
 I_{+-}+I_{-+}
 =-2|\mathscr C_*|\mathfrak a_h
 \left[
 \cos\!\left(\frac{\Phi_*}{h}+2\arg \mathcal S_*(h)\right)
 +\cO(\ell_h^{-1})
 \right].
\end{equation}
The cosine is a consequence of exact inversion/reflection symmetry, not an
approximate symmetry of the local calculation.

We now use the polynomial normalization bound proved in \Cref{app:radial-normalization}:
\begin{equation}\label{eq:Gamma-main-input}
 \Gamma_h\ge ch^{3/2},\qquad \Gamma_h^{-1}\le Ch^{-3/2}.
\end{equation}
Thus missing factors of $\Gamma_h$ contribute only polynomial losses.
\begin{lemma}[Inactive cells are negligible on the active envelope]
\label{lem:inactive-relative}
There is $c>0$ such that
\begin{equation}\label{eq:inactive-relative}
 \sum_{(\alpha,\beta)\notin\{(+,-),(-,+)\}}
 |I_{\alpha\beta}(h)|
 \le C\mathfrak a_he^{-c/h}.
\end{equation}
\end{lemma}

\begin{proof}
Divide \Cref{eq:inactive-cell-absolute} by
\Cref{eq:active-envelope}.  \Cref{lem:Gamma-lower} and $\Gamma_h>0$ give
\[
 \frac{\Gamma_h^2+1}{\Gamma_h^2}
 =1+\Gamma_h^{-2}\le Ch^{-3},
\]
so the normalization mismatch contributes only a polynomial loss.  By
\Cref{eq:Qk-lower-rough} with $k=2$,
\[
 |\mathcal S_*(h)|^{-2}
 \le \exp\!\left[
 2\beta\ell_h^2-4\beta\ell_h\log\ell_h+C\ell_h
 \right].
\]
These subexponential losses are absorbed by the fixed factor
$e^{-7\delta_{\rm hop}/h}$, proving the claim after decreasing the exponential
constant.
\end{proof}

\begin{theorem}[Full hopping asymptotic]
\label{thm:full-hopping}
As $\lambda\to\infty$,
\begin{equation}\label{eq:full-hopping}
 \rho_\lambda
 =-2|\mathscr C_*|\mathfrak a_{1/\lambda}
 \left[\cos\Theta(\lambda)
 +\cO\!\left((\log\lambda)^{-1}\right)\right],
\end{equation}
\begin{equation}\label{eq:Theta}
 \Theta(\lambda):=\lambda\Phi_*+2\Ima\log \mathcal S_*(\lambda^{-1}).
\end{equation}
Since $\mathcal S_*(h)\ne0$ for small $h$, the logarithm in
\Cref{eq:Theta} is chosen continuously along the positive $\lambda$-axis
for large $\lambda$.
\end{theorem}

\begin{proof}
Combine \Cref{eq:rho-active-reduction}, \Cref{eq:two-active-cosine}, and
\Cref{lem:inactive-relative}, and set $h=\lambda^{-1}$.
\end{proof}

\subsection{Phase monotonicity and sign changes}\label{sec:phase-monotonicity}
The remaining point is to control the phase correction in \Cref{eq:Theta}.

\begin{lemma}[Phase asymptotics]\label{lem:phase-asymptotics}
As $\lambda\to\infty$,
\begin{equation}\label{eq:vartheta-bounds}
 \vartheta(\lambda)=\cO(\log\lambda),
 \qquad
 \vartheta'(\lambda)=\cO(\lambda^{-1}),
\end{equation}
and consequently
\begin{equation}\label{eq:phase-derivative}
 \Theta(\lambda)=\lambda\Phi_*+\cO(\log\lambda),
 \qquad
 \Theta'(\lambda)=\Phi_*+\cO(\lambda^{-1}).
\end{equation}
In particular, $\Theta$ is strictly increasing for all sufficiently large
$\lambda$ and tends to $+\infty$.
\end{lemma}

\begin{proof}
Put
\begin{equation}\label{eq:w-star-lambda}
 \omega_*(\lambda)
 :=w_*(\lambda^{-1})
 =W_0\!\left(
 \frac{c_*t_*}{2\beta}e^{3/(2\beta)}\lambda\right).
\end{equation}
The argument of the Lambert function lies on a fixed ray contained in the open
right half-plane.  Thus \Cref{eq:wk-size} gives
$\Rea\omega_*(\lambda)=\log\lambda+\cO(\log\log\lambda)$, while
$\Ima\omega_*(\lambda)=\cO(1)$.  The explicit formula \Cref{eq:Q-star}
therefore gives
\[
 \Ima\log\mathcal S_*(\lambda^{-1})=\cO(\log\lambda),
\]
which proves the first estimate in \Cref{eq:vartheta-bounds}.

Since
$\omega_*'=\omega_*/[\lambda(1+\omega_*)]$, direct differentiation of
\Cref{eq:Q-star} gives the exact identity
\begin{equation}\label{eq:dlogQ}
 \frac{\dd}{\dd\lambda}\log \mathcal S_*(\lambda^{-1})
 =-\frac{2\beta \omega_*(\lambda)}{\lambda}
 -\frac{\omega_*(\lambda)}{2\lambda(1+\omega_*(\lambda))^2}.
\end{equation}
Taking imaginary parts and using $\Ima\omega_*(\lambda)=\cO(1)$ gives
$\vartheta'(\lambda)=\cO(\lambda^{-1})$.  The two conclusions in
\Cref{eq:phase-derivative} now follow from \Cref{eq:Theta} and $\Phi_*>0$.
\end{proof}

\begin{corollary}[Infinitely many hopping sign changes and zeros]
\label{cor:hopping-zeros}
There are interlaced sequences
$\lambda_n^+,\lambda_n^-\to\infty$ such that
\begin{equation}\label{eq:hopping-signs}
 \rho_{\lambda_n^+}>0,
 \qquad
 \rho_{\lambda_n^-}<0.
\end{equation}
Consequently the hopping coefficient has infinitely many zeros tending to
infinity.
\end{corollary}

\begin{proof}
For each sufficiently large integer $n$, \Cref{lem:phase-asymptotics}
gives unique phase points $\lambda_n^-$ and $\lambda_n^+$ at which
$\Theta$ equals $2\pi n$ and $(2n+1)\pi$, respectively.  The error in
\Cref{eq:full-hopping} tends to zero and the envelope is positive, so
\Cref{eq:full-hopping} has the signs asserted in \Cref{eq:hopping-signs} at
these points.  The isolated simple one-well ground-state projection depends
continuously on $h=\lambda^{-1}$, and hence so does the defining hopping
matrix element.  The intermediate value theorem therefore gives a zero
between each consecutive pair.
\end{proof}

We have therefore established the oscillatory asymptotic and infinitely many
sign changes of the hopping coefficient itself.  To obtain exact eigenvalue
crossings, it remains to show that the defect, residual, and Schur terms are
negligible relative to the active hopping envelope.  This is the purpose of
\Cref{sec:spectral-transfer}.

\section{Envelope-relative spectral transfer and exact crossings}
\label{sec:spectral-transfer}
We now return to the Schur equations of \Cref{sec:schur-reduction}.
The objective is to verify the interface
\Cref{eq:schur-sharp-interface}: every nonoscillatory defect, residual, and
Schur correction is $o(\widehat{\mathfrak a}_h)$, where
$\widehat{\mathfrak a}_h>0$ is the scaled active envelope.  The key point is
that a residual is a one-source-to-one-target propagation and, after it is
squared in the Schur correction, carries one additional complete bridge
action compared with the hopping term.  This fixed bridge reserve makes the
Schur-reduction errors exponentially small relative to the active envelope.

Recall
\begin{equation}\label{eq:scaled-envelope}
 \widehat{\mathfrak a}_h:=h^2\mathfrak a_h.
\end{equation}
Also put
\begin{equation}\label{eq:Jstar-positive}
 J_{*,h}:=\cJ_{\cE_h}(2L-R).
\end{equation}
Since $2L-R>0$ and $\cE_h\to1$, there is
$j_*>0$ such that
\begin{equation}\label{eq:Jstar-lower}
 J_{*,h}\ge j_*
\end{equation}
for all sufficiently small $h$.

\subsection{Envelope-relative defect and residual estimates}
\label{sec:defect-residual}
Choose
\begin{equation}\label{eq:asrc}
 a_{\rm src}:=\frac12\min\{a_-,\varkappa\}>0.
\end{equation}
\Cref{thm:source-new}, after weakening $\beta$ to $\beta_0$, gives the
coarse route bound
\begin{equation}\label{eq:cusp-source-route-bound}
 |F_\lambda(\Psi_\pm(t,s))|
 \le C h^{-M}c_h\Gamma_h
 e^{-(G_{p,h}+a_{\rm src}t)/h}e^{-\beta_0\ell(t)^2}.
\end{equation}
The additional global log-flat factor in the scattered part has simply been
discarded in this upper bound.

For a source component $\alpha$ and target component $\beta$, define the
one-source route action by
\begin{equation}\label{eq:joint-route-action}
 D_{\alpha\beta,h}(z,w)
 :=\begin{cases}
 \cJ_{\cE_h}(|z+w-2d|),&\alpha=0,\\
 G_{p,h}+a_{\rm src}t+\cJ_{\cE_h}(|z+w-2d|),
 &z=\Psi_\alpha(t,s),\ \alpha\in\{+,-\}.
 \end{cases}
\end{equation}
The active one-source route action is
\begin{equation}\label{eq:Dstar}
 D_{*,h}:=G_{p,h}+J_{*,h}.
\end{equation}
Since $A_{*,h}=2G_{p,h}+J_{*,h}$, one also has
$A_{*,h}-D_{*,h}=G_{p,h}>0$.  The comparison below uses that squaring this
active one-source route leaves one additional complete bridge action relative
to the hopping action $A_{*,h}$.

\begin{lemma}[One-source route comparison]\label{lem:route-ledger}
There are $c_{\rm rt}>0$ and $h_0>0$ such that
\begin{align}
 \inf_{S_+\times S_-}D_{+-,h}
 &=D_{*,h},
 &
 \inf_{S_-\times S_+}D_{-+,h}
 &=D_{*,h},
 \label{eq:active-route-min}
\end{align}
with unique minimizing tip pairs.  Every other source--target route satisfies
\begin{equation}\label{eq:inactive-route-gap}
 \inf_{S_\alpha\times S_\beta}D_{\alpha\beta,h}
 \ge D_{*,h}+c_{\rm rt},
 \qquad
 (\alpha,\beta)\notin\{(+,-),(-,+)\}.
\end{equation}
The same lower bound holds on either active route outside a fixed neighborhood
of its tip pair.
\end{lemma}

\begin{proof}
Consider first $S_+\times S_-$.  If
$z=\Psi_+(t,s)$ and $w=\Psi_-(u,r)$, then
\Cref{eq:one-sided-x} gives
\[
 |z+w-2d|\ge2L-R+c_x(t+u).
\]
Hence
\begin{align*}
 D_{+-,h}(z,w)
 &\ge G_{p,h}+a_{\rm src}t
 +\cJ_{\cE_h}(2L-R+c_x(t+u))\\
 &\ge D_{*,h}
 +a_{\rm src}t
 +c_x\inf_{\rho\in[2L-R,\,2L-R+2c_xt_0]}
   \cJ_{\cE_h}'(\rho)(t+u).
\end{align*}
The derivative in the last line has a fixed positive lower bound for small
$h$.  Equality is therefore possible only when $t=u=0$, proving the first
identity and uniqueness.  Reflection gives the second active route.  On the
compact complement of any fixed neighborhood of the two tip pairs, the
continuous excess action has a positive minimum.

For the inactive routes, use 
\Cref{thm:support-action-separation} directly.  Fix $\sigma\in\{+,-\}$.
If the source is the core and the target is the cusp $S_\sigma$, then
\Cref{eq:core-cusp-gap} and
$A_{*,h}=2G_{p,h}+J_{*,h}$ give
\[
 \inf D_{0\sigma,h}
 =\inf_{S_0\times S_\sigma}\cJ_{\cE_h}(|z+w-2d|)
 \ge G_{p,h}+J_{*,h}+8\delta_{\rm hop}
 =D_{*,h}+8\delta_{\rm hop}.
\]
For the reversed mixed route, the reversed form of
\Cref{eq:core-cusp-gap} gives
\begin{align*}
 \inf D_{\sigma0,h}
 &\ge G_{p,h}
 +\inf_{S_\sigma\times S_0}\cJ_{\cE_h}(|z+w-2d|)\\
 &\ge A_{*,h}+8\delta_{\rm hop}
 \ge D_{*,h}+8\delta_{\rm hop}.
\end{align*}
For a same-cusp route, \Cref{eq:same-cusp-gap} implies
\[
 \inf_{S_\sigma\times S_\sigma}\cJ_{\cE_h}(|z+w-2d|)
 \ge J_{*,h}+8\delta_{\rm hop},
\]
and hence
\[
 \inf D_{\sigma\sigma,h}
 \ge G_{p,h}+J_{*,h}+8\delta_{\rm hop}
 =D_{*,h}+8\delta_{\rm hop}.
\]
Finally, \Cref{eq:core-core-gap} gives
\[
 \inf D_{00,h}
 =\inf_{S_0\times S_0}\cJ_{\cE_h}(|z+w-2d|)
 \ge A_{*,h}+8\delta_{\rm hop}
 \ge D_{*,h}+8\delta_{\rm hop}.
\]
Taking the minimum of these finitely many reserves and the two active
off-tip reserves proves \Cref{eq:inactive-route-gap}.
\end{proof}

For $\beta\in\{0,+,-\}$, write
\begin{equation}\label{eq:translated-supports}
 S_\beta^L:=\{x\in\R^2:x+d\in S_\beta\},
 \qquad
 S_\beta^R:=\{x\in\R^2:-x+d\in S_\beta\}.
\end{equation}
Thus $S_\beta^L$ and $S_\beta^R$ are the corresponding support components of
the left and right wells.

If $x$ lies in $S_\beta^R$ and $w=-x+d\in S_\beta$, the exact
source representation of the left orbital gives
\begin{equation}\label{eq:left-tail-route-representation}
 |\phi_h^L(x)|
 \le h^2\sum_{\alpha\in\{0,+,-\}}
 \int_{S_\alpha}|F_\alpha(z)|
 K_h^{\cE_h}(z+w-2d)\dd z.
\end{equation}
The magnetic phases have modulus one and do not enter this estimate.

\begin{lemma}[Left-orbital tail on the right support]
\label{lem:right-support-tail}
There are $C,M,c_0>0$ such that
\begin{align}
 \sup_{x\in S_+^R\cup S_-^R}|\phi_h^L(x)|
 &\le Ch^{-M}c_h\Gamma_h
 e^{-D_{*,h}/h},
 \label{eq:tail-right-cusps}\\
 \sup_{x\in S_0^R}|\phi_h^L(x)|
 &\le Ch^{-M}c_h\Gamma_h
 e^{-(D_{*,h}+c_0)/h}.
 \label{eq:tail-right-core}
\end{align}
The reflected estimates with left and right interchanged also hold.
\end{lemma}

\begin{proof}
Fix a target point $x\in S_\beta^R$ and put $w=-x+d\in S_\beta$.  \Cref{prop:K-Laplace} gives, uniformly for
$z\in S_0\cup S_+\cup S_-$,
\[
 K_h^{\cE_h}(z+w-2d)
 \le Ch^{-3/2}
 \exp\left[-\frac{\cJ_{\cE_h}(|z+w-2d|)}h\right].
\]
Insert this estimate into \Cref{eq:left-tail-route-representation} and
split the source sum into cusp and core components.

For a cusp source $S_\alpha$, \Cref{eq:cusp-source-route-bound} shows that
the exponential part of the integrand is bounded by
\[
 \exp\left[-\frac{D_{\alpha\beta,h}(z,w)}h\right]
 e^{-\beta_0\ell(t)^2}.
\]
\Cref{lem:route-ledger} therefore gives the baseline action $D_{*,h}$
on the active cusp-to-cusp route and a fixed additional reserve on every
inactive cusp route.  The remaining normal integral is finite:
\[
 \int_0^{t_0}t^2e^{-\beta_0\ell(t)^2}\dd t<\infty.
\]
All fixed Jacobian and kernel prefactors are absorbed into $Ch^{-M}$.

For a core source, \Cref{lem:component-source-L1} gives the required
$L^1$ bound but carries neither $c_h$ nor $\Gamma_h$.  Every core-source
route is inactive by \Cref{lem:route-ledger}.  After decreasing its
fixed action reserve slightly, \Cref{eq:ch-lower} and
\Cref{lem:Gamma-lower} from \Cref{app:radial-normalization}, namely
$c_h\ge\frac12$ and $\Gamma_h^{-1}\le Ch^{-3/2}$, absorb the missing
factor $c_h\Gamma_h$ into the polynomial prefactor.

If $\beta\in\{+,-\}$, the preceding bounds give
\Cref{eq:tail-right-cusps}.  If $\beta=0$, every source-to-target route is
inactive, so the minimum of the finitely many route reserves yields some
$c_0>0$ and gives \Cref{eq:tail-right-core}.  Reflection gives the
left-right interchanged estimates.
\end{proof}

\begin{theorem}[Envelope-relative defect and residual bounds]
\label{thm:defect-residual}
There are $C,M,c_{\rm res}>0$ such that
\begin{align}
 |\widehat\delta_h|
 &\le Ch^{-M}c_h^2\Gamma_h^2
 e^{-2D_{*,h}/h},
 \label{eq:delta-bound}\\
 \norm{r_h^L}^2+\norm{r_h^R}^2
 &\le Ch^{-M}c_h^2\Gamma_h^2
 e^{-2D_{*,h}/h}.
 \label{eq:residual-bound}
\end{align}
Consequently,
\begin{equation}\label{eq:defect-residual-relative}
 |\widehat\delta_h|
 +h^{-1}\bigl(\norm{r_h^L}^2+\norm{r_h^R}^2\bigr)
 \le Ce^{-c_{\rm res}/h}\widehat{\mathfrak a}_h.
\end{equation}
\end{theorem}

\begin{proof}
By \Cref{eq:delta-scaled} and \Cref{eq:residuals}, boundedness of the
right potential gives
\begin{align*}
 |\widehat\delta_h|
 &\le C\int_{\supp v^R}|\phi_h^L(x)|^2\dd x,\\
 \norm{r_h^L}^2
 &\le C\int_{\supp v^R}|\phi_h^L(x)|^2\dd x.
\end{align*}
The right support is the disjoint union of the three fixed components
$S_0^R,S_+^R,S_-^R$.  Applying \Cref{lem:right-support-tail} on each
component proves \Cref{eq:delta-bound} and the left part of
\Cref{eq:residual-bound}; reflection supplies the right residual.

The decisive identity is
\begin{equation}\label{eq:twoDminusA}
 2D_{*,h}-A_{*,h}=J_{*,h}\ge j_*>0.
\end{equation}
Using \Cref{eq:scaled-envelope} and \Cref{eq:active-envelope}, the ratio of
the common right side in \Cref{eq:delta-bound}--\Cref{eq:residual-bound}
to the positive scaled envelope satisfies, after absorbing fixed constants and
powers of $h$ into $Ch^{-M}$,
\[
 \frac{h^{-M}c_h^2\Gamma_h^2e^{-2D_{*,h}/h}}
 {\widehat{\mathfrak a}_h}
 \le Ch^{-M}
 e^{-(2D_{*,h}-A_{*,h})/h}|\mathcal S_*(h)|^{-2}
 =Ch^{-M}e^{-J_{*,h}/h}|\mathcal S_*(h)|^{-2}.
\]
The additional factor $h^{-1}$ multiplying the squared residuals in
\Cref{eq:defect-residual-relative} is another polynomial loss.  By
\Cref{eq:Qk-lower-rough},
$|\mathcal S_*(h)|^{-2}\le h^{-M}e^{(2\beta+1)\ell_h^2}$ after enlarging
$M$.  Since $J_{*,h}\ge j_*>0$, the complete bridge factor
$e^{-J_{*,h}/h}$ dominates all these polynomial and subexponential losses,
which proves \Cref{eq:defect-residual-relative}.
\end{proof}

\begin{corollary}[Envelope-relative Schur estimate]
\label{cor:envelope-schur-certificate}
Uniformly in $\sigma\in\{+1,-1\}$ and in the low-energy window
\Cref{eq:low-window},
\begin{equation}\label{eq:envelope-schur-certificate}
 \norm{B_{h,\sigma}}^2
 \left\|
 \left[Q_{h,\sigma}(H_h-E)Q_{h,\sigma}
 \big|_{\Ran Q_{h,\sigma}}\right]^{-1}
 \right\|
 \le Ce^{-c/h}\widehat{\mathfrak a}_h.
\end{equation}
Consequently,
\begin{equation}\label{eq:Sigma-envelope-late}
 \Sigma_{h,\sigma}(E)
 \le Ce^{-c/h}\widehat{\mathfrak a}_h.
\end{equation}
\end{corollary}

\begin{proof}
By \Cref{eq:residuals} and \Cref{eq:psi-sigma},
$\norm{B_{h,\sigma}}^2\le
C(\norm{r_h^L}^2+\norm{r_h^R}^2)$.  Combine
\Cref{thm:defect-residual} with the compressed inverse bound
\Cref{eq:parity-compressed-inverse}; its factor $h^{-1}$ is absorbed by the
fixed exponential margin.  The estimate for $\Sigma_{h,\sigma}$ follows from
\Cref{eq:Sigma}.
\end{proof}

Together with \Cref{thm:defect-residual}, this proves the
interface \Cref{eq:schur-sharp-interface} announced in
\Cref{sec:schur-reduction}.

\subsection{Transfer to the exact spectrum}
Define the scaled and unscaled signed splittings by
\begin{equation}\label{eq:signed-splittings}
 S_h:=E_{h,-}-E_{h,+},
 \qquad
 S_v(\lambda):=h^{-2}S_h
 =E_{\rm odd}(\lambda)-E_{\rm even}(\lambda).
\end{equation}

\begin{theorem}[Signed-splitting transfer]
\label{thm:signed-transfer}
There is $c>0$ such that
\begin{align}
 S_h&=-2\widehat\rho_h
 +\cO\!\left(e^{-c/h}\widehat{\mathfrak a}_h\right),
 \label{eq:S-scaled-transfer}\\
 S_v(\lambda)&=-2\rho_\lambda
 +\cO\!\left(e^{-c\lambda}\mathfrak a_{1/\lambda}\right).
 \label{eq:S-unscaled-transfer}
\end{align}
In particular,
\begin{equation}\label{eq:S-cosine}
 S_v(\lambda)
 =4|\mathscr C_*|\mathfrak a_{1/\lambda}
 \left[\cos\Theta(\lambda)
 +\cO\!\left((\log\lambda)^{-1}\right)\right].
\end{equation}
\end{theorem}

\begin{proof}
The exact Rayleigh quotient algebra gives
\begin{equation}\label{eq:rayleigh-difference}
 a_{h,-}-a_{h,+}
 =\frac{2(\widehat\delta_hs_h-\widehat\rho_h)}{1-s_h^2}
 =-2\widehat\rho_h
 +\frac{2\widehat\delta_hs_h-2\widehat\rho_hs_h^2}{1-s_h^2}.
\end{equation}
The coarse overlap estimate makes $s_h$ exponentially small.  
\Cref{thm:defect-residual} controls $\widehat\delta_h$ relative to the active
envelope, while \Cref{thm:full-hopping} gives
$|\widehat\rho_h|\le C\widehat{\mathfrak a}_h$.  Thus the second term
in \Cref{eq:rayleigh-difference} is exponentially small relative to the
envelope. Crucially, this comparison is made with the positive envelope
$\widehat{\mathfrak a}_h$, not with the oscillatory quantity
$|\widehat\rho_h|$.

The exact parity Schur equation gives
\[
 S_h=(a_{h,-}-a_{h,+})
 -\Sigma_{h,-}(E_{h,-})+\Sigma_{h,+}(E_{h,+}).
\]
\Cref{cor:envelope-schur-certificate} proves
\Cref{eq:S-scaled-transfer}.  Multiplication by $h^{-2}$ gives the unscaled
statement, and insertion of \Cref{eq:full-hopping} proves
\Cref{eq:S-cosine}.
\end{proof}

The next lemma supplies the continuity input used to pass from alternating
signs to exact zeros by the intermediate value theorem.

\begin{lemma}[Continuity of parity energies and hopping]
\label{lem:continuity}
The maps
$\lambda\mapsto E_{\rm even}(\lambda)$,
$\lambda\mapsto E_{\rm odd}(\lambda)$, and
$\lambda\mapsto S_v(\lambda)$ are continuous on $(0,\infty)$.  On every
large-$\lambda$ interval on which the one-well ground state is simple,
$\lambda\mapsto\rho_\lambda$ is continuous.
\end{lemma}

\begin{proof}
Fix $\lambda_0>0$ and let $(\mathsf U_su)(x)=su(sx)$, with
$s=(\lambda/\lambda_0)^{1/2}$.  A direct calculation gives
\begin{equation}\label{eq:dilation-continuity}
 \mathsf U_s^*H_v(\lambda)\mathsf U_s
 =\frac{\lambda}{\lambda_0}
 \left(-i\nabla-\frac{b\lambda_0}{2}x^\perp\right)^2
 +\lambda^2\bigl(v^L(x/s)+v^R(x/s)\bigr).
\end{equation}
For $\lambda$ near $\lambda_0$, smooth compact support gives uniform
continuity of the potential multipliers.  Equip the common Landau form domain
$\mathcal Q_A$ with one fixed shifted form norm.  By Riesz representation the
dilated closed form is a bounded coercive operator
$\mathsf A_\lambda:\mathcal Q_A\to\mathcal Q_A$ depending continuously on
$\lambda$ in operator norm.  If $j:\mathcal Q_A\hookrightarrow L^2$ is the
inclusion, the associated shifted physical resolvent is
$j\mathsf A_\lambda^{-1}j^*$.  The resolvent identity
\[
 \mathsf A_\lambda^{-1}-\mathsf A_{\lambda_0}^{-1}
 =\mathsf A_\lambda^{-1}(\mathsf A_{\lambda_0}-\mathsf A_\lambda)
  \mathsf A_{\lambda_0}^{-1}
\]
and the local coercive bounds therefore give norm-resolvent continuity.
Dilation preserves parity, proving continuity of the two sector bottoms and
their difference.

The same argument applies to the one-well family.  A simple isolated ground
projection has a local norm-continuous unit vector.  Magnetic translations
are strongly continuous in $h=\lambda^{-1}$, first on $C_c^\infty$ by
dominated convergence and then on $L^2$ by density.  Hence
$\phi_h^L,\phi_h^R$ vary continuously in norm, and so does
$h^{-2}\ip{\phi_h^L}{v^L\phi_h^R}$.  This scalar is independent of the
constant phase chosen for the one-well ground state, so the local choices
agree.
\end{proof}

\begin{corollary}[Infinite exact crossings]
\label{cor:infinite-crossings}
The signed splitting changes sign infinitely often and
has infinitely many zeros tending to infinity.  At every sufficiently large
zero, the full ground eigenspace is exactly two-dimensional, with one even and
one odd eigenfunction.  The parity of the ground state changes infinitely
often.
\end{corollary}

\begin{proof}
At the phase points used in \Cref{cor:hopping-zeros}, the cosine in
\Cref{eq:S-cosine} alternates between $+1$ and $-1$.  Its relative error
tends to zero and the envelope is positive, so $S_v(\lambda)$ has opposite
signs at consecutive phase points.  Continuity gives a zero between them.
The parity Schur theorem supplies exactly one simple low eigenvalue in each
parity sector and places every remaining eigenvalue an order $h$ higher.  At
a zero the two sector ground values coincide, so their orthogonal
eigenfunctions span an exactly two-dimensional full ground eigenspace.  The
alternating signs determine alternating ground-state parity.
\end{proof}

\subsection{Completion of the Main Theorem}
\begin{proof}[Proof of \Cref{thm:main}]
Fix the radial core, the radius $R$, the two exterior cusp components $q_\pm$, and all cusp
parameters defining their supports and profiles in the order stated after \Cref{eq:delta-hop}.  This produces one
potential $v$ and a separation threshold $L_0$.  Fix any $L\ge L_0$.
\Cref{sec:exact-hopping}--\Cref{sec:full-hopping-section} prove the full
hopping asymptotic with a strictly increasing phase, hence infinitely many
hopping zeros and sign changes.  \Cref{sec:schur-reduction} and
\Cref{sec:spectral-transfer} prove the envelope-relative identity
\[
 E_{\rm odd}(\lambda)-E_{\rm even}(\lambda)
 =-2\rho_\lambda
 +\cO\!\left(e^{-c\lambda}\mathfrak a_{1/\lambda}\right).
\]
No error is divided by the oscillatory quantity $|\rho_\lambda|$; every
error is measured against the positive active envelope.  The explicit phase
points and continuity now give infinitely many exact splitting zeros,
infinitely many parity changes, and the asserted two-dimensional eigenspaces.
Since the construction of $v$ is unchanged for every admissible choice of
$L$, the conclusions hold for all $L\ge L_0$.  This proves the theorem.
\end{proof}

\clearpage
\crefalias{section}{appendix}
\crefname{appendix}{Appendix}{Appendices}
\Crefname{appendix}{Appendix}{Appendices}

\appendix
\section{Weighted control of the scattered response}\label{app:control-scat-response}
The exact Feshbach formula already resums all powers of the fixed perturbation.
To compare the resulting correction with the incoming radial branch, we first
define a weight adapted to the outgoing normal coordinate near each cusp, then prove a weighted inverse
bound, and finally propagate that bound pointwise by local elliptic estimates.
These estimates are used in \Cref{sec:global-source} to separate the
resulting correction from the incoming radial branch by a positive multiple of
$\log^2(1/h)$.

\begin{lemma}[A cusp-adapted weight]\label{lem:cusp-weight}
There are fixed open sets
\[
 \overline{S_+\cup S_-}\subset U'_{\rm pkt}\Subset U_{\rm pkt}
 \Subset\R^2\setminus\supp v^\circ
\]
and a nonnegative function $T\in W^{1,\infty}(\R^2)$ with the following
properties:
\begin{enumerate}[label=\textup{(T\arabic*)},leftmargin=2.4em]
\item $T(\Psi_\pm(t,s))=t$ on the cusp supports;
\item $T=0$ on a neighborhood of the core and outside $U_{\rm pkt}$;
\item $\norm{\nabla T}_{L^\infty}\le C_T$;
\item there is $\varkappa_T>0$ such that, for every
$0\le\varkappa\le\varkappa_T$,
\begin{equation}\label{eq:cusp-weight-tail}
 \norm{(e^{\varkappa T/h}-1)\phi_h^\circ}
 \le e^{-c_T/h},
 \qquad
 \norm{e^{\varkappa T/h}\phi_h^\circ}\le C
\end{equation}
for all sufficiently small $h$.
\end{enumerate}
\end{lemma}

\begin{proof}
Choose disjoint fixed neighborhoods $U_\pm$ of $S_\pm$ whose closures are
disjoint from the core, and choose $\chi_\pm\in C_c^\infty(U_\pm)$ equal to
one on neighborhoods of $S_\pm$.  Since $n_\pm\perp\tau_\pm$, the normal
coordinate in the cusp chart is the linear function
\[
 t=(x-p_\pm)\cdot n_\pm.
\]
Using the positive-part notation $r_+:=\max\{r,0\}$, define
\[
 T(x)=\chi_+(x)\bigl((x-p_+)\cdot n_+\bigr)_+
      +\chi_-(x)\bigl((x-p_-)\cdot n_-\bigr)_+.
\]
This function is nonnegative and Lipschitz, equals $t$ on each cusp
support, and has fixed compact support away from $\supp v^\circ$.  Enlarging
$U_+\cup U_-$ slightly gives the stated sets $U'_{\rm pkt}\Subset
U_{\rm pkt}$.

By \Cref{lem:radial-agmon}, the radial ground state is
$O(h^{-M}e^{-c_F/h})$ on $\supp T$ in $L^2$, for some fixed $c_F>0$.
Choose $\varkappa_T>0$ so that
$\varkappa_T\norm{T}_{L^\infty}<c_F/2$.  On $\supp T$ the exponential
weight then loses at most $e^{c_F/(2h)}$, while it is equal to one outside
$\supp T$.  This proves \Cref{eq:cusp-weight-tail}, after absorbing a
fixed polynomial power of $h^{-1}$ into a smaller exponential constant.
\end{proof}

Fix henceforth $\varkappa>0$ so small that
\begin{equation}\label{eq:kappa-weight}
 \varkappa<\min\{a_-/2,\varkappa_T\},
 \qquad \varkappa^2C_T^2<\frac1{16}(1-\delta_0).
\end{equation}

\begin{lemma}[Weighted compressed inverse]\label{lem:weighted-inverse}
There is $C>0$ such that, for $|E-E_h^\circ|\le ch/2$,
\begin{equation}\label{eq:weighted-resolvent}
 \norm{e^{\varkappa T/h}
 \Bigl[\cQ_h(H_h^{\rm atom}-E)\cQ_h\big|_{\Ran\cQ_h}\Bigr]^{-1}
 \cQ_he^{-\varkappa T/h}}_{L^2\to L^2}
 \le Ch^{-1}.
\end{equation}
The estimate is uniform for the fixed cusp parameters chosen above.
\end{lemma}

\begin{proof}
Let $u\in\Ran\cQ_h$, put
\[
 f=\cQ_h(H_h^{\rm atom}-E)\cQ_hu,
 \qquad U=e^{\varkappa T/h}u,
 \qquad F=e^{\varkappa T/h}f.
\]
Because $u\perp\phi_h^\circ$,
\begin{equation}\label{eq:compressed-rank-term}
 (H_h^{\rm atom}-E)u=f+c_u\phi_h^\circ,
 \qquad
 c_u=\ip{\phi_h^\circ}{(H_h^{\rm atom}-E)u}
 =\ip{W\phi_h^\circ}{u}.
\end{equation}
The last equality follows from $(H_h^{\rm atom}(v^\circ)-E_h^\circ)\phi_h^\circ=0$ and $u\perp\phi_h^\circ$.

Since $T\in W^{1,\infty}$, the following conjugation identities are understood at the quadratic-form level; this regularity is sufficient. At this level, conjugation gives
\[
 e^{\varkappa T/h}\left(hP-\frac b2x^\perp\right)e^{-\varkappa T/h}
 =\left(hP-\frac b2x^\perp\right)+i\varkappa\nabla T.
\]
Consequently, the real part of the conjugated quadratic form is exactly
\begin{equation}\label{eq:weighted-form-identity}
 \Re\ip{U}{e^{\varkappa T/h}(H_h^{\rm atom}-E)e^{-\varkappa T/h}U}
 =\norm{\left(hP-\frac b2x^\perp\right)U}^2
 +\int\bigl(v-E-\varkappa^2|\nabla T|^2\bigr)|U|^2.
\end{equation}
Next, use the same IMS partition as in \Cref{lem:complement-coercivity}.  On the core piece $T=0$.  Although the weighted vector $U$ need not be exactly orthogonal to $\phi_h^\circ$, one has
\[
 \abs{\ip{\phi_h^\circ}{U}}
 =\abs{\ip{(e^{\varkappa T/h}-1)\phi_h^\circ}{u}}
 \le e^{-c/h}\norm{U}
\]
by \Cref{lem:radial-agmon} and property (T4).  Thus the order-$h$ radial complement gap applies with an exponentially small error.  On the cusp and transition pieces, \Cref{eq:epsilon-small}, \Cref{eq:final-energy-scale}, and \Cref{eq:kappa-weight} make the coefficient in \Cref{eq:weighted-form-identity} uniformly positive.  The fixed-partition magnetic IMS error is $\cO(h^2)\norm{U}^2$ and is absorbed.  We obtain
\begin{equation}\label{eq:weighted-coercive-form}
 ch\norm{U}^2
 \le C\abs{\ip{U}{F}}
 +C|c_u|\,\norm{U}\norm{e^{\varkappa T/h}\phi_h^\circ}
 +e^{-c/h}\norm{U}^2.
\end{equation}
Moreover, using \Cref{eq:compressed-rank-term},
\[
 |c_u|
 =\abs{\ip{e^{-\varkappa T/h}W\phi_h^\circ}{U}}
 \le \norm{e^{-\varkappa T/h}W\phi_h^\circ}\norm{U}
 =o(h)\norm{U},
\]
where the last estimate follows from \Cref{lem:radial-agmon}; property (T4) gives $\norm{e^{\varkappa T/h}\phi_h^\circ}\le C$.  Absorb the last two terms in \Cref{eq:weighted-coercive-form}, apply Cauchy--Schwarz to the first, and obtain
\[
 ch\norm{e^{\varkappa T/h}u}
 \le C\norm{e^{\varkappa T/h}f}.
\]
This proves the estimate for data in $\Ran\cQ_h$.  For arbitrary input $G\in L^2$, the inverse acts on $\cQ_he^{-\varkappa T/h}G$; the rank-one identity
\[
 e^{\varkappa T/h}\cQ_he^{-\varkappa T/h}
 =1-e^{\varkappa T/h}\phi_h^\circ\otimes
   (e^{-\varkappa T/h}\phi_h^\circ)^*
\]
and (T4) show that this additional projection has uniformly bounded norm.  Hence \Cref{eq:weighted-resolvent} holds as stated.  Uniformity follows from compactness of the finite parameter neighborhood and from the strict inequalities in the construction.
\end{proof}

\begin{lemma}[Weighted interior elliptic propagation]
\label{lem:weighted-interior-elliptic-propagation}
Let $U'\Subset U$ be fixed Euclidean open sets in a bounded region, let
$T\in W^{1,\infty}(U)$, and put $w=e^{\varkappa T/h}$.  Suppose that
\[
 (H_h^{\rm atom}(v)-E)u=f\quad\hbox{on }U,
 \qquad |E-E_h^\circ|\le Ch.
\]
For an integer $k\ge0$, set
\[
 \mathcal N_k:=\|wu\|_{L^2(U)}
 +\sum_{|\gamma|\le k}\|w(h\nabla)^\gamma f\|_{L^2(U)}.
\]
Then there are $C_k,M_k>0$, depending only on $U'\Subset U$, the fixed
potential, $k$, $\varkappa$, and $\|\nabla T\|_\infty$, such that
\begin{equation}\label{eq:weighted-elliptic}
 \sup_{x\in U'}w(x)\bigl|(h\nabla)^\alpha u(x)\bigr|
 \le C_kh^{-M_k}\mathcal N_k,
 \qquad |\alpha|\le k.
\end{equation}
In particular, $U'$ may contain either cusp tip.
\end{lemma}

\begin{proof}
For small $h$, cover $\overline{U'}$ by balls $B(x_0,h)$ satisfying
$B(x_0,2h)\Subset U$.  The Lipschitz bound gives
\[
 e^{-2\varkappa\|\nabla T\|_\infty}
 \le \frac{w(x)}{w(x_0)}
 \le e^{2\varkappa\|\nabla T\|_\infty}
 \qquad(x\in B(x_0,2h)).
\]
Set $x=x_0+hy$.  After multiplication by the bounded gauge factor that
removes the constant vector $-\frac b2x_0^\perp$, the equation on
$B(0,2)$ is a uniformly elliptic second-order equation whose coefficients
have bounded $C^j$ norms, uniformly in $x_0$ and $h$.  Standard interior
regularity therefore gives
\[
 \|\widetilde u\|_{H^{k+2}(B_1)}
 \le C_k\left(
   \|\widetilde u\|_{L^2(B_2)}
   +\|\widetilde f\|_{H^k(B_2)}
 \right).
\]
Sobolev embedding in two dimensions controls the $C^k(B_1)$ norm.  Undoing
the rescaling converts $\nabla_y$ into $h\nabla_x$; the change of measure
and the finite number of local derivatives introduce only a fixed power of
$h^{-1}$.  Multiplication by $w(x_0)$ and the bounded oscillation of the
weight on $B(x_0,2h)$ give \Cref{eq:weighted-elliptic}.  The argument is
entirely Euclidean and therefore applies without change to balls centered at
a cusp tip.
\end{proof}

\section{Radial normalization and the exterior coefficient}\label{app:radial-normalization}
This appendix supplies the polynomial control of the exterior normalization
$\Gamma_h$ used in \Cref{lem:inactive-relative} and
\Cref{lem:right-support-tail}.  It is independent of the oscillatory
cross-channel calculation.  In particular, we prove a polynomial lower bound
for $\Gamma_h$, ensuring that terms containing fewer factors of $\Gamma_h$
cannot gain an exponential advantage when they are later compared with the
active cross-channel scale.

\begin{lemma}[Rescaled radial core profile]\label{lem:radial-core-profile}
For $\psi_h$ defined in \Cref{eq:rescaled-radial-ground-state}, one has
\begin{equation}\label{eq:core-profile-convergence}
 \psi_h\longrightarrow\psi_0
 \quad\hbox{in }H^2_{\rm loc}(\R^2)\hbox{ and }C^0_{\rm loc}(\R^2)
 \quad\hbox{as }h\to0.
\end{equation}
\end{lemma}
\begin{proof}
\Cref{lem:radial-spectral} gives strong local $L^2$ convergence.  In
the rescaled variable $y=h^{-1/2}x$, the eigenvalue equation
becomes
\[
 \mathcal H_{\rm osc}\psi_h+V_h(y)\psi_h=\mu_h\psi_h,
 \qquad \mu_h=\mu_0+\cO(h^{1/2}),
 \qquad V_h(y)=\cO(h^{1/2}|y|^3)
\]
locally uniformly.  Interior elliptic estimates upgrade the convergence to
$H^2_{\rm loc}$, and the two-dimensional Sobolev embedding gives the
$C^0_{\rm loc}$ convergence.
\end{proof}

The active envelope contains the radial exterior coefficient $\Gamma_h^2$.  A
core--cusp cell carries only one factor of $\Gamma_h$, while the
core--core cell carries none.  The following lemma prevents this difference
from hiding an exponential loss.

\begin{lemma}[Polynomial lower bound for $\Gamma_h$]\label{lem:Gamma-lower}
For the radial ground state of $H_h^{\rm atom}(v^\circ)$ there are $c>0$ and $h_0>0$ such
that
\begin{equation}\label{eq:Gamma-lower}
 \Gamma_h\ge c h^{3/2},
 \qquad 0<h<h_0.
\end{equation}
In particular,
\begin{equation}\label{eq:Gamma-inverse}
 \Gamma_h^{-1}\le C h^{-3/2}.
\end{equation}
\end{lemma}

\begin{proof}
Restrict $H_{h}^{\mathrm{Lan}}+\cE_h^\circ$ to the angular-momentum-zero sector.  On radial
functions it is
\begin{equation}\label{eq:radial-free-operator}
 \mathscr L_{0,h}
 =-h^2\left(\partial_r^2+r^{-1}\partial_r\right)
 +\frac{b^2r^2}{4}+\cE_h^\circ
\end{equation}
acting in $L^2(\R_+,2\pi r\dd r)$.  Let $G_{0,h}(r,s)$ be its positive
Green kernel with respect to the measure $2\pi s\dd s$, so distributionally
\begin{equation}\label{eq:radial-green-delta}
 \mathscr L_{0,h}G_{0,h}(\,\cdot\,,s)
 =\frac{\delta_s}{2\pi s}.
\end{equation}
Let $u_h^{\rm reg}$ be the regular solution of
$\mathscr L_{0,h}u=0$, normalized by
\begin{equation}\label{eq:u-normalization}
 u_h^{\rm reg}(0)=1,
 \qquad (u_h^{\rm reg})'(0)=0.
\end{equation}
For $r>s$, Sturm--Liouville factorization initially gives
\begin{equation}\label{eq:radial-G-factor-pre}
 G_{0,h}(r,s)=C_hu_h^{\rm reg}(s)K_h^{\cE_h^\circ}(r)
\end{equation}
for a constant $C_h$.  The zero-angular-momentum projection of the full Landau Green kernel is
$G_{0,h}$.  More explicitly, with $x=(r,0)$ and
$y=(s\cos\vartheta,s\sin\vartheta)$,
\begin{equation}\label{eq:radial-green-angular-average}
 G_{0,h}(r,s)=\frac1{2\pi}\int_0^{2\pi}
 e^{-\frac{ib}{2h}rs\sin\vartheta}
 K_h^{\cE_h^\circ}
 \!\left(\sqrt{r^2+s^2-2rs\cos\vartheta}\right)\dd\vartheta.
\end{equation}
This is the kernel relative to the radial measure $2\pi s\dd s$.  Dominated
convergence gives
\begin{equation}\label{eq:radial-green-origin-limit}
 \lim_{s\downarrow0}G_{0,h}(r,s)=K_h^{\cE_h^\circ}(r),
 \qquad r>0.
\end{equation}
\Cref{eq:u-normalization}--\Cref{eq:radial-green-origin-limit}
force $C_h=1$.  Thus, writing $u_h=u_h^{\rm reg}$,
\begin{equation}\label{eq:radial-G-factor}
 G_{0,h}(r,s)=u_h(s)K_h^{\cE_h^\circ}(r),
 \qquad r>s.
\end{equation}

The equation for $u_h$ may be written
\begin{equation}\label{eq:u-monotone}
 (ru_h')'
 =\frac r{h^2}\left(\frac{b^2r^2}{4}+\cE_h^\circ\right)u_h.
\end{equation}
Starting from \Cref{eq:u-normalization}, a first-zero argument shows that
$u_h>0$, and then \Cref{eq:u-monotone} gives $u_h'\ge0$.  Hence
\begin{equation}\label{eq:u-lower}
 u_h(s)\ge1
 \qquad(s\ge0).
\end{equation}

The radial resolvent representation of the ground-state equation gives, for
$r>r_0$,
\begin{align}
 \phi_h^\circ(r)
 &=2\pi\int_0^{r_0}G_{0,h}(r,s)
 \bigl(-v^\circ(s)\phi_h^\circ(s)\bigr)s\dd s\notag\\
 &=K_h^{\cE_h^\circ}(r)
 \left[2\pi\int_0^{r_0}u_h(s)
 \bigl(-v^\circ(s)\phi_h^\circ(s)\bigr)s\dd s\right].
 \label{eq:Gamma-integral}
\end{align}
Comparison with \Cref{eq:radial-tail-again} yields the bracketed formula for
$\Gamma_h$.

By the core-profile convergence \Cref{eq:core-profile-convergence},
$h^{1/2}\phi_h^\circ(h^{1/2}y)$ converges locally uniformly to the strictly
positive oscillator ground state $\psi_0$.  Hence, for $0\le s\le h$ and small
$h$,
\begin{equation}\label{eq:phi-core-lower}
 \phi_h^\circ(s)\ge c_0h^{-1/2}.
\end{equation}
Also $-v^\circ(s)\ge1/2$ there.  Inserting
\Cref{eq:u-lower}--\Cref{eq:phi-core-lower} into
\Cref{eq:Gamma-integral} gives
\[
 \Gamma_h\ge c h^{-1/2}\int_0^h s\dd s=ch^{3/2}.
\]
This proves \Cref{eq:Gamma-lower} and \Cref{eq:Gamma-inverse}.
\end{proof}

\section{Complex continuation of the Landau kernel}\label{app:complex-landau-kernel}
\begin{proof}[Proof of \Cref{thm:complex-kernel}]
To show holomorphy of the integral,
write
\[
 \mathcal H(\tau;r,\cE)
 =\cE\tau+\frac b4\coth(b\tau)r^2,
 \qquad
 a(\tau)=\frac{b}{4\pi\sinh(b\tau)}.
\]
The assumptions $\Rea(r^2)>0$ and $\Rea\cE>0$ make the positive real
$\tau$-integral uniformly convergent: as $\tau\downarrow0$ its modulus is
bounded by $Ce^{-c/(h\tau)}$, and as $\tau\to\infty$ it is bounded by
$Ce^{-c\tau/h}$.  Differentiation under the integral is therefore valid on
compact subsets, so the integral is holomorphic in $(r,\cE)$.

We next prove the uniform complex Laplace expansion.  On the real compact set
$K_r\times K_\cE$, the critical point
\[
 \tau_*(r,\cE)=b^{-1}\operatorname{arsinh}
 \left(\frac{br}{2\sqrt\cE}\right)
\]
ranges in a compact subinterval of $(0,\infty)$ and is uniformly
nondegenerate.  Choose numbers $0<\tau_-<\tau_+$ and a pole-free strip
containing $[\tau_-,\tau_+]$ so that every real critical point lies in its
interior.  By compactness, after decreasing $\tau_-$ and increasing
$\tau_+$ if necessary, there is $\delta>0$ such that the two real tails and
the part of $[\tau_-,\tau_+]$ outside a fixed small neighborhood of the
critical point satisfy
\begin{equation}\label{eq:complex-kernel-reserve}
 \mathcal H(\tau;r,\cE)-\cJ_\cE(r)\ge4\delta
\end{equation}
for all real $(r,\cE)\in K_r\times K_\cE$.  Near zero and infinity this
follows directly from the terms $r^2/(4\tau)$ and $\cE\tau$; on the remaining
compact set it follows from uniqueness of the minimum.

The implicit-function theorem gives a unique holomorphic critical point
$\tau_*(r,\cE)$ after the parameter intervals are thickened slightly into
complex neighborhoods.  The parameter-dependent holomorphic Morse lemma
then gives, uniformly in a neighborhood of the real compact set, a
holomorphic coordinate $y$ and its inverse $\tau=\mathcal T(y;r,\cE)$ such
that
\begin{equation}\label{eq:complex-kernel-morse}
 \mathcal H(\mathcal T(y;r,\cE);r,\cE)
 =\cJ_\cE(r)+\frac{y^2}{2},
 \qquad |y|<2\rho,
\end{equation}
where $\rho>0$ is fixed.  For real parameters, $\mathcal T(y;r,\cE)$ is real
when $y$ is real.

For completeness, we describe the contour deformation uniformly.  Cover the
real parameter compact set by finitely many small polydiscs and, in each
polydisc, fix a real center.  Retain the two real heat-time tails.  In the
middle interval replace the real segment through the center's critical point
by the Morse curve
\[
 \gamma_{r,\cE}:=\{\mathcal T(y;r,\cE):-\rho\le y\le\rho\}.
\]
Join its endpoints to the retained real pieces by straight connectors.  If
the polydiscs are sufficiently small, all connectors remain in the fixed
pole-free strip and in the annular region where
$\Rea(\mathcal H-\cJ_\cE)\ge2\delta$; the same lower bound holds on the
retained middle pieces by continuity from
\Cref{eq:complex-kernel-reserve}.  Cauchy's theorem therefore gives the
same integral, while every connector and retained noncritical piece is
$O(e^{-(\Rea\cJ_\cE+\delta)/h})$.  This construction is finite and uniform in
the parameters.

On the Morse curve, the integral is
\[
 e^{-\cJ_\cE(r)/h}
 \int_{-\rho}^{\rho}
 a(\mathcal T(y;r,\cE))\,\partial_y\mathcal T(y;r,\cE)
 e^{-y^2/(2h)}\dd y.
\]
Taylor expansion of the holomorphic amplitude at $y=0$, followed by Gaussian
integration, gives a complete expansion in integer powers of $h$; the
truncated Gaussian tails are exponentially small.  Multiplication by the
prefactor $h^{-2}$ in \Cref{eq:K-integral} yields
\Cref{eq:complex-K-expansion-r}.  The leading coefficient is
\Cref{eq:k0} on real parameters, hence is positive there; after shrinking
the complex neighborhoods it has no zero.  Uniformity of the expansion on a
slightly larger neighborhood and Cauchy estimates give every fixed
parameter derivative of the remainder.

Finally, on a fixed complex neighborhood of a positive real annulus there is
a unique branch
\[
 \mathfrak r(\zeta)=(\zeta_1^2+\zeta_2^2)^{1/2}
\]
which equals $|\zeta|$ for real $\zeta$.  Its derivatives are bounded there.
Composition with the expansion above proves the spatial statement.
\end{proof}

\section{The complex log-flat \texorpdfstring{Lambert-$W$}{Lambert-W} saddle}\label{app:lambert-saddle}
\begin{proof}[Proof of \Cref{thm:log-flat-complex-saddle}]
Choose $0<t_2<t_0$ so that $\chi_a=1$ on $[0,t_2]$.  The integral over
$[t_2,t_0]$ is $O(e^{-c_K/h})$, uniformly for $c\in K$, and is negligible
relative to the asserted leading term.  In the remaining integral put
$t=t_*e^{-y}$ and $y_2=\log(t_*/t_2)$.  Up to the factor $t_*^{k+1}$, the
integral is
\begin{equation}\label{eq:y-integral}
 \int_{y_2}^{\infty}e^{-f(y)}\dd y,
 \qquad
 f(y)=\beta y^2+\frac{ct_*}{h}e^{-y}+(k+1)y.
\end{equation}
Let $a_k=(k+1)/(2\beta)$, $w=w_k(c,h)$, and
\[
 y_c=w-a_k,
 \qquad x_c=\Rea y_c,
 \qquad v_c=\Ima y_c=\Ima w.
\]
The critical-point equation and its first two consequences are
\begin{align}
 f'(y_c)&=0,\label{eq:yc}\\
 f(y_c)&=\beta(w^2+2w)-\frac{(k+1)^2}{4\beta},
 \label{eq:fc}\\
 f''(y_c)&=2\beta(1+w).
 \label{eq:f2c}
\end{align}
The argument of the Lambert function remains in a fixed closed subsector of
the open right half-plane.  Hence, on the principal branch,
\begin{equation}\label{eq:w-sector-properties}
 \Rea w\longrightarrow\infty,
 \qquad |\Ima w|\le C_K,
 \qquad |\arg w|\le C_K/\Rea w.
\end{equation}
For small $h$ the identity $we^w=z$ may be read without a $2\pi$ ambiguity
and gives
\begin{equation}\label{eq:arg-relation}
 \arg c=v_c+\arg w.
\end{equation}
Since $K$ is compactly contained in the right half-plane and
$\arg w=o(1)$, there is $\epsilon_K>0$ such that
$|v_c|\le\pi/2-\epsilon_K$ for all small $h$.  Under the inverse change
$t=t_*e^{-y}$, every contour below therefore remains in the fixed sector
$|\arg t|\le\pi/2-\epsilon_K$; in particular it never meets the branch
cut of the principal logarithm used in the cusp profile.

We now give the contour deformation.  Since the integrand in
\Cref{eq:y-integral} is entire, deform the real ray to the vertical segment
$y_2+i[0,v_c]$ followed by the horizontal ray
$[y_2,\infty)+iv_c$.  The closing segment at $\Rea y=R$ tends to zero as
$R\to\infty$ because of the term $\beta y^2$.  Along the finite vertical
connector, \Cref{eq:arg-relation} shows that
$\arg(ce^{-i\eta})$ stays between $\arg c$ and $\arg w$; these angles lie in
a fixed compact subinterval of $(-\pi/2,\pi/2)$.  Therefore the connector is
$O(e^{-c_K/h})$.

On the horizontal ray, write $y=x+iv_c$.  The critical-point identity gives
\begin{equation}\label{eq:horizontal-exponential}
 \frac{ct_*}{h}e^{-x-iv_c}
 =2\beta w e^{x_c-x},
\end{equation}
so its real part is positive.  Consequently
\[
 \frac{\dd^2}{\dd x^2}\Rea f(x+iv_c)
 =2\beta+2\beta\Rea w\,e^{x_c-x}>0.
\]
Thus $\Rea f$ is strictly convex on this ray and has its unique global
minimum at $x=x_c$.  In particular, no critical point from another branch of
$W$ lies on the deformed contour.

Next set $q=x-x_c$.  There is an exact identity
\begin{equation}\label{eq:saddle-difference-exact}
 f(y_c+q)-f(y_c)
 =\beta q^2+2\beta w\bigl(e^{-q}-1+q\bigr).
\end{equation}
For real $q$, the function $e^{-q}-1+q$ is nonnegative and vanishes only at
zero.  The strict convexity above therefore gives uniform exponential bounds
outside a fixed neighborhood of $q=0$.  Since $x_c\to\infty$, the lower
endpoint $y_2-x_c$ tends to $-\infty$; extending the local horizontal
integral to the full real line costs $O(e^{-c|w|})$ relative to its Gaussian
size.  Inside a fixed neighborhood of zero,
\begin{equation}\label{eq:saddle-local-expansion}
 f(y_c+q)-f(y_c)
 =\beta(1+w)q^2-\frac{\beta w}{3}q^3+O(|w|q^4).
\end{equation}
Since $\Rea(1+w)\ge c|1+w|$, the local Gaussian is uniformly integrable.
Expanding the cubic and quartic perturbations through second order, the term
linear in $q^3$ vanishes on the full symmetric Gaussian contour.  Standard
Gaussian moments give
\[
 |w|\frac{\int_{\R}|q|^4e^{-\beta\Rea(1+w)q^2}\dd q}
              {\int_{\R}e^{-\beta\Rea(1+w)q^2}\dd q}
 +|w|^2\frac{\int_{\R}|q|^6e^{-\beta\Rea(1+w)q^2}\dd q}
                {\int_{\R}e^{-\beta\Rea(1+w)q^2}\dd q}
 =O(|w|^{-1}).
\]
The omitted local tails are exponentially small in $|w|$.  Hence
\[
 \int_{y_2+iv_c}^{\infty+iv_c}e^{-f(y)}\dd y
 =e^{-f(y_c)}\sqrt{\frac{\pi}{\beta(1+w)}}
  \left(1+O(|w|^{-1})\right),
\]
with the square-root branch determined by $\Rea(1+w)>0$.  Together with
\Cref{eq:fc}, this proves \Cref{eq:Ik-asymptotic}.  The same convexity and
Gaussian estimates, without cancellation, show that the absolute integral
on the deformed contour is $O(|\mathcal S_k|)$.

All constructions remain uniform on a slightly larger compact subset of the
right half-plane.  To make the differentiated statement precise, fix an
integer $m\ge0$ and apply Cauchy's formula on circles of radius
$\delta/|w_k(c,h)|$.  On such a circle, the identity
$\partial_cw_k=w_k/[c(1+w_k)]$ shows that $w_k$ changes by $O(|w_k|^{-1})$,
while $\mathcal S_k$ changes by at most a fixed multiplicative factor.
Cauchy's estimate applied to the holomorphic remainder therefore gives
\[
 \partial_c^m\bigl(\mathcal I_k-\mathcal S_k\bigr)
 =O\bigl(|w_k|^{m-1}|\mathcal S_k|\bigr),
\]
with the convention that the right side is
$O(|w_k|^{-1}|\mathcal S_k|)$ for $m=0$.  Direct differentiation of
\Cref{eq:Qk-def} gives
$\partial_c^m\mathcal S_k=O(|w_k|^m|\mathcal S_k|)$.
This proves the asserted differentiated expansion.
\end{proof}

\begin{proof}[Proof of \Cref{lem:active-normal-tail}]
By \Cref{eq:Qk-lower-rough}, for $t\ge t_h^{\rm act}$ and
$M_0>4\beta/A_0$, the real phase
$P_A(t)=\beta\ell(t)^2+At/h$ is increasing.  At $t=t_h^{\rm act}$,
\[
 P_A(t_h^{\rm act})
 \ge\beta\ell_h^2-2\beta\ell_h\log\ell_h
 +\bigl(A_0M_0-C\log M_0-C\bigr)\ell_h.
\]
The coefficient in the last parentheses tends to $+\infty$ with $M_0$.
Choose it larger than $C_K+N+|k|+3$.  Monotonicity then bounds the integral in
\Cref{eq:active-normal-tail} by the right side, using
\Cref{eq:Qk-lower-rough}; the region where $t$ is bounded below by a fixed
constant is even $O(e^{-c/h})$.

The saddle satisfies
$t_c=2\beta h w_k/c=O_K(h\ell_h)$, so the same choice of $M_0$ places it in
$|t|<t_h^{\rm act}/2$.  Truncate the rectangular deformation in the $y$-plane at
$\Rea y=\log(t_*/t_h^{\rm act})$.  Its horizontal part maps to a ray contained in
$|t|\le t_h^{\rm act}$.  Compactness of $K$ and the angular margin in the proof of
\Cref{thm:log-flat-complex-saddle} give $a_K>0$ such that the rotated
linear coefficient on the whole vertical connector is at least $a_K$.  At
the endpoint $t_h^{\rm act}$, split
\[
 \frac{a_Kt_h^{\rm act}}{h}
 =\frac{a_Kt_h^{\rm act}}{2h}+\frac{a_Kt_h^{\rm act}}{2h}.
\]
The first half contributes the exact factor $h^{a_KM_0/2}$.  The remaining
endpoint value is bounded by the saddle exponential for the auxiliary
positive real slope $a_K/2$.  By the compact-uniform slope comparison, its
Gaussian main term differs from $|\mathcal S_k(c,h)|$ by at most a fixed
polynomial factor in $h^{-1}$.  Increasing $M_0$ absorbs that polynomial loss
and yields the connector bound $h^N|\mathcal S_k(c,h)|$ uniformly for
$c\in K$.
\end{proof}

\section{Formalization}

We provide in \cite{buck26formal} a formalization of
Theorem~\ref{thm:main} in Lean~4.30.0, using Mathlib, up to two known results that have been admitted in the formalization:
\begin{itemize}
    \item Essential self-adjointness of the magnetic Schr\"odinger
    operator on $C_c^\infty(\mathbb R^2)$ for smooth bounded real
    potentials (see for instance \cite[Theorem~5.2]{Shubin01}). This includes the variational formulation of the eigenvalues and the identification of its eigenspaces
    with smooth $L^2$ solutions  (see \cite[Section~5]{Shubin01} and \cite[Theorem~14.2]{Dyatlov22}).


    \item 
    The radial-core spectral input from
Lemma~\ref{lem:radial-spectral}, specifically the upper-bound
consequence of \eqref{eq:base-energy}, the gap
\eqref{eq:base-gap}, and the existence of a normalized
positive radial ground state.
    These follow from Theorem~1.1 (1--2) and Proposition~2.1
    in \cite{HK24},
    after conversion to our normalization.

    %
\end{itemize}
The precise statements of these inputs, their references, and
natural-language proofs of the versions used in Lean and the  are included
in the accompanying documentation of \cite{buck26formal}. We deliberately leave their
formalization outside the scope of this project: not only they would require a large amount of formalization, but the first one would be better developed as reusable contributions to Mathlib than as
separate implementations tailored to the present application and we chose not to develop ad-hoc formalizations that would risk competing with future library developments. 

The formal proof shows Theorem \ref{thm:main} by exhibiting a given admissible instance of parameter $b$, and potential $v$ obtained from the construction in this
paper such that conditions $(i)$--$(iii)$ of Theorem \ref{thm:main} hold. The chosen scalar parameters are
\[
 r_0=1,\qquad R=16,\qquad b=1,\qquad
 \varepsilon=\frac1{16},\qquad
 a_q=\beta=t_*=s_0=1,\qquad t_0=\frac1{100}.
\]
The cutoffs are fixed using Mathlib's \texttt{ContDiffBump}
construction on $\mathbb R$. Writing $\eta_{r_{\mathrm{in}},
r_{\mathrm{out}}}$ for its bump centred at zero with the indicated
inner and outer radii, we take
\[
 \chi_a=\eta_{1/400,\,1/200},
 \qquad
 \chi_b=\eta_{1/2,\,3/4}.
\]

This formal proof was obtained between September and October 2026 by generating a detailed blueprint of the proof with GPT-6-Astra Pro, that was later autoformalized with Codex using the model GPT-6-Astra Ultra. This autoformalization was done under the supervision and guidance of the authors with many iteration between the authors and the model and took a cumulative total running time of approximately 24h over the course of two weeks. 


\begingroup
\small
\begin{thebibliography}{99}
\setlength{\itemsep}{1pt}
\setlength{\parskip}{0pt}

\bibitem{buck26formal}
K.~Buck, C.~L. Fefferman, J.~G\'omez-Serrano, A.~Hayat, J.~Shapiro, M.~I. Weinstein.
\newblock Infinite Tunneling Formalisation.
\newblock \href{https://github.com/ahayat16/Infinite_Tunneling}{\texttt{github.com/ahayat16/Infinite\_Tunneling}} (2026)


\bibitem{CFFH13}
H. D. Cornean, S. Fournais, R. L. Frank, and B. Helffer,
\emph{Sharp trace asymptotics for a class of
$2D$-magnetic operators},
Ann. Inst. Fourier 63 (2013), no. 6, 2457--2513.
\url{https://doi.org/10.5802/aif.2835}.

\bibitem{Dyatlov22}
S. Dyatlov,
\emph{Lecture notes for 18.155: distributions, elliptic
regularity, and applications to PDEs},
Massachusetts Institute of Technology,
version of December 10, 2022.
\url{https://math.mit.edu/~dyatlov/18.155/155-notes.pdf}.



\bibitem{EM25}
P. Exner and L. Morin,
\emph{The flea on the magnetic elephant},
arXiv:2505.12846 (2025).


\bibitem{FLW18}
C. L. Fefferman, J. P. Lee-Thorp, and M. I. Weinstein,
\emph{Honeycomb Schr\"odinger operators in the strong binding regime},
Comm. Pure Appl. Math. 71 (2018), no. 6, 1178--1270.


\bibitem{FSW22}
C. L. Fefferman, J. Shapiro, and M. I. Weinstein,
\emph{Lower bound on quantum tunneling for strong magnetic fields},
SIAM J. Math. Anal. 54 (2022), no. 1, 1105--1130.

\bibitem{FSW-PNAS}
C. L. Fefferman, J. Shapiro, and M. I. Weinstein,
\emph{Quantum tunneling and its absence in deep wells and strong magnetic fields},
Proc. Natl. Acad. Sci. USA 122 (2025), no. 8, e2420062122.

\bibitem{FSW-excited}
C. L. Fefferman, J. Shapiro, and M. I. Weinstein,
\emph{Lower bounds on quantum tunneling for excited states},
arXiv:2506.12650 (2025).

\bibitem{FSW-absence}
C. L. Fefferman, J. Shapiro, and M. I. Weinstein,
\emph{Magnetic double-wells: absence of tunneling},
arXiv:2509.02857 (2025).

\bibitem{FSW-lower}
C. L. Fefferman, J. Shapiro, and M. I. Weinstein,
\emph{Magnetic double-wells: lower bounds on tunneling},
arXiv:2511.13470v2 (2026).

\bibitem{HK24}
B. Helffer and A. Kachmar,
\emph{Quantum tunneling in deep potential wells and strong magnetic field revisited},
Pure Appl. Anal. 6 (2024), no. 2, 319--352.

\bibitem{HS84}
B. Helffer and J. Sj\"ostrand,
\emph{Multiple wells in the semi-classical limit I},
Comm. Partial Differential Equations 9 (1984), no. 4, 337--408.

\bibitem{HS87}
B. Helffer and J. Sj\"ostrand,
\emph{Effet tunnel pour l'\'equation de Schr\"odinger avec champ magn\'etique},
Ann. Scuola Norm. Sup. Pisa Cl. Sci. (4) 14 (1987), no. 4, 625--657.

\bibitem{Helffer87}
B. Helffer,
\emph{Op\'erateurs de Schr\"odinger avec champ magn\'etique},
S\'eminaire \'Equations aux d\'eriv\'ees partielles
(Polytechnique) (1986--1987), exp. no. 10, 1--14.

\bibitem{Helffer88}
B. Helffer,
\emph{Semi-classical analysis for the Schr\"odinger operator
and applications},
Lecture Notes in Mathematics 1336,
Springer-Verlag, Berlin, 1988.

\bibitem{HunterPDE}
J. K. Hunter,
\emph{Notes on Partial Differential Equations},
lecture notes, University of California, Davis,
Chapters 3--4.
\url{https://www.math.ucdavis.edu/~hunter/pdes/pdes.html}.

\bibitem{Mat94}
H. Matsumoto,
\emph{Semiclassical asymptotics of eigenvalue distributions for Schr\"odinger operators with magnetic fields},
Comm. Partial Differential Equations 19 (1994), no. 5--6, 719--759.

\bibitem{Mor24}
L. Morin,
\emph{Tunneling effect between radial electric wells in a homogeneous magnetic field},
Lett. Math. Phys. 114 (2024), no. 1, art. 29.
\bibitem{Shubin01}
M. Shubin,
\emph{Essential self-adjointness for semi-bounded magnetic
Schr\"odinger operators on non-compact manifolds},
arXiv:math/0007019v2 (2001).
\url{https://arxiv.org/abs/math/0007019v2}.


\bibitem{Simon84}
B. Simon,
Ann. of Math. (2) 120 (1984), no. 1, 89--118.

\bibitem{SW22}
J. Shapiro and M. I. Weinstein,
\emph{Tight-binding reduction and topological equivalence in strong magnetic fields},
Adv. Math. 403 (2022), 108343.
\bibitem{FKSW-large-separation}
C. L. Fefferman, O. Kwon, J. Shapiro, and M. I. Weinstein,
\emph{Lower bounds for magnetic double-well tunneling at asymptotically large well separation},
unpublished manuscript.


\end{thebibliography}
\endgroup
\end{document}